\documentclass[a4paper]{article}

\usepackage[utf8]{inputenc}
\usepackage[T1]{fontenc}       

\usepackage[english]{babel}

\usepackage[bottom=3cm, top=3cm, left=3cm, right=3cm]{geometry}

\usepackage{setspace}

\usepackage{mathtools}
\usepackage{amssymb}
\usepackage{amsthm}
\usepackage{thmtools}
\usepackage{stmaryrd} 
\usepackage{dsfont}       
\usepackage{esvect}

\usepackage{graphicx}
\usepackage{pdflscape}
\usepackage{subcaption}
\usepackage{floatrow}
\usepackage{sidecap}

\usepackage{soul}      
\usepackage{ulem}     
\usepackage{upgreek}
\usepackage[dvipsnames]{xcolor}
\usepackage{authblk}
\usepackage{appendix}
\usepackage{lmodern}
\usepackage[section]{placeins}

\usepackage{csquotes}
\usepackage[unicode]{hyperref}
\usepackage{xurl} 
\usepackage[noabbrev,capitalize,nameinlink]{cleveref}

\usepackage[
    backend=biber,
    uniquename=false,
    bibencoding=utf8,
    sorting=nyt,
    doi=false, isbn=false, url=false,
    style=alphabetic,
    maxcitenames=2,
    uniquelist=false,
    maxbibnames=123,
    backref=false, 
    hyperref=true
]{biblatex}
\newcommand{\R}{\mathds{R}}
\newcommand{\N}{\mathds{N}}

\newcommand{\C}{C}

\newcommand{\Diff}{\operatorname{Diff}}
\newcommand{\id}{\operatorname{id}}

\hypersetup{
    colorlinks=true,
    linktoc=all,
    linkcolor=blue,   
    citecolor=red, 
    urlcolor=red,
    bookmarksnumbered=true,
    unicode=true,
        pagebackref=false,
    pdfpagelabels,
    plainpages=false
}

\theoremstyle{plain}
\newtheorem{theorem}{Theorem}[section]
\newtheorem{lemma}[theorem]{Lemma}
\newtheorem{proposition}[theorem]{Proposition}
\newtheorem{corollary}[theorem]{Corollary}

\theoremstyle{definition}
\newtheorem{definition}[theorem]{Definition}
\newtheorem{example}[theorem]{Example}

\theoremstyle{remark}
\newtheorem{remark}[theorem]{Remark}   
\newtheorem*{remark*}{Remark}          

\newcommand{\app}[4]{\begin{array}{ccl}
   #1 & \longrightarrow & #2 \\
   #3 & \longmapsto & #4 \\
\end{array}}

\newcommand{\eqdef}{\vcentcolon=}
\DeclareMathOperator*{\argmin}{arg\,min}
\sidecaptionvpos{figure}{c}

\addto\captionsfrench{}
\title{A Varifold-Based Score for Decoupling Deformations in Shape Analysis}

\author[1,2]{Rayane Mouhli\thanks{\texttt{rayane.mouhli@math.cnrs.fr}}}
\author[2]{Barbara Gris\thanks{\texttt{barbara.gris@sorbonne-universite.fr}}}
\author[1]{Irène Kaltenmark\thanks{\texttt{irene.kaltenmark@u-paris.fr}}}

\affil[1]{Université Paris Cité, MAP5}
\affil[2]{Sorbonne Université, LJLL}
\date{}

\begin{document}
\maketitle

\begin{abstract}
In computational anatomy, analyzing morphological variability across shape populations often requires multi-component deformation models that combine structured motions and unconstrained diffeomorphisms. However, a major challenge arises during the registration process, as high-dimensional deformations tend to absorb lower-dimensional components, altering the true geometric variability and preventing accurate statistical analysis. To address this issue, we introduce a novel coupling score designed to decouple distinct deformation modes during registration. This score is defined using first variation of varifolds which is a varifold representation of the infinitesimal action of vector fields on shapes. The proposed score quantifies the extent to which the action of a given vector field on a shape can be replicated by another subspace of vector fields. We provide a theoretical analysis of this coupling score, illustrating its behavior on finite-dimensional spaces. Finally, we integrate this score as a penalization term in registration problems. Numerical experiments illustrate its efficiency in various use cases such as enforcing or preventing specific motions, and iteratively correcting complex matching scenarios by enforcing structured directional priors.
\end{abstract}

\tableofcontents
\section{Introduction} 

In computational anatomy, a central objective is to study the geometric variability within a family of shapes. 
More than pure statistical descriptions, explaining the nature of morphological variability is key to biological applications. In particular, mapping shapes and characterizing their differences provide necessary structural insights. A common approach involves shape registration via large deformations, which relies on constructing diffeomorphisms to map these shapes onto one another \cite{LDDMM,Grenanderetal2007,Thirion1998,Vercauterenetal2009}. 

Beyond pairwise registration, the deformation itself becomes an object of study, which is fundamental for longitudinal analyses and for the modeling of population-wide changes, as in the contexts of organ growth, atrophy, or aging \cite{kaltenmark_thesis}. In particular, controlling the deformation throughout the registration process ensures that the underlying anatomical structures are inherently respected. This circumvents the need for the explicit segmentations typically required to guarantee meaningful biological correspondences. To generate meaningful and interpretable shape deformations, a natural strategy is to construct the total deformation as a combination of components (isometries, local scalings, unstructured diffeomorphisms etc.) using semidirect products of deformation groups \cite{Bruveris_2010,Bruveris_2012,mouhli2025}, modular deformations \cite{gris_module} or multiscale approaches \cite{Risser2010,vialard:hal-00663393,debroux2023multiscale,MODIN20191009,sommer2011,sommer2013sparse,BME,pierron2024}. These decompositions aim to enhance the biological fidelity of the matching and to isolate a specific component for each deformation pattern, allowing one to individually track and analyze the underlying structural changes. While recent deep learning models \cite{tian2024unigradicon,shenNetworksJointAffine2019} offer effective registration, current implementations prevent the separation and the interpretation of the various components of the resulting transformations.

However, multi-component deformation models raise the challenge of accurately weighting, balancing, and measuring the individual contribution of each deformation type. Those lying in high-dimensional spaces are extremely flexible and tend to compensate for lower-dimensional deformations. For instance, in a two-mode deformation model combining rotations and unstructured diffeomorphisms, the last component might perform rotations during optimization, preventing the model from properly separating the roles of the different deformations. To achieve a meaningful decomposition, the rotation performed by the total deformation should be captured entirely by the rotation component, leaving only the remaining deformations to the diffeomorphism. This enables an accurate measurement of the true global rotation performed by the model. Another motivation for considering this multi-component framework is that unconstrained diffeomorphisms often struggle to reproduce highly structured deformations due to the shape's reparameterization invariances, separating unstructured and structured deformations becomes essential to accurately capture these distinct behaviors.
A standard approach to decouple different types of deformations employs a sequential two-step registration process, where shapes are first aligned globally at a coarse scale before refining the alignment at a finer scale. Another baseline strategy models the transformation as a diffeomorphism generated by a sum of vector fields, where each field represents a distinct deformation mode. During optimization, the relative influence of these modes is controlled by adjusting their respective penalty weights. However, this method requires manual hyperparameter tuning, and poorly balanced weights can lead to inconsistent deformations. To improve upon these baselines and optimize all scales simultaneously, Sommer et al. \cite{sommer2011multiscale} introduced a multiscale kernel bundle framework to decouple deformations within a single registration step.More recently, \cite{mouhli2025} relied on geometric mechanics and reduction theory \cite{MARSDEN1974121,marsden1994introduction} to achieve this decoupling in the specific case of deformations represented by a semidirect product of a finite-dimensional Lie group and a groups of diffeomorphisms.

In this paper, we introduce a novel coupling score that enforces the decoupling of distinct deformation modes during the matching process. We consider shapes as elements of a shape space as defined in \cite{arguillere2015generalsettingshapeanalysis}, which is a manifold acted upon by a group of diffeomorphisms under specific regularity assumptions. All deformations considered are modeled as diffeomorphisms generated by integrating a sum of time-varying vector fields, where each individual field represents a distinct deformation mode. Deforming a shape via such a diffeomorphism induces an infinitesimal action of the associated vector field on the shape. The definition of the coupling score requires a metric capable of comparing the respective actions of the deformations on a given shape, rather than merely comparing the deformation themselves. Since we are interested in the geometric deformation of the shape, we seek a metric that is invariant under reparameterization. Directly comparing the infinitesimal actions of vector fields on shapes lacks this invariance. Furthermore, while intrinsic metrics \cite{SRV,bauer2022intrinsic, michor2007overview} successfully achieve reparameterization invariance, they remain restricted to curves and surfaces and cannot easily generalize to other shapes. To overcome these limitations, we define a metric using varifolds \cite{Almgren,Allard,Buet_2017}, a concept originating from geometric measure theory and later adapted for computational anatomy by Charon and Trouvé \cite{Charon_2013}. A varifold provides a representation of shapes, defined as submanifolds embedded in $\mathbb{R}^d$, that is reparameterization invariant. Within the RKHS framework, varifolds enable to define distances between shapes \cite{Kaltenmark_2017_CVPR, Charon_2013} or more complex structures like trees \cite{maignant2026velotree}. The use of varifolds has been extended to numerous applications, including partial matching \cite{antonsanti2021partial,sukurdeep2022new,hsieh2021} and classification/regression tasks \cite{hartman2025svarm}.
 The first variation of varifolds \cite{Allard} provides a representation of the infinitesimal deformation of the underlying shape induced by a vector field. Conceptually, it acts as the analog from the varifold viewpoint of the infinitesimal action of a vector field on shapes, as defined by Arguillère \cite{arguillere2015generalsettingshapeanalysis}. In the context of a deformation model generated by two spaces of vector fields $V$ and $W$, for a fixed shape $q$, our coupling score computes the projection of the first variation of $q$ induced by a vector field $v \in V$ onto the space of first variations of $q$ induced by $W$. In other words, it measures the extent to which the action of $v$ on a shape can be replicated by an element of $W$. The primary advantage of using the first variation rather than comparing the infinitesimal action of vector fields on shapes, is that it ignores tangential deformations. Consequently, it strictly captures the deformation of the geometric support without being affected by local reparameterizations. 
 This article focuses specifically on the example of curves to ensure simpler and more interpretable expressions, even though the results generalize to larger families of shapes, such as surfaces. From a practical viewpoint, the coupling score is designed for registration problems that require the decoupling of multiple deformation modes, or where one seeks to enforce or prevent specific behaviors during the matching. This score can be integrated as a penalization term in the optimization functional of any registration problem that optimizes over the space of vector fields.
 
\paragraph{Main contributions} The main contributions of this article are summarized as follows:

\begingroup
\renewcommand\labelenumi{(\theenumi)}
\begin{enumerate}
    \item \textbf{First Variations of Varifolds}: We adapt the first variation of varifolds framework to represent the infinitesimal action of a vector field on a curve for shape analysis. Then, we consider a RKHS continuously embedded in the space of varifolds, whose metric allows to define a tractable dissimilarity measure on this space.
    \item \textbf{Introduction of a Coupling Score}: We define a coupling score that quantifies the extent to which the action of a vector field $v \in V$ on a shape can be replicated by a vector field from a subspace $W$. Furthermore, we provide a theoretical study of this score, illustrating its behavior with concrete examples of actions of finite-dimensional vector field spaces, such as global translations and one-dimensional scalings on curves.
    \item \textbf{Decoupling for Shape Registration and Numerical Experiments}: We integrate this coupling score into the LDDMM registration framework and show different use cases of this score through numerical experiments, such as decoupling deformations,  preventing/enforcing a certain type of deformations or iteratively correcting complex matching scenarios by enforcing structured directional priors.
\end{enumerate}
\endgroup

\paragraph{Structure of the paper} \cref{sec:varifold} reviews the fundamental definitions and properties of varifolds and first variations of varifolds, along with their associated RKHS metrics. \cref{sec:correlation} introduces the coupling score, deriving its general expression and illustrating its behavior when applied to finite-dimensional spaces such as translations and scalings. \cref{sec:stat_to_dyn} adapts the coupling score setting, defined statically in the previous section, into a dynamic framework suitable for shape registration problems. Finally, \cref{sec:experiments} presents numerical experiments demonstrating the versatility of the coupling score in various scenarios such as decoupling deformations, explicitly preventing or enforcing specific deformation modes, and iteratively correcting complex matching errors in population studies.
\section*{Notations}

\begin{itemize}

    \item Unless otherwise specified, the norm $\Vert \cdot \Vert$ refers to the norm in $\R^d$.
    
    \item Given a vector space $H$, we denote $H'$ its dual space, that is the space of continuous linear functionals $H \rightarrow \R$.

    \item  $C^k_0(\R^d,\R^d)$ endowed with the norm $\Vert v \Vert_{C_0^k} = \sum_{i=0}^k \Vert d^i v \Vert_\infty$ is the Banach space of $C^k(\R^d,\R^d)$ functions that vanish at infinity. 
    
    \item $\Diff_{C_0^k}(\R^d)$ is the space of $k$-times continuously differentiable functions that tend towards identity at infinity.
    \item Let $v : \R^d \to \R^d$ be a vector field and let $q_{f^1}, q_{f^2}$ be two landmarks in $\R^d$ forming an edge $f$. Let us denote:
    \begin{itemize}
        \item $\ell_f = \Vert q_{f^2} - q_{f^1} \Vert$,
        \item $c_f = \frac{q_{f^1} + q_{f^2}}{2}$,
        \item $\vec{t}_f = \frac{q_{f^2} - q_{f^1}}{\ell_f}$,
        \item $v_f = \frac{v(q_{f^1}) + v(q_{f^2})}{2}$,
        \item $(\delta \ell_f)_v = \left\langle \frac{v(q_{f^2}) - v(q_{f^1})}{\ell_f}, \vec{t}_f \right\rangle$,
        \item $\nabla^\perp v_f = \frac{v(q_{f^2}) - v(q_{f^1})}{\ell_f} - \left\langle \frac{v(q_{f^2}) - v(q_{f^1})}{\ell_f}, \vec{t}_f \right\rangle \vec{t}_f$.
    \end{itemize}
    
    \item Let $q : I \to \R^d$ be a parameterized curve such that $\Vert q'(s) \Vert > 0$ for all $s \in I$. Let us denote:
    \begin{itemize}
        \item $\tau_q(s) = \frac{q'(s)}{\Vert q'(s) \Vert}$,
        \item $v(q(s))^\perp = v(q(s)) - \langle v(q(s)), \tau_q(s) \rangle \tau_q(s)$.
    \end{itemize}

\end{itemize}

\section{Varifolds and first variations of varifolds} \label{sec:varifold}

Varifolds were originally introduced by Almgren \cite{Almgren} and further developed by Allard \cite{Allard} to address Plateau's problem, which consists in finding surfaces of minimal area bounded by a given contour. Charon and Trouvé \cite{Charon_2013} later adapted this framework to provide an effective representation of shapes as measures in the context of computational anatomy. Notably, they introduced a kernel metric on the space of varifolds, yielding a dissimilarity measure between shapes. Then, they generalized the notion of first variation of varifolds, first introduced by Allard \cite{Allard}, to study the variation of varifolds with respect to their underlying shapes. In particular, it offers a varifold representation of the action of a vector field on a shape. In continuation of this work, we will consider first variations of varifold to define a dissimilarity measure between vector fields.

\subsection{Varifolds}

This section provides a brief overview of the standard definitions and classical properties related to varifolds and Reproducing Kernel Hilbert Space (RKHS) theory \cite{aronszajn50reproducing}. The definition of varifolds and further properties, some of which are recalled here, are presented in \cite{Charon_2013,Kaltenmark_2017_CVPR}.

\subsubsection{Definition and first properties}
The representation of oriented submanifolds as distributions was first formalized by the framework of currents \cite{glaunes_distrib,glaunes_vaillant}. In this setting, an oriented $k$-dimensional submanifold $X$ embedded in $\R^d$ is associated with the current $C_X : \omega \mapsto \int_X \omega$, where $\omega$ belongs to the space of $k$-dimensional differential form. Varifolds extend this framework by enabling representation of non-oriented submanifolds as a measure over unoriented tangent spaces \cite{Charon_2013}. We first give the general definition of varifolds.

\begin{definition}[Varifolds] Let $G_k(\R^d)$ denotes the Grassmannian of dimension $k$, that is the set of $k$-dimensional linear subspaces of $\R^d$.
A $k$-dimensional varifold on $\R^d$ is a Borel finite measure on the product space $\R^d \times G_k(\R^d)$ or equivalently by Riesz theorem, a continuous linear form on $C^0_0(\R^d \times G_k(\R^d),\R)$. 
\end{definition} 
Given a non-oriented submanifold $X$ of dimension $k$, its tangent space at $x \in X$, denoted $T_xX$ belongs to $G_k(\R^d)$. Therefore, one can associate to $X$ a varifold $\mu_X$ defined by, for $\omega \in C^0_0(\R^d \times G_k(\R^d),\R)$
    \begin{equation} \label{eq:general_varifold}
        \mu_X(\omega) = \int_X \omega(x, T_xX)d \mathcal{H}^k(x)
    \end{equation}
where $\mathcal{H}^k$ denotes the $k$-dimensional Hausdorff measure on $\R^d$.

The framework introduced in this paper holds for any submanifolds embedded in $\R^d$ such as open and closed curves in $\R^2$ and $\R^3$, surfaces or curve bundles. However, we will restrict our study to oriented curves in $\R^d$ as this allows us to identify the oriented tangent space with the unit sphere $\mathbb{S}^{d-1}$. Consequently, the tangent space $T_x X$ can be represented by a unit vector, denoted $\vec{t}(x) \in \mathbb{S}^{d-1}$. 

\begin{definition}[Curves]\label{def:curves}  Let  $I=[0,1]$ or $I=\mathbb{S}^{1}$. We define the space of curves as the space of $C^1$-immersions  $Q =\{ q \in C^{1}(I,\mathbb{R}^d) \mid \Vert q'(s) \Vert \ne 0, \forall s \in I \}$. A discrete representation of a curve $q \in Q$ is defined by a sample of $N$ distinct points along the curve $(q_i)_{i=1,...,N}$ linked by segments in $F_q \subset \{1,...,N\}^2$, forming a piecewise linear approximation.
\end{definition}

In what follows, we introduce specific notations for the various varifolds representations presented in \cref{exam:smooth_curve} and \cref{exam:discretized_curve}. For a curve $q \in Q$, we denote its associated varifold by $\tilde{\mu}_q$. Given a discrete representation $((q_i)_i,F_q)$ of $q$, its associated varifold is denoted $\hat{\mu}_q$ and can be approximated by a sum of Dirac masses $\mu_q$, which is also a varifold as detailed below.

\begin{example}[Varifold representation of a curve]\label{exam:smooth_curve}

Using \cref{eq:general_varifold}, the varifold associated with a curve $q \in Q$ is given for any $\omega \in \C_0^0(\R^d \times \mathbb{S}^{d-1},\R)$ by
\begin{equation} \label{eq:smooth_varifold}
    \tilde{\mu}_q(\omega)=\int_I \omega\left(q(s),\tau_{q}(s)\right) \Vert q'(s) \Vert ds 
\end{equation}
with $\tau_{q}(s) = \dfrac{q'(s)}{\Vert q'(s)\Vert}$. Note that by change of variable, this expression is independent of any positive reparameterization of $q$. This formula remains valid for piecewise curves.
\end{example}

\begin{example}[Varifold representation of a discretized curve]\label{exam:discretized_curve}
   For a curve $q \in Q$, the varifold associated with a discretized representation $\left((q_i)_i,F_q \right)$ is defined by, for $\omega \in C_0^0(\R^d \times \mathbb{S}^{d-1},\R)$
     \begin{equation}
         \hat{\mu}_q(\omega)= \sum_{f \in F_q} \int_{f} \omega(x, \vec{t}_f) dx
     \end{equation}
     where $\vec{t}_f=\dfrac{q_{f^2}-q_{f^1}}{\Vert q_{f^2}-q_{f^1} \Vert}$ denotes a unit tangent vector to this edge. Moreover, by denoting $c_f=\dfrac{q_{f^1} + q_{f^2}}{2}$ and $\ell_f = \Vert q_{f^2} - q_{f^1} \Vert$ the center and the length of the edge $f$, the finite weighted sum of Dirac masses
    \begin{equation} \label{eq:varifold_discrete}
        \mu_q(\omega) = \sum_{f \in F_q} \ell_f \updelta_{(c_f,\vec{t}_f)}(\omega) = \sum_{f \in F_q} \ell_f \omega(c_f,\vec{t}_f)
    \end{equation} is a varifold which approximates $\hat{\mu}_q$, in the sense defined by \cite{Kaltenmark_2017_CVPR}, with an approximation error bounded by the maximal edge length. For computational convenience, the varifold representation of a curve that is used in practice is the approximation by a weighted sum of Diracs $\mu_q$. In what follows, unless otherwise specified, a curve $q$ is assumed to be represented by the sum of Dirac $\mu_q$ with an implicitly chosen discretization $((q_i)_i,F_q)$.
\end{example}

In the context of computational anatomy, curve matching consists in the deformation of a template curve by $q$ a diffeomorphism $\varphi$ \cite{LDDMM}, defined by the action $\varphi \cdot q = \varphi \circ q$. This action naturally extends to the transport of the associated varifold given by the push-forward $\varphi_* \mu_q := \mu_{\varphi(q)} $, which defines a varifold again \cite[Proposition 3.2]{Charon_2013}.

\begin{example}[Transport of curves by diffeomorphisms] For $q \in Q$, the transport of the varifold $\tilde{\mu}_q$ by a diffeomorphism $\varphi \in \Diff(\R^d)$ is given,  for $\omega \in C_0^0(\R^d \times \mathbb{S}^{d-1},\R)$, by
\begin{equation*}
   \tilde{\mu}_{\varphi(q)}(\omega)=\int_I \omega\left(\varphi(q(s)),\dfrac{d_{q(s)} \varphi(\tau_{q(s)})}{\Vert d_{q(s)} \varphi(\tau_{q(s)}) \Vert}\right) \big\Vert d_{q(s)}\varphi(q'(s)) \big\Vert ds.
\end{equation*} 
 Similarly, given a discrete representation $((q_i)_i, F_q)$ of $q$, the transport of its representation by a sum of Diracs $\mu_q$, is
\begin{equation*}
   \mu_{\varphi(q)} (\omega) = \sum_{f \in F_q} \Vert \varphi(q_{f^2})-\varphi(q_{f^1}) \Vert ~\omega(c(\varphi(q)_f),\vec{t}(\varphi(q)_f))
\end{equation*}
where $c(\varphi(q)_f) = \dfrac{\varphi(q_{f^1}) + \varphi(q_{f^2})}{2} $ and $\vec{t}(\varphi(q)_f) = \dfrac{ \varphi(q_{f^2}) - \varphi(q_{f^1})  }{\Vert \varphi(q_{f^2}) - \varphi(q_{f^1}) \Vert}$.
\end{example}

\subsubsection{RKHS metric on varifolds }

The space of varifolds is not canonically equipped with a tractable metric. This issue can be addressed by considering the inner product associated with a Reproducing Kernel Hilbert Space (RKHS)  \cite{aronszajn50reproducing} continuously embedded in $C_0^0(\R^d \times \mathbb{S}^{d-1},\R)$, which leads to a simple and computationally practical expression for discrete curves. A RKHS $H \hookrightarrow C_0^0(\R^d \times \mathbb{S}^{d-1},\R)$ can be defined by the completion of the vector space spanned by the fundamental functions
\begin{equation*}
    k_{H}((x,\vec{t}),\cdot) : \app{\R^d \times \mathbb{S}^{d-1}}{\R}{(x',\vec{t}')}{k_{H}((x,\vec{t}),(x',\vec{t}'))} \end{equation*} \\
where $k_{H}$ is a positive kernel \cite{glaunes2005transport}. For $(x,V)\in \R^d \times \mathbb{S}^{d-1}$,  the reproducing property of the kernel states that the Dirac measure $\updelta_{(x,V)} : \omega \in H \mapsto \omega(x,V) \in \R$ is mapped to the function $ k_{H}((x,V),\cdot)$  thanks to the Riesz isomorphism $K_{H} : H' \rightarrow H$. Then, the inner product between Dirac measures is defined by, for every $x,x' \in \R^d$ and $V,V' \in \mathbb{S}^{d-1}$,
\begin{align}\label{inner_product_varifold}
    \langle \updelta_{(x,V)}, \updelta_{(x',V')} \rangle_{H'} &=  \langle K_H \updelta_{(x,V)}, K_H \updelta_{(x',V')} \rangle_{H}\\
    &= \langle k_{H}((x,V),\cdot), k_{H}((x',V'),\cdot) \rangle_{H} \notag \\
    &= k_{H}((x,V),(x',V')) \notag
\end{align}
and it can be extended to any varifolds associated with curves $\mu_{q_a}, \mu_{q_b} \in H'$ by bilinearity and continuity of the inner product \cite{Charon_2013,glaunes2005transport}. 
The definition of a RKHS on $C_0^0(\R^d \times \mathbb{S}^{d-1},\R)$ requires the construction of a kernel on the product space $\R^d \times \mathbb{S}^{d-1}$. Such a kernel can be built by considering the tensor product of a positive kernel $k_E$ on positions and a positive kernel $k_T$ on tangents
\begin{equation*}\label{eq: product varifold kernel}
    k_{H}((x,V),(x',V'))=k_E(x,x')k_T(V,V').
\end{equation*}
If $k_T$ is continuous and $k_E$ is continuous, bounded and for all $x \in \R^d$ the function $k_E(x,\cdot)$ vanishes at infinity, then the RKHS associated with $k_E \otimes k_T$ is continuously embedded into $C_0^0(\R^d \times \mathbb{S}^{d-1},\R)$. The kernel $k_T$ on $\mathbb{S}^{d-1}$  can be induced by the restriction of a positive kernel on  the space of linear mapping from $\R^d$ to $\R^d$, denoted $\mathcal{L}(\R^d)$, by identifying $V \in \mathbb{S}^{d-1}$ with the orthogonal projection on $V$ denoted $p_V \in \mathcal{L}(\R^d)$ \cite{Charon_2013}.

\begin{example}[RKHS inner product between varifolds \cite{Kaltenmark_2017_CVPR}]

    Let $H \hookrightarrow C_0^0(\R^d \times \mathbb{S}^{d-1},\R)$ be a RKHS generated by a kernel $k_{H} = k_E \otimes k_T$. For $q_a, q_b \in Q$, the inner product between the varifolds $\tilde{\mu}_{q_a}$ and $\tilde{\mu}_{q_b}$, is given by 
    \begin{equation*}
        \langle \tilde{\mu}_{q_a} , \tilde{\mu}_{q_b} \rangle_{H'} = \int_I \int_I  k_E \left(q_a(s),q_b(t)\right) k_T \left(\tau_{q_a}(s),\tau_{q_b}(t)\right) \Vert q_a'(s) \Vert \Vert q_b'(t) \Vert ds \, dt
    \end{equation*}
and between their approximation $\mu_{q_a}$ and $\mu_{q_b}$, by
\begin{equation*}
    \langle \mu_{q_a} , \mu_{q_b} \rangle_{H'} = \sum_{\substack{f \in F_{q_a} \\ g \in F_{q_b}}}\ell_{a,f} \, \ell_{b,g} \, k_E(c_{a,f},c_{b,g}) \,k_T(\vec{t}_{a,f},\vec{t'}_{b,g})
\end{equation*}.
\end{example}

The inner product on $H'$ induces only a pseudo-distance on the space of varifolds $C_0^0(\R^d \times \mathbb{S}^{d-1},\R)'$ since the dual mapping $i^* : C_0^0(\R^d \times \mathbb{S}^{d-1},\R)' \rightarrow H'$ is not necessarily injective \cite{Charon_2013}. To ensure that $i^*$ is an embedding and consequently that the pseudo-distance is a distance, we can consider kernels $k_{H}$ that satisfy the $C_0$-universal property, which is equivalent to the property of $H$ being dense in $C_0^0(\R^d \times \mathbb{S}^{d-1}, \R)$ \cite{carmeli2008vectorvaluedreproducingkernel}. As an example, the Gaussian kernel $k_{\sigma} : (x,y) \mapsto \exp\left({-\dfrac{\Vert y - x \Vert^2}{\sigma^2}}\right)$, widely used in computations, is a $C_0$-universal kernel on $\R^d$. In the case where $k_H = k_E \otimes k_T$, if $k_E$ is a $C_0$-universal kernel on $\R^d$ and $k_T$ is a Gaussian kernel on $\mathcal{L}(\R^d)$ then $k_{H}$ is a $C_0$-universal kernel. 

\begin{remark}
As noticed in \cite[Section 4.2.3]{glaunes2005transport}, the transport of varifolds by diffeomorphisms, given by $(\varphi_* \mu)(\omega)=\mu(\varphi^*\omega)$ for $ \mu \in C_0^0(\R^d \times \mathbb{S}^{d-1},\R)'$, cannot be defined for $ \mu \in H'$ since the transport of test functions $\varphi^* \omega$ might not belong to the RKHS $H$. Indeed, let $H_{\sigma} \hookrightarrow C_0^0(\R^d \times \mathbb{S}^{d-1},\R)$ be a RKHS induced by a Gaussian kernel $k_{\sigma}$ with scale $\sigma>0$ on $\R^d$ and a degenerated kernel $k_T = 1$ on $\mathbb{S}^{d-1}$.  For $\sigma' >0$, the pull-back of  $\omega = k_{\sigma}(0,\cdot)  \in H_{\sigma}$ by the rescaling diffeomorphism $\varphi : x \mapsto \dfrac{\sigma}{\sigma'} x $ is 
$\varphi^* \omega : x \mapsto (\dfrac{\sigma}{\sigma'})^d \exp({-\dfrac{\Vert x \Vert^2}{\sigma^{'2}}}) $ which belongs to $H_{\sigma'}$.
However, $\varphi^* \omega$ does not belong to $H_{\sigma}$ for $\sigma'\leq \dfrac{\sigma}{2}$. Note that this theoretical limitation only applies to general varifold $\mu$ and not to varifolds associated with curves. Indeed, for any curve $q$, the action of a diffeomorphism $\varphi$ on $\tilde{\mu}_q \in H'$ is $\tilde{\mu}_{\varphi(q)}$ which also belongs to $H'$.
\end{remark}

\begin{remark}\label{rem:deg_var}
Historically, varifolds have been introduced in computational anatomy as an extension of the  currents framework \cite{glaunes_vaillant}, which is a tool from geometric measure theory providing a representation of oriented submanifolds. One limitation of the currents presented in \cite{Charon_2013} is that a submanifold $X$ of positive mass may have arbitrary small RKHS norm. This cancellation effect can occur in shapes like sharp spins or tails due to the opposite orientation at nearby points. Even though, this phenomenon may result in some loss of information, it could be useful to represent noisy data. However, RKHS metrics defined on varifolds avoid this cancellation phenomenon by controlling the volume of a submanifold by the $H'$-norm. Indeed, for an RKHS $H$, under standard assumptions on its kernel, a submanifold $X$ with a non-zero 1-dimensional Hausdorff measure induces a non-zero norm for $\mu_X$  \cite[Theorem 4.1]{Charon_2013}.
\end{remark}

Given a varifold $\mu_q \in C_0^0(\R^d \times \mathbb{S}^{d-1},\R)'$ and a RKHS $H \hookrightarrow C_0^0(\R^d\times \mathbb{S}^{d-1},\R)$, the Riesz isomorphism maps the varifold to a scalar-valued function $K_{H} \mu_q \in H$. For a degenerated kernel on the tangents $k_T=1$, its expression is given by \begin{equation} \label{eq:varifold}
    K_{H} \mu_q  : \app{\R^d}{\R}{x}{\sum_{f \in F_q} \ell_f ~ k_E(c_f,x)}. 
\end{equation} 
where $\ell_f=\Vert q_{f^2}-q_{f^1} \Vert$ and $c_f = \dfrac{q_{ f^1} + q_{f^2}}{2}$. The mapping $K_H  \mu_q$ is a representation of the varifold $\mu_q$ as an image in $\R^d$.
This representation assigns an intensity to any point of the ambient space $x \in \R^d$ by convolving the centers of the edges $c_f$ by the kernel $k_E$. As illustrated in \cref{fig:varifold_heart}, for a Gaussian kernel $k_{\sigma}$, the scale parameter $\sigma$ controls the spatial interaction range in the ambient space. An intermediate scale ($\sigma=5$, center) captures effectively the global geometry of the curve, while a small scale ($\sigma=0.5$, left)
strongly localizes the intensity around the curve and a large scale ($\sigma=20$, right) leads to a loss of the geometry due to an excessive smoothing of the intensity.

\begin{figure}[!h]
    \centering
    \includegraphics[width=1\textwidth]{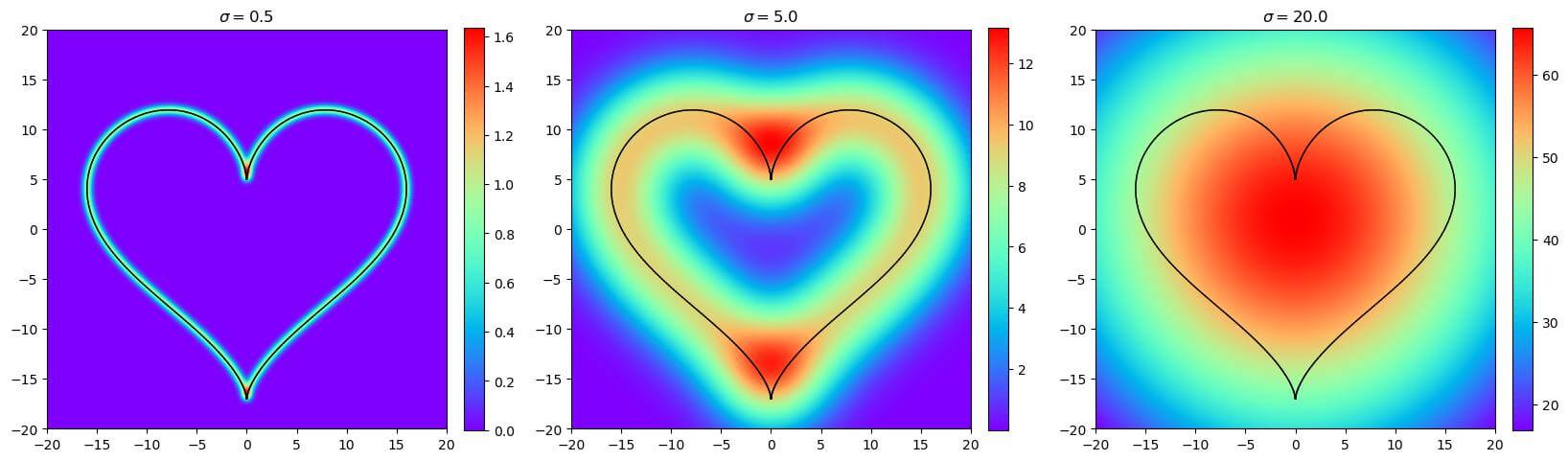}
    \caption{Varifold representation $K_H \mu_q$ (\cref{eq:varifold}) of a heart-shaped curve using a Gaussian kernel on positions with different scales $\sigma \in \{0.5, 5, 20\}$ and a degenerated kernel on tangents.}
    \label{fig:varifold_heart}
\end{figure}

\subsection{First variation of varifolds induced by a vector field}\label{sec:first_var} 

The deformation of a curve $q$ by a diffeomorphism $\varphi$ is defined by the group action $\varphi \cdot q = \varphi \circ q$, which induces an infinitesimal action of vector fields on curves, as later detailed.
In the same way as curves can be represented by varifolds, infinitesimal actions of vector fields on curves can be represented by first variations of varifolds, a notion first introduced in geometric measure theory by Allard \cite{Allard}.
We recall that we focus our study on curves embedded in $\R^d$, even though the results can be generalized to larger families of shapes such as surfaces.

\subsubsection{Definitions and first properties}\label{sec:def_prop_first_var}

Varifolds are defined as continuous linear forms on the space of test functions $C^0_0(\R^d \times \mathbb{S}^{d-1})$. To differentiate these test functions, we recall in \cref{app:diff_submanifolds} some definitions and properties about differential calculus on submanifolds of $\R^d$. In particular, for test function $\omega \in C_0^1(\R^d \times \mathbb{S}^{d-1},\R)$, we denote by $\nabla_x\omega$ and $\nabla_{\tau}\omega$ its gradients with respect to the position and tangent variables.

We remind that for $E, F$ Banach spaces and $U$ an open subset of $E$, the function $f : U \to F$ is called Fréchet differentiable at $x \in U$ if there exists a bounded linear operator $A_x \colon E \to F$ such that $ \lim_{h \to 0} \dfrac{\| f(x+h) - f(x) - A_x(h) \|_F}{\| h \|_E} = 0$. In this case, $A_x$ is called the Fréchet derivative of $f$ at the point $x$.

\begin{proposition}[Fréchet differentiability of varifolds]\label{prop:frechet_varifold}
    Let $I = [0,1]$ or $\mathbb{S}^{1}$ and $Q = \{ q \in C^1(I, \mathbb{R}^d) \mid \inf_{s \in I} \|q'(s)\| > 0 \}$ be the set of immersions. The mapping
    $$\tilde{\mu} : \app{Q}{C_0^2(\mathbb{R}^d \times \mathbb{S}^{d-1}, \mathbb{R})'}{q}{\tilde{\mu}_q := \omega \mapsto \int_I \omega\left(q(s), \dfrac{q'(s)}{\|q'(s)\|}\right) \|q'(s)\| ds}$$
    is continuously Fréchet differentiable.
\end{proposition}

\begin{proof}
    The proof is given in \cref{app:proof_frechet}
\end{proof}

By definition, varifolds belong to $C_0^0(\mathbb{R}^d \times \mathbb{S}^{d-1}, \mathbb{R})'$. However, the Fréchet differentiability of the mapping $q \mapsto \tilde{\mu}_q$ requires the test functions to be sufficiently regular. By considering test functions in $C_0^2$, we ensure that $\omega$ is differentiable and its derivatives are controlled by Lipschitz continuity. Although differentiation induces a loss of one degree of regularity, the dual inclusions $(C_0^0)' \subset (C_0^1)' \subset (C_0^2)'$ guarantee that both the varifolds and their differentials belong to $C_0^2(\mathbb{R}^d \times \mathbb{S}^{d-1}, \mathbb{R})'$. In practice, to benefit from RKHS properties as for varifolds, we consider test functions belonging to an RKHS $H \subset C_0^\infty(\mathbb{R}^d \times \mathbb{S}^{d-1}, \mathbb{R})$ generated by a smooth kernel, such as the Gaussian kernel. Consequently, the dual inclusions $(C_0^2)' \subset (C_0^\infty)' \subset H'$ ensure that this higher regularity requirement is naturally satisfied and does not restrict the model.

As detailed in the previous section, a curve $q$ can be represented by a varifold $\mu_q$, and the transport of the curve by a diffeomorphism $\varphi$, defined by $\varphi \circ q$ naturally extends to its associated varifold as $\varphi_* \mu_q = \mu_{\varphi \circ q}$. The infinitesimal action of a vector field $v$ on $q$ is therefore defined as the differential of the mapping $\varphi \mapsto \varphi \circ q$ at $\varphi = \operatorname{id}$ applied to $v$~\cite{arguillere2015generalsettingshapeanalysis}. The varifold counterpart of this infinitesimal action, referred to as the first variation of varifolds, is defined as the differential of the mapping $\varphi \mapsto \mu_{\varphi \circ q}$ at $\varphi = \operatorname{id}$ applied to $v$.

\begin{definition}[First variation of varifolds]\label{def:first_var}
    Let $v \in C_0^2(\R^d,\R^d)$ be a vector field and $q_0 \in Q$ be a curve. The first variation of the varifold $\tilde{\mu}_{q_0}$ induced by $v$ is defined as the Fréchet derivative of the mapping $q \mapsto \tilde{\mu}_q$ evaluated at $q_0$ in the direction of $v \circ q_0$. The first variation is denoted by:
    \begin{equation*}
        \delta \tilde{\mu}_{q_0}(v) := d\tilde{\mu}_{q_0}(v \circ q_0) \in C_0^2(\R^d \times \mathbb{S}^{d-1},\R)'.
    \end{equation*}
\end{definition}

The expression of the first variation is an adaptation to the particular case of curves of \cite[Theorem 4.2]{Charon_2013}.

\begin{proposition}[Expression of the first variation] \label{prop:continuous_first_var}
Let $v \in C_0^2(\R^d,\R^d)$ be a vector field and $q \in Q$ be a curve represented by the varifold $\tilde{\mu}_q$. For $\omega \in C_0^2(\R^d \times \mathbb{S}^{d-1},\R)$, the expression of the first variation is given by
\begin{align*}
    \delta \tilde{\mu}_{q}(v)(\omega) &= \int_I \left( \left\langle \nabla_x \omega ( q(s),\tau_q(s) ), v(q(s)) \right\rangle + \left\langle \nabla_{\tau} \omega ( q(s),\tau_q(s) ), (dv(q(s))\cdot \tau_q(s))^\perp \right\rangle \right) \Vert q'(s) \Vert \, ds  \\
    &+ \int_I \omega ( q(s),\tau_q(s) ) \langle dv(q(s)) \cdot \tau_q(s), \tau_q(s) \rangle \Vert q'(s) \Vert \, ds 
\end{align*}
where $\tau_q(s) = \dfrac{q'(s)}{\Vert q'(s) \Vert}$ denotes the unit tangent vector at the point $q(s)$, and for $w : \R^d \to \R^d$,  $w(q(s))^\perp = w(q(s)) - \langle w(q(s)) ,\tau_q(s) \rangle \tau_q(s)$ denote its normal components with respect to the curve.
\end{proposition}

\begin{proof}
The explicit derivation of this expression follows directly from the computation of the Fréchet differential detailed in the proof of \cref{prop:frechet_varifold}.
\end{proof}

Motivated by computational applications, we primarily rely on the varifold representation $\mu_q$ of a curve $q$ by a weighted sum of Dirac measures, defined in \cref{eq:varifold_discrete}. The following proposition provides the expression of the first variation associated with $\mu_q$. Before stating the result, for a face $f$, we recall the notations $\ell_f=\Vert q_{f^2} - q_{f^1} \Vert$, $c_f = \dfrac{q_{f^1}+q_{f^2}}{2}$ and $\vec{t}_f = \dfrac{q_{f^2}-q_{f^1}}{\Vert q_{f^2}-q_{f^1} \Vert}$. 

\begin{proposition}[Expression of the first variation for weighted sums of Diracs] \label{exam:formula_first_var_discrete}A varifold representation of the curve $q$, for a chosen discretization $((q_i)_i,F_q)$, by a weighted sum of Diracs is given by  $\mu_q : \omega \mapsto \sum_{f \in F_q} \ell_f \omega(c_f,\vec{t}_f)$.   For  $v \in C_0^2(\R^d,\R^d)$ and $\omega \in C_0^2(\R^d \times \mathbb{S} ^{d-1},\R)$, the first variation is expressed by 
    \begin{equation} \label{eq:first_var_discrete}
    \delta \mu_{q}(v)(\omega)= \sum_{f \in F_q}  \ell_f (\delta \ell_f)_v ~ \omega(c_f,\vec{t}_f) + \ell_f ~  \nabla_x \omega(c_f,\vec{t}_f)^{\top}  v_f  + \ell_f   \nabla_T \omega(c_f,\vec{t}_f)^{\top}  (\nabla^\perp v_f)
    \end{equation}
    where $v_f= \dfrac{v(q_{f^1})+v(q_{f^2})}{2}$, $(\delta \ell_f)_v = \langle \dfrac{v(q_{f^2})-v(q_{f^1})}{\ell_f}, \vec{t}_f\rangle $ and $ \nabla^\perp v_f=\dfrac{v(q_{f^2})-v(q_{f^1})}{\ell_f} - \Big \langle \dfrac{v(q_{f^2})-v(q_{f^1})}{\ell_f} , \vec{t}_f\Big \rangle\vec{t}_f$.
\end{proposition}

\begin{remark} \label{rem:decomp} The first variation in \cref{eq:first_var_discrete} is a combinations of variations of masses $\ell_f (\delta \ell_f)_v \omega(c_f,\vec{t}_f)$, positions $\ell_f \nabla_x \omega(c_f,\vec{t}_f)^{\top} v_f$ and tangents  $\ell_f \nabla_T \omega(c_f,\vec{t}_f)^{\top} (\nabla^{\perp}v_f)$. The interpretation of each different terms will be later studied in \cref{sec:rkhs_metric_first_var}.
\end{remark}

The first variation extends naturally to union of curves, allowing for a decomposition of complex shape.

\begin{proposition}[Additivity of the first variation] \label{prop:additivity_first_var}
    Let $q_A, q_B \in Q$ be two  distinct curves and let $q_{A \cup B}$ denote their union. For any vector field $v \in C_0^2(\mathbb{R}^d, \mathbb{R}^d)$, the first variation is additive:
    \begin{equation*}
        \delta \tilde{\mu}_{q_{A \cup B}}(v) = \delta \tilde{\mu}_{q_A}(v) + \delta \tilde{\mu}_{q_B}(v)
    \end{equation*}
\end{proposition}

\begin{proof}
    This directly comes from the additivity of the varifold representation stated in \cite{Kaltenmark_2017_CVPR}[Remark 1].
\end{proof}

Recall from \cref{exam:discretized_curve} that the continuous varifold representation $\tilde{\mu}_q$ can be approximated by a weighted sum of Dirac masses $\mu_q$ for a sufficiently fine discretization. The following proposition states that, for a vector field $v \in C_0^2(\R^d,\R^d)$, the first variation $\delta \tilde{\mu}_q(v)$
can be approximated by $\delta \mu_q(v)$ for a sufficiently fine discretization. This approximation justifies the evaluation of the first variation using $\delta \mu_q(v)$ instead of its continuous counterpart $\delta \tilde{\mu}_q(v)$, providing a more computationally tractable formulation.

\begin{proposition}[Approximation of the first variation of varifolds]\label{prop:approx_first_var}
    Let $q \in Q$  and $((q_i)_i,F_q) $ a discretization of $q$. For any vector field $v \in C_0^2(\mathbb{R}^d,\mathbb{R}^d)$, 
    \begin{equation*}
        \left\Vert \delta \tilde{\mu}_q(v) - \delta \mu_q(v) \right\Vert_{C_0^2(\mathbb{R}^d \times \mathbb{S}^{d-1},\mathbb{R})'} \leq C_q \Vert v \Vert_{C_0^2} \max_{f \in F_q} \ell_f
    \end{equation*}
    where $C_q > 0$ is a constant that depends on the curve and its discretization.
\end{proposition}
\begin{proof} 
    See \cref{app:proof_approx}.
\end{proof}

The first variation of varifolds $\delta \tilde{\mu}_q(v)$ quantifies the variation of $\operatorname{im}(q)=\{q(s) \mid s \in I \}$ under the action of the vector field $v$. Because this representation is invariant by reparameterization, it does not capture the action of vector fields that leaves the image of the curve unchanged. For example, for a curve $q$ shaped as a circle and a non-zero rotation vector field $A$, it follows that $\delta \tilde{\mu}_q(A)=0$ since a rotation leaves invariant the circle. The following proposition states that a first variation equals to zero imply that the image of the curve $\operatorname{im}(q)= \{ q(s) \mid s \in I \}$ is invariant under the action of $v$.
\begin{proposition}[Invariance of the curve under the action of a vector field] \label{prop:continuous_invariance}
    Let $q \in C^1(I, \mathbb{R}^d)$ be an immersed curve with a finite number of self-intersections and let $v \in C_0^2(\mathbb{R}^d, \mathbb{R}^d)$ be a vector field. If $\delta \tilde{\mu}_q(v) = 0$, then $\operatorname{im}(q)$ is invariant by the action of $v$.
\end{proposition}

\begin{proof}
    The proof is given in \cref{app:continuous_invariance}
\end{proof}

As presented earlier, the first variation can be defined for a continuous varifold representation $\tilde{\mu}_q$, a discretized representation $\hat{\mu}_q$, or an approximation by Diracs $\mu_q$. In this paper, unless otherwise specified, the term first variation refers to the variation of the Dirac approximation, denoted by $\delta \mu_q(v)$.

\subsubsection{RKHS metric on first variations of varifolds} \label{sec:rkhs_metric_first_var}
The RKHS framework developed for varifolds can be adapted to first variations of varifolds by considering a RKHS $H \hookrightarrow C^2_0(\R^d \times \mathbb{S}^{d-1},\R)$ generated by a kernel $k_{H}=k_E \otimes k_T \in C_0^4((\R^d \times \mathbb{S}^{d-1})^2,\R)$. Note that the regularity of the kernel has to be exactly twice larger than the regularity of its associated RKHS, as detailed in \cite[Theorem 2.11]{micheli2013matrixvaluedkernelsshapedeformation}. For the applications and numerical examples presented in this work, the comparison of first variations via the RKHS norm is performed using only the kernel on positions $k_E$, ensuring the resulting expression remain mathematically simple and interpretable. Since the $C_0$-universality property no longer holds for a degenerated kernel on tangents $k_T=1$, we will rather consider, when needed, restriction of varifolds and first variations, initially defined over $\R^d \times \mathbb{S}^{d-1}$, to $\R^d$. 

The image representation of varifolds thanks to Riesz isomorphism can be extended to first variations. Using the expression of $\delta \mu_q(v)$ given in \cref{eq:first_var_discrete} and kernel properties, the first variation can also be represented as an image:

\begin{equation} \label{eq:RKHS_first_var}
    K_{H}\delta \mu_q(v) : x \in \R^d \mapsto  \sum_{f \in F_q} \ell_f (\delta \ell_f)_v ~ k_E(c_f, x) + \ell_f   \nabla_1 k_E(c_f,x)^{\top} v_f  
    \end{equation} 

\begin{remark}
$K_H \delta \mu_q(v)$ is a sum of functions with non-constant sign, contrary to the classic image representation of a varifold $K_H \mu_q : x \mapsto \sum_{f \in F_q} \ell_f k_E(c_f,x)$ which is a sum of non-negative functions. A consequence is that the change of sign will lead to compensatory effect due to summation over edges.
\end{remark}

The dual of the first variation $K_H \delta \mu_q(v)$ can be decomposed into a variation of mass, 
\begin{equation} \label{eq:mass_var}
K_H \delta \mu_q(v)^{\text{mass}} (x)= \sum_{f \in F_q} \ell_f (\delta \ell_f)_v ~ k_E(c_f, x) 
\end{equation}
and a variation of position
\begin{equation} \label{eq:pos_var}
    K_H \delta \mu_q(v)^{\text{pos}}(x) =  \sum_{f \in F_q} \ell_f   \nabla_1 k_E(c_f,x)^{\top} v_f \,.
\end{equation}
This decomposition directly comes from the decomposition of $\mu_q$ in variation of masses $\ell_f$ and positions $c_f$ as evoked in \cref{rem:decomp}.

\cref{fig:decomp_first_var} illustrates this decomposition by showing the first variation of varifolds associated with a segment and induced by scaling vector fields $x \mapsto \dfrac{1}{2}(x-a)$ centered at the left extremity of the segment $a=(-2,0)$ (\cref{fig:scaling1}) and under the segment $a=(0,-2)$ (\cref{fig:scaling2}). The curve and the vector field are represented on the first panel, followed by its associated image representation of the first variation given by \cref{eq:RKHS_first_var}, its decomposition in a variation of mass (\cref{eq:mass_var}) and in a variation of positions (\cref{eq:pos_var}). For the first vector field centered in $a=(-2,0)$ (\cref{fig:scaling1}), the infinitesimal deformation induced by the vector field is a stretching of the segment towards the right. Although the parameterized curve $q$ is modified along its entire support, the first variation only captures the deformation of the geometric shape, identified as $\operatorname{im}(q)$, on the right extremity. In the left and the middle of the segment there is no intensity due to the invariance by reparameterization of the first variation. The variation of mass (third column), which is exactly the same for both vector fields, can be interpreted as a Gaussian smoothing of the mass variation term  $ (\delta \ell_f)_v = \langle \dfrac{v(q_{f^2}) - v(q_{f^1})}{\ell_f} , \dfrac{q_{f^2} - q_{f^1}}{\ell_f} \rangle $ which is constant here, equal to $\dfrac{1}{2}$ along the segment. The variation of position (fourth column) is represented as a global intensity gradient resulting from a sum of local translation gradients oriented from left to right with growing intensity (see \cref{fig:dirac_translation_all} below that illustrates a similar phenomenon). This phenomenon is analogous to a propagating wave, where the movement induces a drop in intensity behind the wave  and a rise in intensity at the wavefront. \cref{fig:scaling2} shows curved patterns in the representation of the variation of position, illustrated in the fourth panel. This phenomemon is due to the Gaussian smoothing which consists in a spatial diffusion of the gradients $\nabla_1k(x,c_f)$ along $v_f$, the evaluations of the vector field on the curve.

\begin{figure}[!h]
\centering
\begin{subfigure}{\textwidth}
    \centering
    \includegraphics[width=\textwidth]{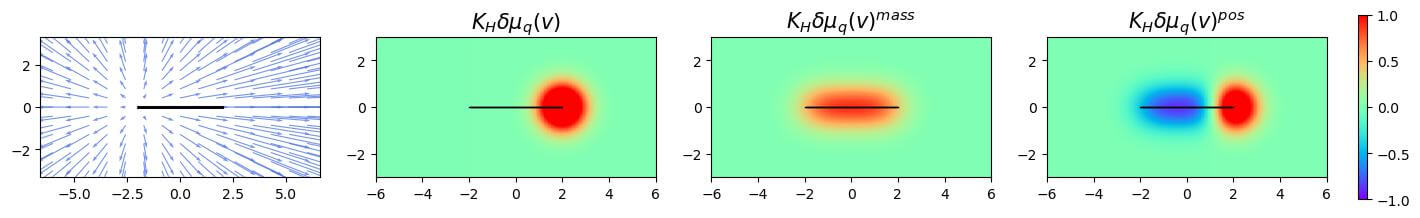}
    \caption{The vector field represented is $x \mapsto \dfrac{1}{2}(x-a)$ which is a scaling centered in $a=(-2,0)$.}
    \label{fig:scaling1}
\end{subfigure}
\vfill
\begin{subfigure}{\textwidth}
    \centering
    \includegraphics[width=\textwidth]{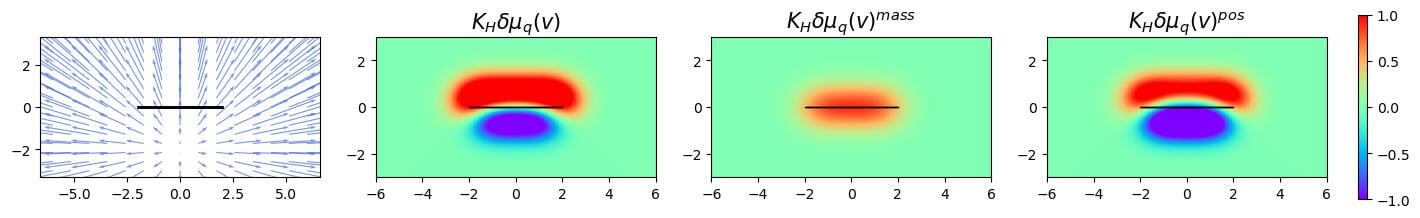}
    \caption{The vector field represented is $x \mapsto \dfrac{1}{2}(x-a)$ which is a scaling centered in $a=(0,-2)$.}
    \label{fig:scaling2}
\end{subfigure}
\caption{Both subfigures illustrate a curve representing a segment in black, a scaling vector field (blue arrows) and the associated first variation with its decomposition in variation of mass and position given by \cref{eq:RKHS_first_var}, \cref{eq:mass_var} and \cref{eq:pos_var}. The RKHS considered is generated by a Gaussian kernel with scale $\sigma=1$. }
\label{fig:decomp_first_var}
\end{figure}

As detailed for varifolds in \cref{inner_product_varifold}, we introduce the inner product between first variations by restricting the space of test functions to a RKHS $H$ induced by a kernel $k_E$. 
Given two vector fields $v,w \in C_0^2(\R^d,\R^d)$ and two curves $q_a$ and $q_b$, the expression of the inner product between their respective first variations in the dual space $H'$ is given by:
\begin{align*} \label{inner_product_first_var}
\langle \delta \mu_{q_a}(v), \delta \mu_{q_b}(w)\rangle_{H'} 
    &= \sum_{\substack{f_a \in F_{q_a} \\ f_b \in F_{q_b}}} \ell_{f_a} \ell_{f_b} \left( (\delta \ell_{f_a})_{v} (\delta \ell_{f_b})_{w} k_E(c_{f_a},c_{f_b})  +   (v_{f_a})^{\top} \nabla^2_{1,2}k_E(c_{f_a},c_{f_b}) w_{f_b} \right) \notag\\
    &+ \sum_{\substack{f_a \in F_{q_a} \\ f_b \in F_{q_b}}}  \ell_{f_a} \ell_{f_b} \left( \delta \ell_{f_a})_{v}  \nabla_1 k_E(c_{f_a},c_{f_b})^\top w_{f_b} + (\delta \ell_{f_b})_{w} \nabla_2 k_E(c_{f_a},c_{f_b})^{\top} v_{f_a} \right) 
\end{align*}

\begin{remark}
As for varifolds, this inner product defines a metric on  $H'$ but defines a metric on $C_0^2(\R^d \times \mathbb{S}^{d-1},\R)'$ only for $C_0$-universal kernels.
\end{remark}

Similarly, this inner product can be split into inner products between mass and position variations using the decomposition given by \cref{eq:mass_var} and \cref{eq:pos_var}. To study these inner products, we focus on the fundamental case of first variations associated with two segments $f$ and $g$ that represent edges of a discretized curve $q$, and two different vector fields $v,w \in C_0^2(\mathbb{R}^d,\mathbb{R}^d)$. We define $\mu_f = \ell_f \updelta_{c_f}$ and $\mu_g = \ell_g \updelta_{c_g}$ as two weighted Dirac measures, where $\ell_f$ and $c_f$ denote the length and the center of the edge $f$ (and analogously for $g$). Then, the first variation of the varifold $\mu_f$ induced by $v$ is 
\begin{eqnarray*}
K_H \delta \mu_f(v) &=& \ell_f (\delta \ell_f)_v  k_E(c_f,\cdot) + \ell_f  \nabla_1 k_E(c_f,\cdot)^{\top} v_f \\
&=& K_H \delta \mu_f(v)^{\text{mass}} + K_H \delta \mu_f(v)^{\text{pos}}.
\end{eqnarray*}

In what follows, we assume that the RKHS inner products are defined using a Gaussian kernel with scale $\sigma$, given by $k_{\sigma}(x,y)= \exp(-\dfrac{\Vert y - x\Vert^2 }{\sigma^2})$. Consequently, this kernel and its partial derivatives vanish when the segments $f$ and $g$ are very far apart. In other words, for $\Vert c_f - c_g \Vert \gg \sigma$, $\delta \mu_f(v)$ and $\delta \mu_g(v)$ become almost orthogonal. In the remainder of this section, we assume $\sigma$ is chosen large enough to avoid this phenomenon.
\paragraph{Inner product between mass variations} The inner product is given by
\begin{equation} 
\langle \delta \mu_f(v)^{\text{mass}} , \delta \mu_g(w)^{\text{mass}} \rangle_{H'} = \ell_f \ell_g (\delta\ell_f)_v (\delta \ell_g)_w   k_E(c_f,c_g)
\end{equation}
This inner product is characterized by $(\delta \ell_f)_v$ and $(\delta \ell_g)_w$ representing length variations of the segments under the action of vector fields. Indeed, when $(\delta \ell_f)_v = \langle  \dfrac{v(q_{f^2}) - v(q_{f^1})}{\ell_f},\vec{t}_f\rangle < 0$, the segment $f$ shrinks, which results in a local mass decrease. So, the inner product is strictly positive when both vector fields induce the same local mass change (positive or negative). Note that the mass variation $ \delta \mu_f(v)^{\text{mass}} $ is constant equal to zero when $v(q_{f^2})-v(q_{f^1})$ is orthogonal to the tangent $\vec{t}_f$, in particular if $v$ is locally constant on $f$, which corresponds to a translation of the segment $f$.

\paragraph{Inner product between positions variations}\label{par:inner_prod_pos_var}

Recall the notation $v_f = \dfrac{v(q_{f^1}) + v(q_{f^2})}{2}$.
The variation of position of a particle $ K_H \delta \mu_f(v)^{pos} = \ell_f \nabla_1 k_E(c_f,\cdot)^\top v_f$ is illustrated in the first panel of \cref{fig:dirac_translation_ver} and \cref{fig:dirac_translation_hor} for a Gaussian kernel. Using these notations, the inner product between position variations is given by
\begin{equation}\label{eq:inn_prod_pos}
\langle \delta \mu_f(v)^{\text{pos}} , \delta \mu_g(w)^{\text{pos}} \rangle_{H'} = \ell_f \ell_g v_f^{\top} \nabla^2_{1,2} k_E(c_f,c_g) w_g \, .
\end{equation}
In particular, for a Gaussian kernel, the Hessian matrix is
\begin{equation}
    \nabla_{1,2}^2 k_{\sigma}(c_f,c_g) = \dfrac{2}{\sigma^2}\exp\left({-\dfrac{ \Vert c_g-c_f \Vert^2}{\sigma^2}}\right) \left(I_d - \dfrac{2}{\sigma^2} (c_g-c_f)(c_g-c_f)^{\top}\right) 
\end{equation}
Its eigenvalues are $\lambda_{f,g}^{\perp} =\dfrac{2}{\sigma^2} \exp\left({-\dfrac{\Vert c_g - c_f \Vert^2}{\sigma^2}}\right)$ and $\lambda_{f,g}^{//} = \lambda_{f,g}^{\perp} (1-2\dfrac{\Vert c_g - c_f \Vert^2}{\sigma^2})$ which correspond respectively to the $(d-1)$-dimensional orthogonal complement  $(c_g-c_f)^{\perp}$ and the one-dimensional eigenspace  $\operatorname{span}(c_g - c_f)$. Consequently, the vectors $v_f$ and $w_g$ can be decomposed along these eigenspaces, namely $v_f = v_f^{//} + v_f^{\perp}$. The expression of the inner product then simplifies to 
\begin{equation*} 
\langle \delta \mu_f(v)^{\text{pos}} , \delta \mu_g(w)^{\text{pos}} \rangle_{H'} = \ell_f \ell_g (\lambda_{f,g}^{\perp}  \langle v_f^{\perp},w_g^{\perp} \rangle + \lambda_{f,g}^{//} \langle v_f^{//},w_g^{//}\rangle) \, .
\end{equation*}
The eigenvalues are bounded by $(-0.45\dfrac{2}{\sigma^2} \approx)-\dfrac{4 \exp\left({-\dfrac{3}{2}}\right)}{\sigma^2} \leq \lambda_{f,g}^{//} \leq \lambda_{f,g}^{\perp} \leq \dfrac{2}{\sigma^2}$. In particular, for $\Vert c_g - c_f \Vert < \sigma $, the orthogonal eigenvalue is larger in magnitude than the radial one, whereas the opposite holds for $\Vert c_g - c_f\Vert > \sigma$.

To better interpret this expression, we focus on the norm of the first variation associated  with both particles in the particular case where the vector fields $v$ and $w$ are constant to $u \in \R^d$. In this setting, the norm is
\begin{eqnarray}\label{eq:norm_first_var_sum}
\Vert \delta \mu_f(u)^{\text{pos}} + \delta \mu_g(u)^{\text{pos}} \Vert^2_{H'} &=& \Vert \delta \mu_f(u)^{\text{pos}} \Vert^2 + \Vert \delta \mu_g(u)^{\text{pos}} \Vert^2 \\
&+& \dfrac{4\Vert u \Vert^2}{\sigma^2} \ell_f \ell_g \exp\left({-\dfrac{\Vert c_g - c_f\Vert^2}{\sigma^2}}\right)  \left ( 1 - \dfrac{2}{\sigma^2} \left\langle \dfrac{u}{\Vert u \Vert}, c_g - c_f \right\rangle^2  \right) \notag
\end{eqnarray}
Since $\delta \mu_f(v) + \delta \mu_g(v) = \delta (\mu_f + \mu_g)(v) $, the norm can be interpreted as a measure of the infinitesimal displacement of the particles $c_f$ and $c_g$ together by the translation $u$.

\begin{figure}[!h]
    \centering

    \begin{subfigure}{0.7\linewidth}
        \centering
        \includegraphics[width=\linewidth]{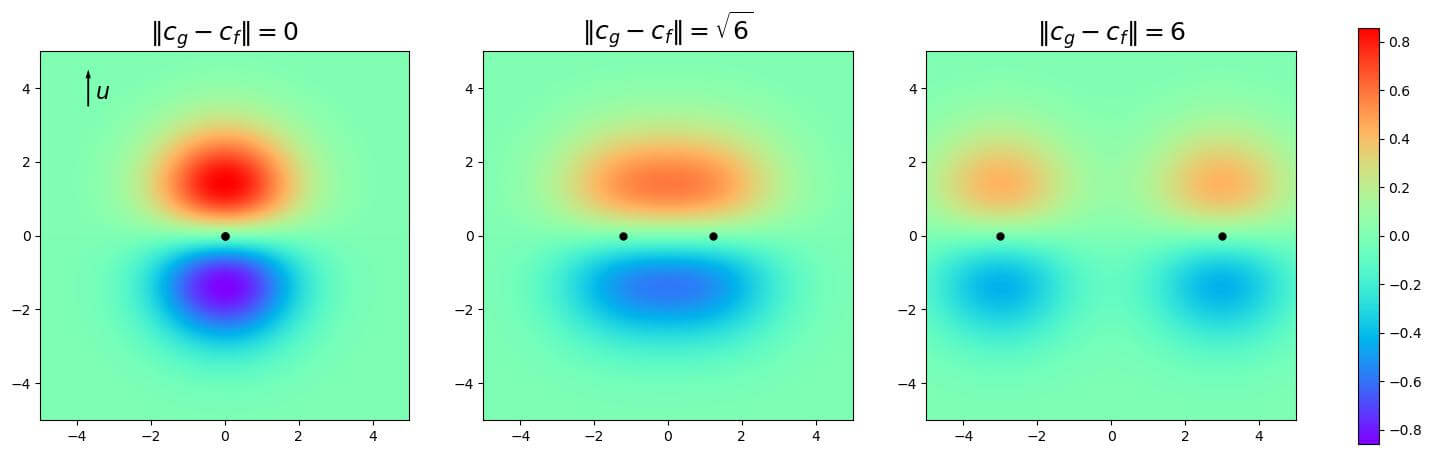}
        \caption{Representation of the first variation $K_H \delta \mu_{fg}(u)$ for $u=[0,1]$ and different distances between particles.}
        \label{fig:dirac_translation_ver}
    \end{subfigure}
    \hfill
    \begin{subfigure}{0.29\linewidth}
        \centering
        \includegraphics[width=1.2\linewidth]{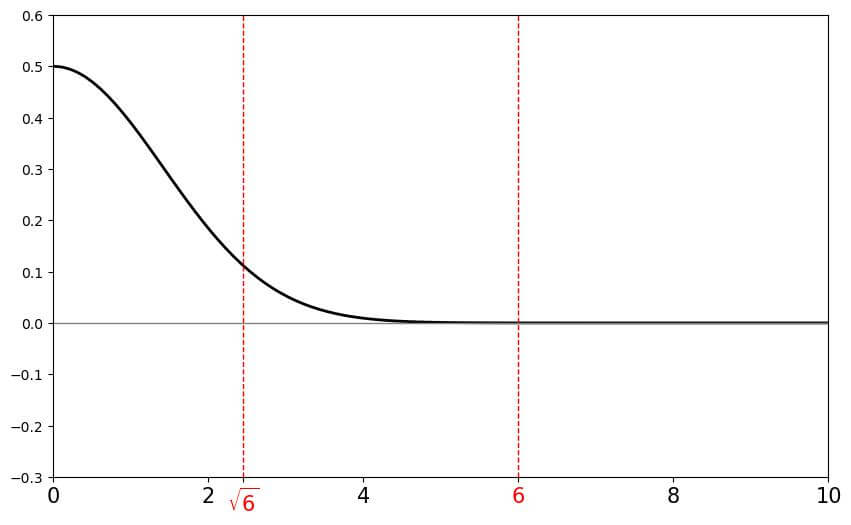}
        \caption{Representation of the mapping $x \mapsto \dfrac{2}{\sigma^2} \exp\left({-\dfrac{x^2}{\sigma^2} }\right)$ for $\sigma=2$.}
        \label{fig:dirac_translation_ver2}
    \end{subfigure}

    \vspace{1em} 
    
    \begin{subfigure}{0.7\linewidth}
        \centering
        \includegraphics[width=\linewidth]{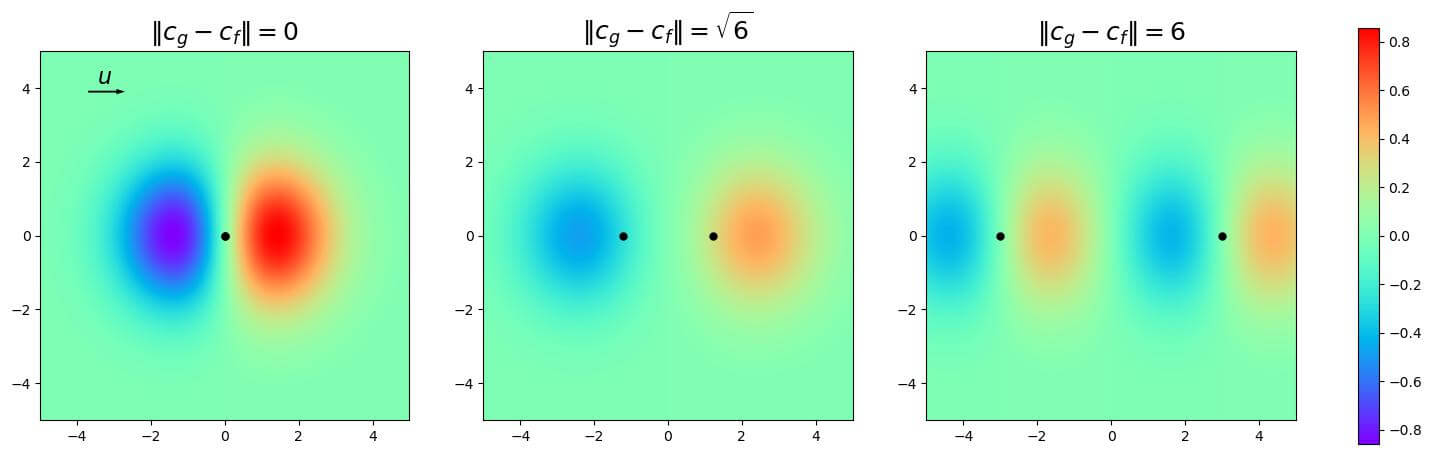}
        \caption{Representation of the first variation $K_H \delta \mu_{fg}(u)$ for $u=[1,0]$ and different distances between particles.}
        \label{fig:dirac_translation_hor}
    \end{subfigure}
    \hfill
    \begin{subfigure}{0.29\linewidth}
        \centering
        \includegraphics[width=1.2\linewidth]{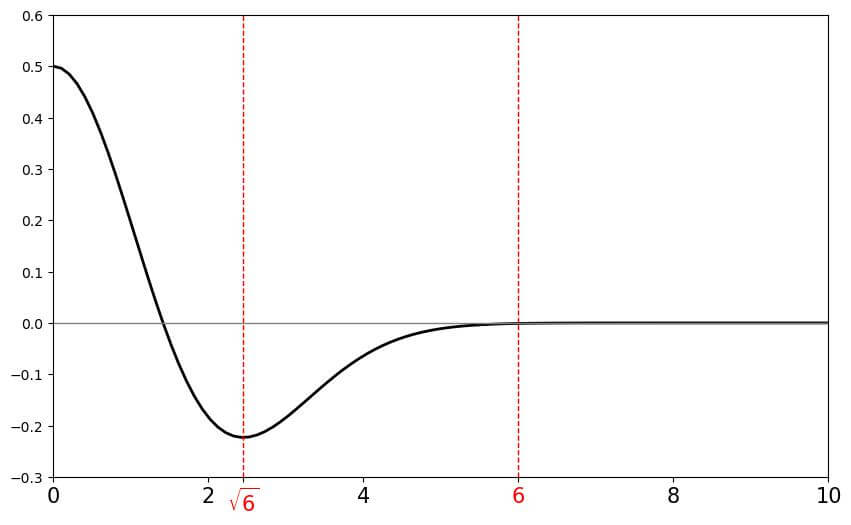}
        \caption{Representation of the mapping $x \mapsto \dfrac{2}{\sigma^2} \exp\left({-\dfrac{x^2}{\sigma^2}}\right)(1-\dfrac{2}{\sigma^2}x^2) $ for $\sigma=2$.}
        \label{fig:dirac_translation_hor2}
    \end{subfigure}

    \caption{Representation of the first variation $K_H (\delta \mu_{f}(u) + \delta \mu_g(u))$ for $u=[0,1]$ (first row) and $u=[1,0]$ (second row). The black points represent the particles $f$ and $g$ respectively located at $c_f$ and $c_g$ with masses $\ell_f=\ell_g=1$. The plots on the right represent the inner product $\langle \delta \mu_f(u)^{\text{pos}},\delta \mu_g(u)^{\text{pos}} \rangle$ with respect to the norm $\Vert c_f - c_g \Vert$. The RKHS is induced by a Gaussian kernel with scale $\sigma=2$. The different values of $\Vert c_g - c_f \Vert \in \{ 0, \sqrt{6},6\}$ corresponds to the case where the particles collapse, to the minimum of the inner product for the orthogonal case and when particles are far apart.
    }
    \label{fig:dirac_translation_all}
\end{figure}

\cref{fig:dirac_translation_all} illustrates the first variation of varifolds associated with both particles $K_H \left( \delta \mu_{f}(u) + \delta \mu_g(u) \right)$ on the left panel and the inner product $\langle \delta \mu_f(u)^{\text{pos}} , \delta \mu_g(u)^{\text{pos}} \rangle$ on the right panel in the particular case of $u=[0,1]$ (\cref{fig:dirac_translation_all}\subref{fig:dirac_translation_ver}) and $u=[1,0]$ (\cref{fig:dirac_translation_all}\subref{fig:dirac_translation_hor}), respectively orthogonal and tangential  to $c_g - c_f$.
On one hand, for $u=[0,1]\in (c_g-c_f)^{\perp}$, illustrated in \cref{fig:dirac_translation_all}\subref{fig:dirac_translation_ver}, the inner product simplifies to 
\begin{equation*}
\langle \delta \mu_f(u)^{\text{pos}} , \delta \mu_g(u)^{\text{pos}} \rangle_{H'} = \dfrac{2}{\sigma^2} \ell_f \ell_g \exp\left({-\dfrac{\Vert c_g - c_f\Vert^2}{\sigma^2}}\right) 
\end{equation*}
This term strictly decreases with the distance  $\Vert c_g-c_f \Vert$. It reaches its maximal value when the particles collapse,  behaving as a single Dirac measure of weight $\ell_f + \ell_g$, and vanishes as the particles move far apart which corresponds to the case of two independent Dirac measures.
Using \cref{eq:norm_first_var_sum}, it follows that $\Vert \delta \mu_f(u) + \delta \mu_g(u) \Vert^2 > \Vert \delta \mu_f(u) \Vert^2 + \Vert \delta \mu_g(u) \Vert^2$. This inequality shows that, for $u \in (c_g - c_f)^\perp$, translating close particles yields a much larger first variation than translating the same particles when they are far apart.
\noindent On the other hand, 
for $u=[1,0] \in \operatorname{span}(c_g-c_f)$,  the inner product reads
\begin{equation*}
    \langle \delta \mu_f(u)^{\text{pos}} , \delta \mu_g(u)^{\text{pos}} \rangle_{H'} = \dfrac{2 \Vert u \Vert^2}{\sigma^2}\ell_f \ell_g  \exp\left({-\dfrac{ \Vert c_g - c_f \Vert^2}{\sigma^2}}\right)  (1 - \dfrac{2}{\sigma^2} \Vert c_g - c_f \Vert^2)
\end{equation*} 
    As illustrated in  \cref{fig:dirac_translation_all}\subref{fig:dirac_translation_hor}, when the two Diracs coincide or are very far apart, the behavior of the inner product is similar to the orthogonal case. However, in this case, the minimum of the inner product is $ -2 \exp\left({-\dfrac{3}{2}}\right)( \approx -0.45)$ reached for $\Vert c_g - c_f \Vert = \sqrt{\dfrac{3}{2}}\sigma$ which is lower than the minimum in the orthogonal case due to compensation effect caused by the interaction term.  Using \cref{eq:norm_first_var_sum}, for $\Vert c_g - c_f \Vert > \dfrac{\sigma}{\sqrt{2}}$, it follows that $\Vert \delta \mu_f(u) + \delta \mu_g(u) \Vert^2 < \Vert \delta \mu_f(u) \Vert^2 + \Vert \delta \mu_g(u) \Vert^2$. This indicates that translating close particles generates a smaller first variation than translating particles that are far apart. This phenomenon corresponds to the non-positive part for $\Vert c_g - c_f\Vert > \sqrt{2}$, in the graph shown in \cref{fig:dirac_translation_hor2}. Indeed, since $u \in \operatorname{span}(c_g-c_f)$, the translation displaces both particles along their common axis and tend to substitute one particle with the other, leaving the overall shape, constituted of both particles, minimally perturbed. Conversely, the opposite holds for $\Vert c_g - c_f \Vert < \dfrac{\sigma}{\sqrt{2}}$ since the particles are considered as a single particle with a larger mass to displace.

\paragraph{Inner product between mass and position variations} The inner product is given by

\begin{equation}
\langle \delta \mu_f(v)^{\text{pos}} , \delta \mu_g(w)^{\text{mass}} \rangle_{H'} = \ell_f (\delta \ell_g)_v \nabla_1 k_E(c_f,c_g)^{\top} v_f
\end{equation}
and for a Gaussian kernel, the expression simplifies to 
\begin{equation*}
\langle \delta \mu_f(v)^{\text{pos}} , \delta \mu_g(w)^{\text{mass}} \rangle_{H'} = \dfrac{2}{\sigma^2} \ell_f (\delta \ell_g)_v \exp\left({-\dfrac{\Vert c_g - c_f\Vert^2}{\sigma^2}}\right)  \langle v_f,c_g - c_f\rangle \, .
\end{equation*}
The two main terms involved in this inner product are $(\delta \ell_g)_v$ and $\langle v_f, c_g-c_f\rangle $. The first one corresponds to a variation of mass, that corresponds to an increase of mass when positive and a decrease when negative as discussed earlier. The second term $\langle v_f , c_g - c_f \rangle $ captures the displacement of the particle $f$ relatively to $g$ under the action of the vector field $v$. Therefore, $\delta \mu_f(v)^{\text{pos}}$ and $\delta \mu_g(w)^{\text{mass}}$ are orthogonal when $v$ does not affect the mass of $g$, i.e $(\delta \ell_g)_v = 0$ or when $v_f$ is orthogonal to the segment $c_g-c_f$.

\paragraph{Influence of the scale for Gaussian kernel}
The study of these inner products shows that, for a Gaussian kernel, the scale parameter $\sigma$ acts as a balance between mass and position variations. For small scales, the metric perceives the curve as a collection of distinct particles that do not interact. Consequently, the inner product is highly sensitive to local displacement of these particles, leading the position variation to dominate the inner product. Conversely, for large scales, spatial positions are weakly discriminated, meaning that the curve is perceived by the kernel almost as a single particle. Therefore, the mass variation becomes the predominant term in the inner product.

\subsubsection{The example of translations}
\label{sec:example_translations}

Let $W$ be the space of translation vector fields in $\R^d$, that is  $W \simeq \R^d$.
We first recall the notations introduced previously. A curve $q$ can be discretized by considering a collection of landmarks $(q_i)_{i=1,...n}$ along with a set of edges $F_q \subset   \{ 1,...,n \}^2$. For any edge $f \in F_q$, we denote its length by $\ell_{f}=\Vert q_{f^2} - q_{f^1} \Vert$, its center by $c_f=\dfrac{q_{f^2} + q_{f^1}}{2}$ and the average action of a vector field $v$ on the edge $f$ by $v_f =  \dfrac{v(q_{f^1})+v(q_{f^2})}{2}$. Considering these notations, the expression of the first variation, given in \cref{eq:first_var_discrete}, for a translation vector field $u \in \R^d$ simplifies to
\begin{equation} \label{first_var_trans}
    \delta \mu_q(u) : \omega \mapsto  \sum_{f \in F_q} \ell_f  \nabla \omega(c_f)^{\top} u \, .
\end{equation}

\begin{remark}
The first variation induces by a translation involves only variations of position and no variations of mass since a translation leaves the mass invariant, that is $\delta \mu_q(u) = \delta \mu_q^{pos}(u)$.
\end{remark}

\noindent The expression of its image representation using a RKHS $H$ generated by a kernel $k_E$ is
\begin{equation}
    K_H \delta \mu_q(u) : x \mapsto \sum_{f \in F_q} \ell_f  \nabla_1 k_E(c_f,x)^{\top} u \, .
\end{equation}
\cref{fig:triangle} illustrates the image representation of the first variation associated with a triangle $q$ and induced by translations $ u_1 = [0,-1] $ and $u_2=\dfrac{1}{\sqrt{2}}[1,1]$. The RKHS $H$ is generated by a Gaussian kernel with scale $\sigma \in \{1,3\}$. For both translations $u_1$ and $u_2$, the case $\sigma=3$ shows that the triangle can be almost identified to a single Dirac. Therefore, the most interesting case to study is the $\sigma=1$. For the translation $u_1=[0,-1]$ shown in \cref{fig:triangle_bas}, along the vertical face of the triangle, the intensity nearly vanishes, excluding edges effects. Indeed, the infinitesimal translation acts tangentially to this face so it acts as a local reparameterization and as explained previously, the first variation is invariant by reparameterization, leading to zero local variation. In the neighborhood of the two other faces, the translation induces a gradient of intensity oriented in the direction of the translation. This gradient corresponds to a small bump arising from the derivative of the Gaussian kernel, reflecting the variation of position of these points. \cref{fig:triangle_diag} illustrates the first variation induced by a translation $u_2=\dfrac{1}{\sqrt{2}}[1,1]$ which has the exact same norm as $u_1$. In this case, the representation of the first variation has higher intensity since the translation vector $u_2$ is not parallel to any of the triangle's faces which implies there is no reparameterization phenomenon. Because every edges is deformed by the translation $u_2$, it leads to a RKHS norm $\Vert \delta \mu_q(u_2) \Vert$ larger than  $\Vert \delta \mu_q(u_1) \Vert$.
\begin{figure}[!h]
    \centering
    \begin{subfigure}{\linewidth}
        \centering
        \includegraphics[width=0.8\linewidth]{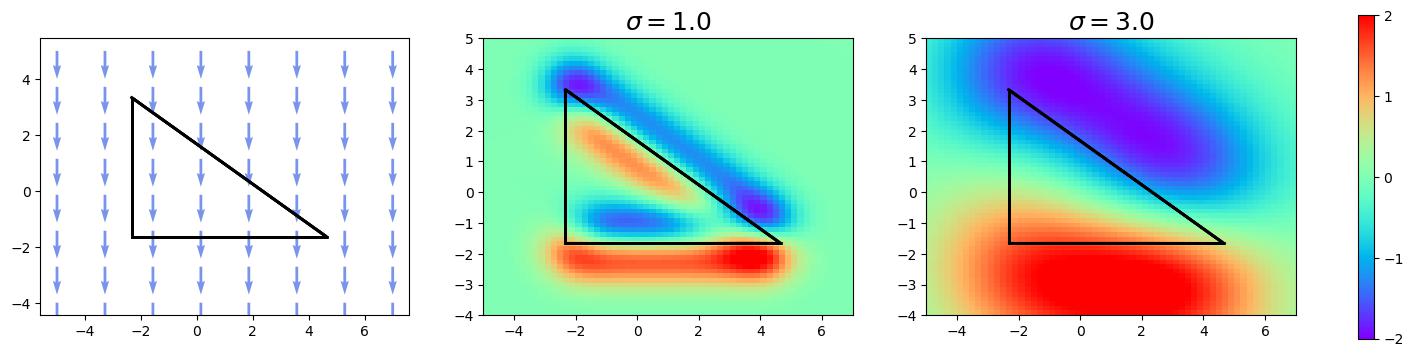}
        \caption{First variation of the varifold associated with the translation $u=[0,-1]$. For $\sigma=1$, the RKHS norm is $\Vert \delta \mu_q(u_1) \Vert = 43.6$}
        \label{fig:triangle_bas}
    \end{subfigure}
    
    \vspace{1.5em} 
    
    \begin{subfigure}{\linewidth}
        \centering
        \includegraphics[width=0.8\linewidth]{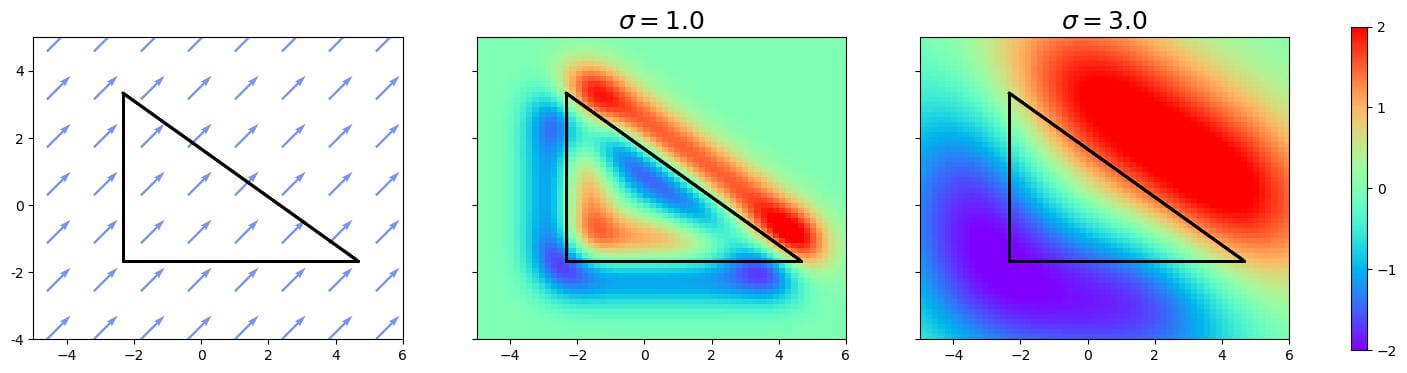}
        \caption{First variation of the varifold associated with the translation $u=\dfrac{1}{\sqrt{2}}[1,1]$.  For $\sigma=1$, the RKHS norm is $\Vert \delta \mu_q(u_2) \Vert = 48$}
        \label{fig:triangle_diag}
    \end{subfigure}
    
    \caption{Representation of the first variation of the varifold associated with a triangle and different translation vector fields. The RKHS is induced by a Gaussian kernel with scale $\sigma \in \{1, 3\}$.}
    \label{fig:triangle}
\end{figure}

\section{Decoupling actions of vectors fields}\label{sec:correlation}
 
We first remind that we only present results for curves embedded in $\R^d$, even though they can be generalized to any rectifiable submanifolds embedded in $\R^d$.

In the large deformation framework \cite{LDDMM}, given a RKHS of vector fields $V\hookrightarrow C_0^2(\R^d,\R^d)$ and $v \in L^2([0,1],V)$ a time-varying vector field, a deformed curve can be generated by solving the ODE $\dot{q}_t = v_t \circ q_t$ where $(v,q) \mapsto v \circ q$ is the infinitesimal action of $V$ on the space of curves. This setting can be extended by considering sum of different spaces of vector fields. For instance, for $W \hookrightarrow C_0^2(\R^d,\R^d)$ another RKHS, the sum $V+W$ is also an RKHS and the evolution of the deformed curve generated by $(v,w) \in L^2([0,1],V \times W)$ is given by the evolution equation $\dot{q}_t = v_t \circ q_t + w_t \circ q_t$. Note that we can also define more complex deformations by assuming that $W$ depends on the state of the curve $q$, as is the case in the modular framework \cite{gris_module}. However,  it is often difficult to clearly distinguish the specific contribution of each mode of deformation, since a part of the infinitesimal deformation induced by a vector field $v \in V$ might be replicated by the infinitesimal deformation induced by  another vector field $w \in W$. To perform meaningful statistics on these different deformations, it is essential to ensure that they are disentangled. To address this issue, we introduce a coupling score that quantifies this replication phenomenon. A core point of our approach is that we compare these infinitesimal deformations through the first variations, rather than  directly comparing the infinitesimal actions or the vector fields themselves.
\begin{figure}[!h]
    \centering
    \begin{subfigure}{\linewidth}
        \centering
        \includegraphics[width=\linewidth]{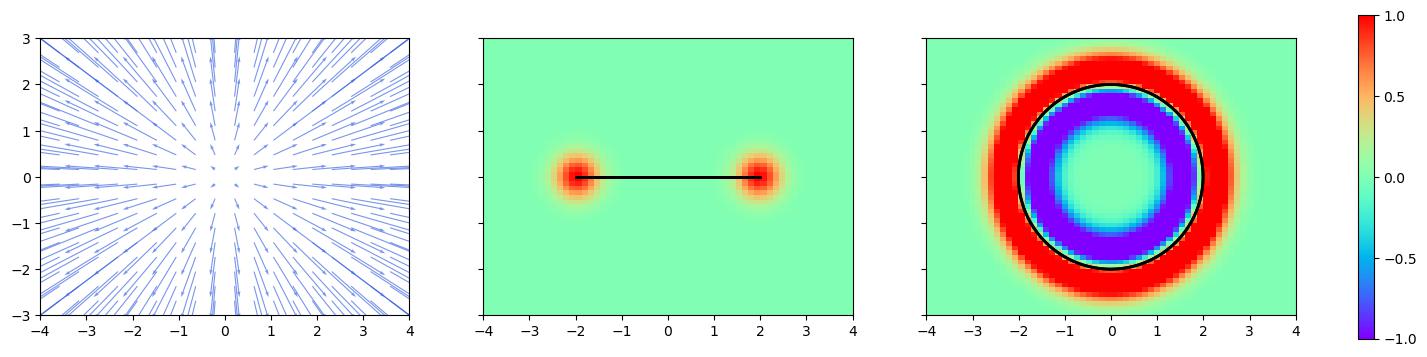}
        \caption{Scaling vector field and its induced first variation of varifolds associated with a segment and a circle}
        \label{fig:scale_vfield}
    \end{subfigure}

    \vspace{1em} 
    
    \begin{subfigure}{\linewidth}
        \centering
        \includegraphics[width=\linewidth]{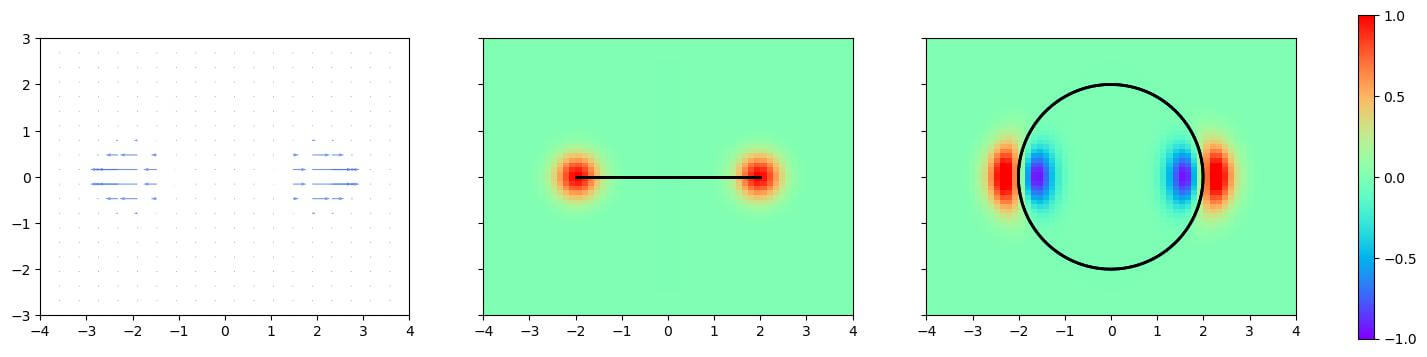}
        \caption{Horizontal stretching vector field and its induced first variation of varifolds associated with a segment and a circle}
        \label{fig:stretch_vfield}
    \end{subfigure}

\caption{Comparison of the first variations of varifolds induced by two vector fields (horizontal strectching and scaling) on two curves (segment and circle). This figure illustrates that different vector fields may induce different or similar first variation of varifolds depending on the curve.}
\label{fig:rod_circle_comparison_variations}
\end{figure}

\cref{fig:rod_circle_comparison_variations} compares the first variations induced by two different vector fields, a scaling and a stretching, on two distinct curves, a segment and a circle. As shown in the second column, both vector fields induce the exact same first variation on the segment. This occurs because the first variation only captures variations of the image of the curve $\operatorname{im}(q)$ and because the scaling is centered on the segment, it acts exactly as an horizontal stretching. In contrast, the third column shows that these two vector fields yield different first variations on the circle, as their respective action induce different variations on $\operatorname{im}(q)$.

\subsection{Coupling score}

Let $V$ and $W$ be two RKHS of vector fields continuously embedded in $C_0^2(\R^d,\R^d)$.
For a curve $q$ and a vector field $v \in V$, the coupling score, defined in the next proposition, measures how the infinitesimal action of $v$ on $q$ can be replicated by the infinitesimal action of an element from $W$ on $q$. Since the proposed coupling score is intended to act as a penalization term improving an existing registration model, the RKHS deformation spaces $V$ and $W$ are considered fixed by the base model and are not treated as tunable parameters in this section.

\begin{definition}[Coupling score]  \label{corr_level}  Let $H \hookrightarrow C_0^2(\R^d \times \mathbb{S}^{d-1},\R)$ be a RKHS of varifolds generated by a kernel $k_{H}=k_E \otimes k_T \in C_0^4((\R^d \times \mathbb{S}^{d-1})^2,\R)$ and let $V$ and $W$ be two RKHS of vector fields continuously embedded in $C_0^2(\R^d,\R^d)$.
The coupling score between a vector field $v \in V$ and the space of vector fields $W$ with respect to a curve $q$ is defined by
\begin{equation}
    C_q(v,W)=\Vert w^* \Vert_{W}^2    
\end{equation}
where 
\begin{equation} \label{min_corr}
    w^* = \operatorname{argmin}_{w \in W} \left\Vert\delta \mu_q(w) - \delta \mu_q(v)  \right\Vert_{H'}^2 + \lambda \Vert w \Vert_{W}^2
\end{equation}
for $\lambda > 0$.
\end{definition}

The existence and uniqueness of the minimizer will be later established in \cref{prop:min_w}. In most applications, the space of vector fields $W$ is finite-dimensional, which implies that $\delta \mu_q(W)$ is closed and therefore ensures the existence of a minimum even without regularization $(\lambda=0)$. The uniqueness of this minimizer is guaranteed provided the linear map $\delta \mu_q$ is injective. This injectivity is generally satisfied except in very specific degenerate cases, such as curves that are invariant under the action of a non-zero vector field in $W$ (e.g., a circle under rotations). For $\lambda = 0$, the optimal vector field $w^*$ is the unique element in $W$ whose associated first variation $\delta \mu_q(w^*)$ is the orthogonal projection of $\delta \mu_q(v)$ onto $\delta \mu_q(W)$ with respect to the varifold RKHS norm. In other words, $w^*$ can be interpreted as the vector field from $W$ that best reproduces the action of $v$ on $q$ via their first variation of varifolds with respect to the norm $\| \cdot \|_{H'}$. Note that the projection $\delta \mu_q(w^*)$, and consequently the coupling score, depends on the choice of $H$. For example, if the varifold RKHS $H$ is generated by a Gaussian kernel, the choice of the scale $\sigma$ significantly influences the coupling score, as discussed in \cref{sec:rkhs_metric_first_var}.

The first variation of varifolds was initially introduced in \cref{def:first_var} as a continuous linear form for a fixed vector field $\delta \mu_q(v) : \omega \mapsto \delta \mu_q(v)(\omega)$ to represent infinitesimal action of vector fields on curves. In the context of coupling, we introduce the first variation operator $\delta \mu_q : v \mapsto \delta \mu_q(v)$ that helps to define a metric on the space of vector fields as detailed later in \cref{sec:stat_to_dyn}. We first prove the continuity of the first variation operator in order to show the existence of the minimizer.

\begin{proposition}[Continuity of $\delta \mu_q$] \label{prop:continuity_first_var}
The first variation of varifolds operator
$$\delta \mu_q : \app{V}{ C_0^2(\R^d \times \mathbb{S}^{d-1},\R)'}{v}{\delta \mu_q(v)}$$ is linear and continuous.

\end{proposition}

\begin{proof}
    The proof is given in \cref{app:proof_continuity_first_var}.
\end{proof}

\begin{corollary}[Continuity of $\delta \mu_q$ in the RKHS topology]
 Let $q \in Q$ be a curve. The first variation of varifolds operator
$$\delta \mu_q : \app{V}{H'}{v}{\delta \mu_q(v)}$$ is linear and continuous.

\end{corollary}

\begin{proof}
    \cref{prop:continuity_first_var} states that $\delta \mu_q : V \to C_0^2(\R^d \times \mathbb{S}^{d-1},\R)'$ is continuous. Since $H$ is continuously embedded in $C_0^2(\R^d \times \mathbb{S}^{d-1},\R)$, the result follows.
\end{proof}
\begin{proposition} \label{prop:min_w}
Let $H$ be a RKHS of varifolds. Given a curve $q$, a vector field $v \in V$ and a regularization parameter $\lambda >0$, the functional 
\begin{equation} \label{min_functional}
    J_v : \app{W}{\R}{w} { \Vert \delta \mu_q(w) - \delta \mu_q(v) \Vert^2_{H'} + \lambda \Vert w \Vert^2_W}
\end{equation}
admits a unique minimizer in $W$ given by
    $$ w^* = ( \Lambda_q \delta \mu^W_q + \lambda \id_W)^{-1} \Lambda_q \delta \mu^V_q(v)$$
where $\delta \mu_q^V \in \mathcal{L}(V,H')$ and $\delta \mu_q^W \in \mathcal{L}(W,H')$ denote the restrictions of $\delta \mu_q$ to $V$ and $W$ respectively, and $\Lambda_q \eqdef K_W (K_H \delta \mu_q^W)^* \in \mathcal{L}(H',W)$.
\end{proposition}

\begin{proof}
    First, $\inf_{w \in W} J_v(w)$ exists since for all $w \in W$, $J_v(w) \geq 0$. Consequently, there exists a minimizing sequence $(w_n)_n$ in $W$ such that $J_v(w_n) \rightarrow \inf_{w \in W} J_v(w)$. Since $(J_v(w_n))_n$ is a converging sequence, $(w_n)_n$ is bounded. Then, up to an extraction, we can assume that there exists a subsequence $(w_{n_k})_k$ converging weakly in $W$. By denoting this weak limit by $w^\infty \in W$, we have $\Vert w^\infty \Vert_W \leq \liminf_{k \to \infty} \Vert w_{n_k} \Vert_W$. Therefore, $\delta \mu_q (w_{n_k})$ weakly converges to $\delta \mu_q (w^\infty)$ in $H'$, which implies $\Vert \delta \mu_q (w^\infty) \Vert_{H'} \leq \liminf_{k \to \infty} \Vert \delta \mu_q (w_{n_k}) \Vert_{H'}$ and $\langle \delta \mu_q (w_{n_k}), \delta \mu_q (v) \rangle_{H'} \to \langle \delta \mu_q (w^\infty), \delta \mu_q (v) \rangle_{H'}$. We conclude that
    \begin{eqnarray*}
        J_v(w^\infty) & = & \Vert \delta \mu_q (w^\infty) \Vert^2_{H'}  - 2 \langle \delta \mu_q (w^\infty), \delta \mu_q (v) \rangle_{H'} + \Vert \delta \mu_q (v) \Vert^2_{H'} + \lambda \Vert w^\infty \Vert^2_W \\
        & \leq & \liminf_{k \to \infty} \Vert \delta \mu_q (w_{n_k}) \Vert^2_{H'}  - \lim_{k \to \infty} 2 \langle \delta \mu_q (w_{n_k}), \delta \mu_q (v) \rangle_{H'} + \Vert \delta \mu_q (v) \Vert^2_{H'} + \lambda \liminf_{k \to \infty} \Vert w_{n_k} \Vert^2_W \\
        & \leq & \liminf_{k \to \infty} J_v (w_{n_k})\\
        & = & \lim_{n \to \infty} J_v(w_n) = \inf_{w \in W} J_v(w)
    \end{eqnarray*}
    so that $w^\infty$ is a minimizer of $J_v$ on $W$. Since $J_v$ is strictly convex, this minimizer is unique. 
    
    Because $J_v$ is differentiable, this unique minimum is determined by the first-order optimality condition $dJ_v(w)(w_0) = 0$ for all $w_0 \in W$, which gives:
    \begin{align*}
        0&=\langle \delta \mu_q^V(v) - \delta \mu_q^W(w), \delta \mu_q^W(w_0) \rangle_{H'} - \lambda \langle w, w_0 \rangle_W \\
       &= \left( \delta \mu_q^V(v) - \delta \mu_q^W(w) \mid K_H \delta \mu_q^W(w_0) \right)_{H',H} - \lambda \langle w, w_0 \rangle_W   \\
        &=\left( (K_H \delta \mu_q^W)^* \left( \delta \mu_q^V(v) - \delta \mu_q^W(w) \right) \mid w_0 \right)_{W',W} - \lambda \langle w, w_0 \rangle_W  \\
       &= \langle K_W (K_H \delta \mu_q^W)^* \left( \delta \mu_q^V(v) - \delta \mu_q^W(w) \right) , w_0 \rangle_W - \lambda \langle w, w_0 \rangle_W .
    \end{align*}
    Since this equality holds for any $w_0 \in W$, it follows that
    \begin{equation*}
        K_W (K_H \delta \mu_q^W)^* \big( \delta \mu_q^V(v) - \delta \mu_q^W(w) \big) - \lambda w = 0.
    \end{equation*}
    By substituting $\Lambda_q \eqdef K_W (K_H \delta \mu_q^W)^*$ and moving the term in $w$ to the other side, we obtain the equation:
    \begin{equation*}
        \big( \Lambda_q \delta \mu_q^W + \lambda \operatorname{id}_W \big) w = \Lambda_q \delta \mu_q^V(v).
    \end{equation*}
    Since, for any $w \in W$, the definition of the adjoint and Riesz isomorphisms gives $\langle \Lambda_q \delta \mu_q^W w, w \rangle_W = \left( \delta \mu_q^W w \mid K_H \delta \mu_q^W w\right)_{H',H} = \Vert \delta \mu_q^W w \Vert_{H'}^2 \ge 0$, the operator  $\Lambda_q \delta \mu_q^W$ is positive-semidefinite. Moreover, since $\lambda > 0$, the operator on the left-hand side is strictly positive-definite and therefore invertible on $W$. Multiplying by its inverse gives the final explicit expression for $w^*$.
\end{proof}

\begin{remark}
    When the functional $J_v$ is strongly regularized ($\lambda \rightarrow \infty$) the minimizer vanishes, that is  $w^*=0$. In absence of regularization ($\lambda \rightarrow 0$), assuming $\delta \mu_q^W$ is injective, the solution becomes $w^* = (\delta \mu_q^W)^{-1} \delta \mu_q^V(v)$. In particular, if $V=W$, then $w^*=v$ and we retrieve that the optimal vector field from $W$ which best reproduces the action of $v \in W$ is $v$ itself.
\end{remark}

\begin{remark} The linearity of the minimizer $w^*$ with respect to $v$ implies that the coupling score 
   $C_q(v,W)$ is quadratic in $v$.
\end{remark}

While \cref{prop:min_w} establishes the minimizer of $J_v$ in a general setting, \cref{prop:closed_form_min_J} below derives its closed-form expression for a discretized curve in $\R^d$, which is the case we are interested in for numerical applications. The Representer Theorem guarantees that this optimal vector field belongs to the finite-dimensional subspace defined in \cref{def:subspace_w}. To characterize this subspace, let $(e_1, \dots , e_d)$ be the canonical basis of $\R^d$ and define the continuous linear evaluation functional $\delta_x^e : w \in W \mapsto \langle w(x), e \rangle$ for $(x, e) \in \R^d \times \R^d$.
\begin{definition} \label{def:subspace_w}
Let $q = (q_i)_{1 \leq i \leq N}$ be a discretized curve in $\R^d$. We define $\Theta_q \subset [| 1, N |] \times [| 1, d |]$ such that the set $\{ w_{i,j} \mid (i,j) \in \Theta_q \}$, where $w_{i,j} := K_W \delta_{q_i}^{e_j}$, forms a maximal linearly independent subset of $\{w_{i,j} \mid 1 \leq i \leq N \,, 1 \leq j \leq d \} \subset W$.
\end{definition}
\begin{example}
    If $W$ is the space of global translations in $\R^d$, then we can set $\Theta_q = \{1\} \times [|1, d |]$. 
\end{example}

\begin{example}
    If $W$ is a Gaussian RKHS, and if the vertices of $q$ are all distinct, then $\Theta_q = [|1, N |] \times [|1, d |]$.
\end{example}

\begin{proposition} \label{prop:closed_form_min_J}
Let $q = (q_i)_{1 \leq i \leq N}$ be a discretized curve in $\R^d$. Using the basis elements $w_{i,j}$ defined above, we construct the column vector 
$$Y = \left( \langle \delta \mu_q (v),  \delta \mu_q (w_{i,j}) \rangle_{H'} \right)_{(i,j) \in \Theta_q} \in \R^{\# \Theta_q}$$ 
and the matrix 
$$S = \left( \langle  (K_W (K_{H} \delta \mu_q)^\ast \delta \mu_q  + \lambda I) w_{i,j}, w_{i',j'} \rangle_{W} \right)_{(i,j) , (i', j') \in \Theta_q} \in \R^{\# \Theta_q \times \# \Theta_q} \,.$$
Then the minimizer of $J_v$ is given by 
$$w^* = \sum_{(i,j) \in \Theta_q} \beta_{i,j} w_{i,j}$$ 
where $\beta = S^{-1}Y$.
\end{proposition}

\begin{proof}
    The proof is given in \cref{app:proof_closed_form}.
\end{proof}

The minimizer is expressed as a linear combination of images of Dirac measures $\delta^{e_j}_{q_k}$ under the mapping $K_W$. Computing this minimum requires inverting a matrix of dimension $(N \times d)^2$ at most. As shown in the previous examples, this dimension strongly depends on the studied space of vector fields $W$ and its associated Riesz isomorphism. While matrix inversion remains computationally tractable for low-dimensional applications, such as isometries or scalings, it becomes expensive for richer spaces $W$. In these scenarios, it is more efficient to directly solve the minimization problem formulated in \cref{min_functional}. The study of the numerical stability and computational efficiency of the numerical methods is out of scope of this work and is left for future works.
\begin{example}
    If $W$ is the space of global translations in $\R^d$, we can fix a single point $q_1$ and set $\Theta_q = \{1\} \times [|1, d|]$. 
    The associated basis elements $e_j = K_W \delta_{q_1}^{e_j}$ are simply the constant vector fields pointing along the canonical axes of $\R^d$. The problem reduces to solving a $d \times d$ linear system $S \beta = Y$.
    The j-th component of the vector $Y$ in $\R^d$ is:
    \begin{equation*}
        Y_j = \langle \delta \mu_q(v), \delta \mu_q(e_j) \rangle_{H'}
    \end{equation*}
    and $S$ is a $d \times d$ symmetric matrix whose entries are:
    \begin{equation*}
        S_{j,j'} = \langle \delta \mu_q(e_j), \delta \mu_q(e_{j'}) \rangle_{H'} + \lambda \updelta_{j,j'}
    \end{equation*}

    The optimal translation vector is then given by $\beta = S^{-1}Y$, which only requires inverting a $d \times d$ matrix. This example is further detailed in \cref{sec:one-dim-space}.
\end{example}

\begin{example}\label{exam:closed_form_dim1}
    If $\dim(W) = 1$, the subset $\Theta_q$ reduces to a singleton $\{(k_0, j_0)\}$. Let $w_0 = K_W \delta_{q_{k_0}}^{e_{j_0}}$ be the associated basis vector. In this particular case, $S$ is the sum of two squared norms $S = \Vert \delta \mu_q(w_0) \Vert_{H'}^2 + \lambda \Vert w_0 \Vert_{W}^2$
    and  $Y = \langle \delta \mu_q(v), \delta \mu_q(w_0) \rangle_{H'}$. 
    The minimizer of $J_v$ is then given directly by:
    \begin{equation*}
        w^* = \left( \dfrac{\langle \delta \mu_q(v), \delta \mu_q(w_0) \rangle_{H'}}{\Vert \delta \mu_q(w_0) \Vert_{H'}^2 + \lambda \Vert w_0 \Vert_{W}^2} \right) w_0
    \end{equation*}
    This example is further detailed in \cref{sec:translation}
\end{example}

\subsection{Applications to finite-dimensional spaces of vector fields}

In this section, we analyze the coupling score between a reference vector field $v \in V$ and various finite-dimensional spaces of vector fields $W$. For the sake of clarity, we simplify the framework, as done in the previous sections, by restricting the definition of first variation to test functions on $\R^d$ rather than $\R^d \times \mathbb{S}^{d-1}$. The examples discussed in this section are illustrated using a Gaussian RKHS with kernel $k_{\sigma}(x,y) = \exp\left({-\dfrac{\Vert y - x \Vert^2}{\sigma^2}}\right)$ although all the results presented hold for any choice of kernels. We remind that the dual of the first variation of varifolds $K_H \delta \mu_q(v)$ induced by a vector field $v \in V$ can be decomposed as a variation of mass and a variation of position
\begin{eqnarray}\label{eq:first_var_decomp}
    K_H \delta \mu_q(v) &=& 
    \sum_{f \in F_q} \ell_f (\delta \ell_f)_v \, k_E(c_f,\cdot) +
        \ell_f  \nabla_1 k_E(c_f,\cdot)^{\top} v_f 
       \\
        &=& K_H \delta \mu_q^{\text{mass}}(v) + K_H \delta \mu_q^{\text{pos}}(v) \notag
\end{eqnarray}
where $\ell_f = \Vert q_{f^2} - q_{f^1} \Vert$ and $(\delta \ell_f)_v = \langle \dfrac{ v(q_{f_2}) - v(q_{f_1}) }{\ell_f} ,\dfrac{ q_{f_2} - q_{f_1} }{\ell_f} \rangle$.
We also remind that the coupling score is defined by 
    $C_q(v,W)=\Vert w^* \Vert_{W}^2$
where the optimal vector field is the solution of the minimization problem
$
    w^* = \operatorname{argmin}_{w \in W} \left\Vert\delta \mu_q(w) - \delta \mu_q(v)  \right\Vert_{H'}^2 + \lambda \Vert w \Vert_{W}^2$.
    
\subsubsection{The example of translations}\label{sec:translation}

We consider the framework presented in \cref{sec:example_translations} where $W \simeq \R^d$ is the space of translation vector fields in $\R^d$. In the case where the vector field $w$ is a global translation constant to $u \in \R^d$, the dual first variation, given by \cref{eq:first_var_decomp}, simplifies to 
\begin{equation} 
   K_H \delta \mu_q(u) = \sum_{f \in F_q} \ell_f  \nabla k(c_f, \cdot)^{\top} u \, .
\end{equation}
The coupling score between a vector field $v \in V$ and the space of translations $\R^d$ quantifies the part of the infinitesimal deformation generated by $v$ that could equivalently result from translating the curve. This score is given by
$$ C_q(v,\R^d) = \Vert u^* \Vert^2$$
where
$u^*$ minimizes the functional $J_v$ from \cref{min_functional}. This functional in the case of translations becomes
\begin{eqnarray}
    J_v(u)&=& \Vert \delta \mu_q(u) - \delta \mu_q (v) \Vert^2_{H'} + \lambda \Vert u \Vert^2 \\
    &=&  u^{\top} (A+\lambda I_d) u- 2b^{\top} u+ c  \notag
\end{eqnarray}
where $c=\Vert \delta \mu_q(v) \Vert^2_{H'}$ is a constant independent of $u$, and the matrix $A$ and the vector $b$ are defined by
\begin{equation}\label{eq:cst}
\left\{
\begin{aligned}
    A &= \sum_{f,g \in F_q} \ell_f \ell_g A_{f,g}  \\
    b &= \sum_{f,g \in F_q} \ell_f \ell_g (\delta \ell_f)_v \nabla_1 k_E(c_f,c_g) + \ell_f \ell_g  A_{f,g} v_f
\end{aligned}
\right.
\end{equation}
with $A_{f,g} = \nabla_{1,2}^2 k_E(c_f,c_g)$.
Although the matrices $A_{f,g}+\lambda I_d$ are not necessarily positive definite (as discussed in \cref{par:inner_prod_pos_var}, \textit{Inner product between position variations}), the matrix $A+\lambda I_d$ is strictly positive definite. Indeed, for  any $X \in \R^d$, we have $ X^{\top}(A+\lambda I_d)X=\Vert \delta \mu_q(X) \Vert_{H'}^2 + \lambda \Vert X \Vert^2$. So $A + \lambda I_d$ is invertible and consequently the functional $J_v$ admits an unique minimum reached for 
\begin{equation}\label{eq:ustar}
    u^*=(A +\lambda I_d)^{-1}b
\end{equation}
We briefly recall the study of the Hessian $A_{f,g}$, which has been detailed for a Gaussian kernel $k_{\sigma}$ in the paragraph \textit{Inner product between position variations} in \cref{par:inner_prod_pos_var}. For such a kernel, the eigenspaces of $A_{f,g}$ are $\operatorname{span}(c_g - c_f)$ and $(c_g - c_f)^{\perp}$ corresponding to the eigenvalues $\lambda_{f,g}^{//}$ and $\lambda_{f,g}^{\perp}$.  Consequently, applying the mixed Hessian $\nabla_{1,2}^2 k_E(c_f,c_g) $ to a vector $a \in \R^d$ decomposes the vector into its radial and orthogonal components $\nabla_{1,2} k_{\sigma}(c_f,c_g)a = \lambda_{f,g}^{//} a^{//} + \lambda_{f,g}^{\perp} a^{\perp} $. 

The second term $b$ in \cref{eq:cst} can be decomposed into a variation of position  
\begin{equation}
    b_{\text{pos}} = \sum_{f,g \in F_q} \ell_f \ell_g  \nabla_{1,2}^2 k_E(c_f,c_g) v_f
\end{equation} and a variation of mass 
\begin{equation} b_{mass}=\sum_{f,g \in F_q} \ell_f \ell_g (\delta \ell_f)_v \nabla_1 k_E(c_f,c_g) \, .\end{equation}
To simplify the interpretation, we suppose $\lambda = 0$. First, if there is no variation of mass, such as in the case of translations or rotations, the optimal translation vector \cref{eq:ustar} simplifies to
\begin{equation}
    u^* = \left(\sum_{f,g \in F_q} \ell_f \ell_g A_{f,g} \right)^{-1}  \left( \sum_{f,g \in F_q}  \ell_f \ell_g  A_{f,g} v_f \right)
\end{equation}

When there is some mass variations, the term $b_{\text{mass}}$ serves as a geometric correction to $b_{\text{pos}}$. Specifically, any edge $f$ experiencing a local mass change due to $v$, i.e $(\delta \ell_f)_v \ne 0$, acts as a source of attraction or repulsion for every other edges $g$ along the axis connecting their centers. An edge gaining mass $((\delta \ell_f)_v > 0 )$ behaves as an attractor, pulling other edges toward it, whereas an edge losing mass ($(\delta \ell_f)_v < 0) $, acts as a repeller, pushing them away. Consequently, $b_{\text{mass}}$ balances the translation $u^*$ to account for these local expansions and shrinkings.

\begin{example} 
    Let $q: s \in [0,1] \mapsto (Ls,0)$ be a horizontal segment of length $L$, and let $v \in C^2(\R^d,\R^d)$ be a vector field inducing an infinitesimal horizontal stretching on the right endpoint of the segment. That is, for every $s \in [0,1]$, the vertical component is $v(q(s))_y= 0$, and the horizontal boundary conditions are $v(q(0))_x=0$ and $v(q(1))_x=a \in \R$. The optimal translation is then given by:
    \begin{equation}\label{eq:exam_ustar}
        u^* = \left(\dfrac{a}{2} \dfrac{1-\exp\left({-\dfrac{L^2}{\sigma^2}}\right)}{\dfrac{\lambda}{2} + 1-\exp\left({-\dfrac{L^2}{\sigma^2}}\right)}, 0 \right).
    \end{equation} 

    The barycenter of the initial segment $q$ is $\Omega = \left(\dfrac{L}{2},0\right)$, and the infinitesimal deformation of the shape by the vector field $v$ transports the barycenter to $\Omega_v = \left(\dfrac{L}{2}+\dfrac{a}{2},0\right)$. However, the transport of $\Omega$ by the translation $u^*$ leads to a new barycenter located at:
    \begin{equation*}
        \Omega_{u^*} = \left(\dfrac{L}{2} + \dfrac{a}{2}\dfrac{1-\exp\left({-\dfrac{L^2}{\sigma^2}}\right)}{\dfrac{\lambda}{2}+1-\exp\left({-\dfrac{L^2}{\sigma^2}}\right)}, 0 \right).
    \end{equation*}

    In the degenerate case where $\sigma \to \infty$, the optimal translation becomes $u^*=0$ and the barycenter remains unchanged. Note that as $\lambda \rightarrow 0$, the optimal translation $u^*$ exactly aligns the barycenters $\Omega_v$ and $\Omega_{u^*}$. For $\lambda >0$, this alignment is relaxed, highlighting the trade-off between minimizing the norm of the optimal translation and matching the barycenters. 
    
  This example echoes the behavior observed in \cref{fig:scaling1}, which illustrated a scaling centered at the left endpoint of the segment. Indeed, the first variation induced by that scaling vector field is similar to the one produced by stretching the right extremity, as presented here.
\end{example}

\begin{example}
 Assume that the curve $q$ is a horizontal segment and $v$ is a vector field such that $v \circ q$ is orthogonal to $q$. For an unregularized projection ($\lambda = 0$) and using a Gaussian kernel $k_\sigma$, the optimal translation $u^*$ reduces to:
    \begin{equation*} 
        u^* = \dfrac{\sum_{f,g \in F_q} \ell_f \ell_g k_{\sigma}(c_f,c_g)v_f}{\sum_{f,g \in F_q} \ell_f \ell_g k_{\sigma}(c_f,c_g)}.
    \end{equation*}
    
   This result highlights that $u^*$ is a weighted average of $v_f$. Because the weights depend on the neighborhood interactions captured by $k_\sigma$, a deformation applied at the center of the segment will have a greater influence on the global translation than a deformation with same intensity located at its extremities.
\end{example}

\subsubsection{One-dimensional spaces}\label{sec:one-dim-space}

In some applications, rather than defining a complex, high-dimensional space $W$ to use in the coupling score $\operatorname{Corr}_q(\cdot,W)$, it might be advantageous to adopt an iterative approach. Following an initial registration step, a specific misalignment can be isolated and encoded into a vector field $w_0$ , as illustrated in \cref{sec:exp_match_pop} and \cref{sec:exp_reg_it}, leading to the definition of the one-dimensional space $W = \operatorname{span}(w_0)$. This specific setting is of practical interest, as it serves as a core component for iteratively refining a registration.

Recall that when $W$ is a one-dimensional space of vector fields, that is $W \simeq \R$, then \cref{exam:closed_form_dim1} states that for $w_0 \in W$ such that $\Vert w_0 \Vert_W = 1$, the optimal vector field $w^*=\alpha^* w_0$ that defines the coupling score is given by
\begin{equation}\label{eq:alpha_star}
    \alpha^* =\dfrac{\Big\langle \delta \mu_q(v), \delta \mu_q(w_0)  \Big\rangle}{\left\Vert \delta \mu_q(w_0)  \right\Vert^2 + \lambda}
\end{equation}

\begin{figure}[!h]
    \centering
    \begin{subfigure}{\linewidth}
        \centering
        \includegraphics[width=\linewidth]{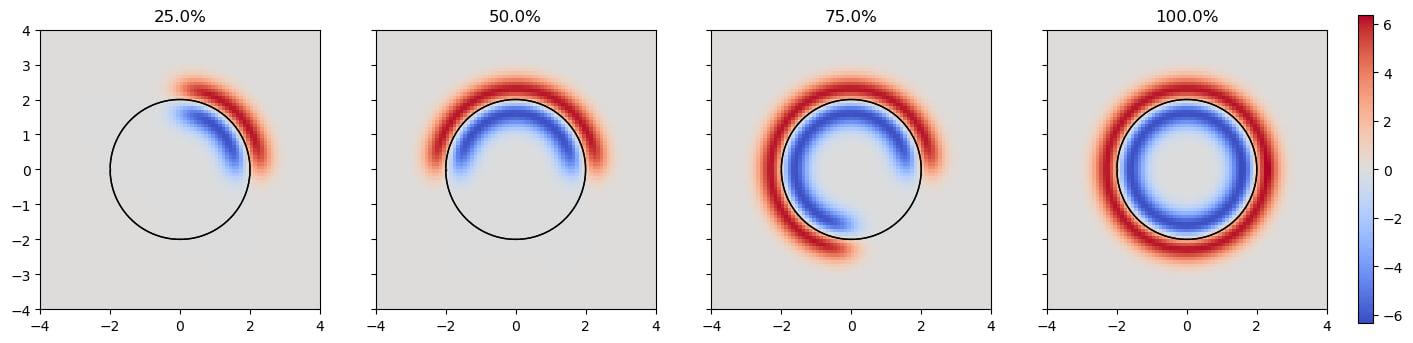}
        \caption{Local scalings applied to 25, 50, 75 and 100\% of the circle.}
        \label{fig:circle_prop}
    \end{subfigure}

    \vspace{1em}
    
    \begin{subfigure}{\linewidth}
        \centering
        \includegraphics[width=\linewidth]{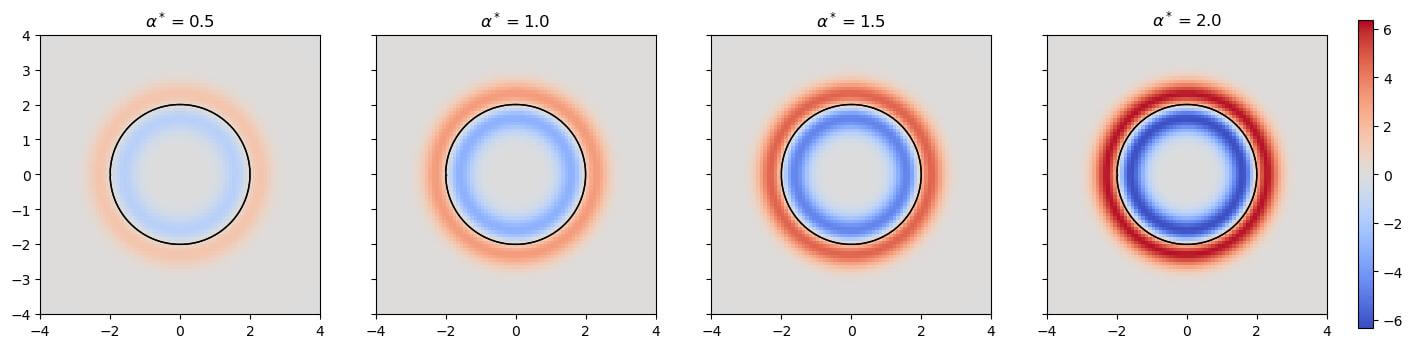}
        \caption{Corresponding optimal global scalings.}
        \label{fig:circle_opti}
    \end{subfigure}
    \caption{The first row represents first variation of a circle, induced by a vector field $v$ that acts as a scaling of factor 2 on 25, 50, 75 and 100\% of the circle, respectively. The second row illustrates the corresponding first variation induced by the optimal scaling vector field $x \mapsto \alpha^* x$ where $\alpha^*$ is given by \cref{eq:alpha_star}. This optimal vector is the scaling that best reproduces the action of the corresponding $v$ from the first row. The RKHS considered is defined by a Gaussian kernel with scale $\sigma=0.5$ and the regularization parameter is set to $\lambda=0$.}
\label{fig:scaling_circle}
\end{figure}

A standard example of a one-dimensional space is the space of scalings $W = \{ w : x \in \mathbb{R}^d \mapsto \alpha x \mid \alpha \in \mathbb{R} \}$ with $w_0 : x \mapsto x$. \cref{fig:scaling_circle} illustrates the first variation of a circle induced by four different vector fields $v$ that act as a scaling of factor $2$ on $25, 50, 75$, and $100\%$ of the circle. Since we set $\lambda=0$, when $v$ acts on the entire circle, the optimal scaling vector field that reproduces the action $v$ is exactly a scaling of factor $2$, as illustrated in the rightmost subfigures of both rows. However, when we restrict the vector field $v$ to a proportion $p$ of the circle, the factor of the optimal vector field diminishes linearly to $\alpha^* = 2p$. This illustrates that the coupling score effectively measures the proportion to which $v$ behaves like a scaling. This paves the way for future works aimed at recovering a deformation even in cases of partial correspondence.

\section{From static to dynamic}\label{sec:stat_to_dyn}
We remind that $V,W$ denote two RKHS continuously embedded in $C_0^{2}(\mathbb{R}^d, \mathbb{R}^d)$. In the previous subsection, we introduced the coupling score in a static framework, that is computing $C_q(v^{(0)},W)$ for a given $v^{(0)} \in V$ and $W$. Since registration frameworks rely on time-varying vector fields, there is a natural motivation to develop dynamic methods for decoupling deformations over time. To build intuition on how the coupling score can be extended to dynamic registration, this section first introduces an intermediate model designed to clarify the decoupling process.

\subsection{Static decoupling of vector fields }

Given two initial vector fields $v^{(0)}$ and $w^{(0)}$, we seek an optimal decomposition $(v^*,w^*)$ where $v^*$ captures the residual deformation of $v^{(0)}$ that cannot be reproduced by $W$, while $w^*$ captures both the original $w^{(0)}$ and the deformation from $v^{(0)}$ than can be reproduced by $W$. This decoupling process relies on two main principles. First, the decoupling is performed by minimizing the functional $(v,w) \mapsto \dfrac{1}{2} \left(\|v\|_V^2 + \|w\|_W^2 + C_q(v, W)\right)$. Second, we require the optimal vector fields $(v^*,w^*)$ to generate the same first variations as the original fields $(v^{(0)},w^{(0)})$, that is $\delta \mu_q(v^*+w^*) = \delta \mu_q(v^{(0)}+w^{(0)})$ . This constraint is less restrictive than the pointwise equality $(v^* + w^*) \circ q = (v^{(0)} + w^{(0)}) \circ q$  since first variations are invariant to reparameterizations of the curve and provide more flexibility to the decoupling problem while ensuring the total deformation remains preserved. Note that by linearity of $v \mapsto \delta \mu_q(v)$ established in \cref{prop:continuity_first_var}, we have $\delta \mu_q(v+w) = \delta \mu_q(v) + \delta \mu_q(w)$. Consequently, we define the following minimization problem
\begin{equation} \label{eq:stat_dyn}
    (v^*, w^*) = \argmin_{(v,w) \in V \times W} \left\{  \dfrac{1}{2} \left(\|v\|_V^2 + \|w\|_W^2 + C_q(v, W)\right) \mid \delta \mu_q (v) + \delta \mu_q (w)  =  \delta\mu_q (v_0) + \delta \mu_q(w_0) \right\} \, .
\end{equation} 
Because the constraint $\delta \mu_q(v^*) + \delta \mu_q(w^*) = \delta \mu_q(v^{(0)}) + \delta \mu_q(w^{(0)})$ is satisfied, the vector fields $(v^*,w^*)$ reproduce the same first variation as $(v^{(0)},w^{(0)})$. Moreover, $v^*$ is effectively decoupled from $W$, as $(v^*,w^*)$ minimize the energy 
$ E(v,w)=\dfrac{1}{2} \left(\|v\|_V^2 + \|w\|_W^2 + C_q(v, W)\right)
$ which penalizes the coupling between $v$ and $W$.

\begin{figure}[!h]
    \centering
    \begin{subfigure}{\linewidth}
        \centering
        \includegraphics[width=\linewidth]{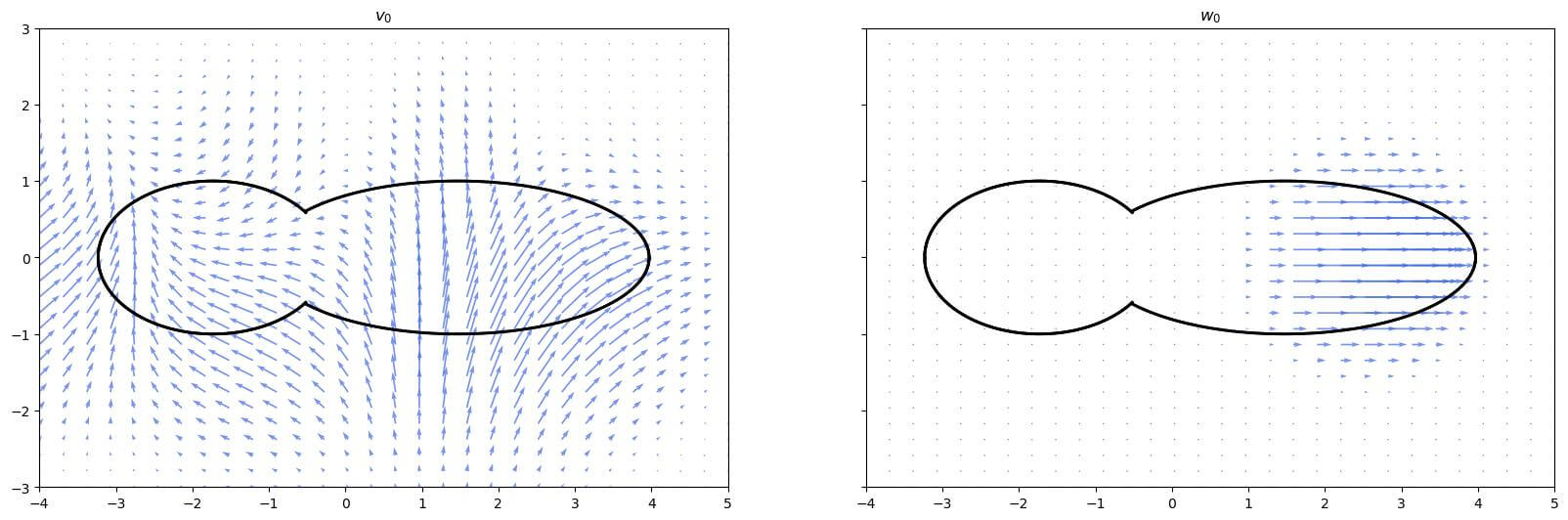}
        \caption{Initial vector fields $v^{(0)}$ and $w^{(0)} : x \mapsto k_{\sigma}(x,z)a$ where $z=[2.5,0]$ is the center of the local translation and $a = [1,0]$ its direction. The kernel $k_{\sigma}$ is a Gaussian kernel with scale $\sigma=1$}
        \label{fig:vfield0}
    \end{subfigure}

    \vspace{1em} 
    
    \begin{subfigure}{\linewidth}
        \centering
        \includegraphics[width=\linewidth]{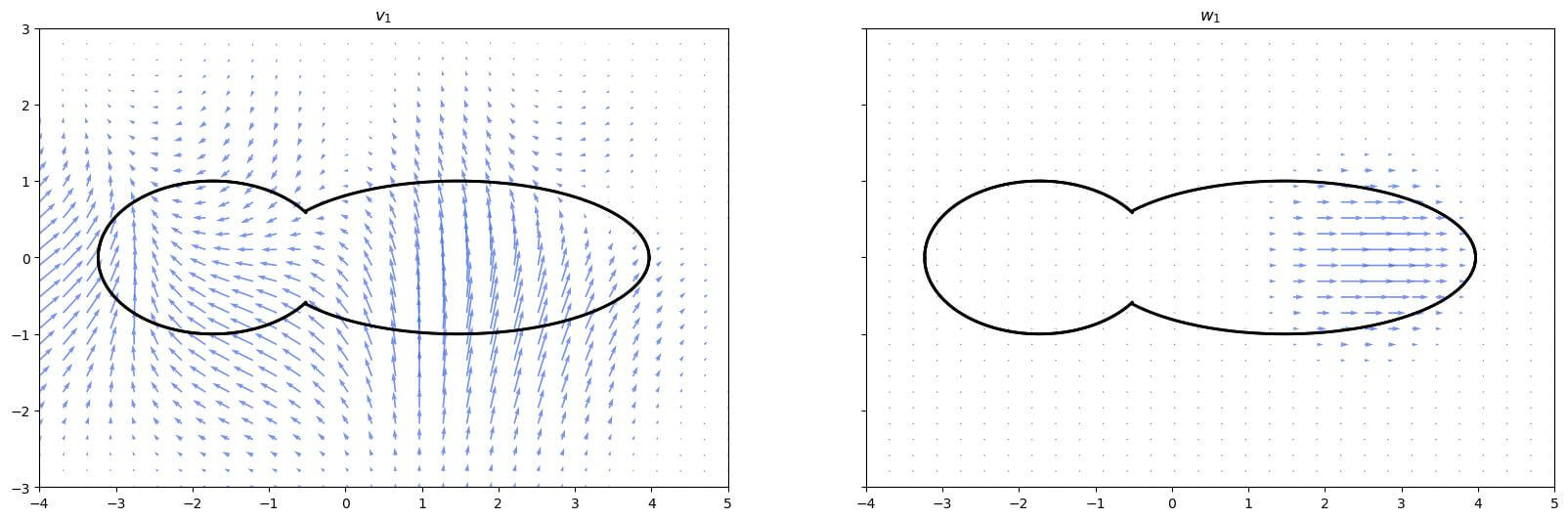}
        \caption{Vector fields $(v^*,w^*)$ obtained after optimization of \cref{eq:stat_dyn} }
        \label{fig:vfield1}
    \end{subfigure}
    \caption{The first row represent the initial vector field $v^{(0)}$ and the vector field $w^{(0)}$ that spans the space of local translations $W$ that we want to decouple from. The second row shows the vector fields $v^*$ and $w^*$ after optimization of \cref{eq:stat_dyn}. The vector field $v$ belongs to a Gaussian RKHS with scale $\sigma_V=1.5$ and the first variation are also computed using a Gaussian kernel with scale $\sigma_H=0.5$.}
\label{fig:vfield_stat_to_dyn}
\end{figure}

\cref{fig:vfield_stat_to_dyn} illustrates this decoupling process with the example of a curve $q$ and a vector field $v^{(0)}$ that we want to decouple from the space of vector fields generating local translations centered in $z=[2.5,0]$ and of direction $a=[-1,0]$ defined by $W_{z,a} = \{ w_{\alpha} : x \mapsto \alpha k_{\sigma}(z,x)a  \mid \alpha \in \R \} $ with $w_0 : x \mapsto k(z,x) a$ and $k$ a Gaussian kernel of scale $\sigma=1$. This is an example of a deformation module \cite{gris_module}. The optimal solution of \cref{eq:stat_dyn}, represented in \cref{fig:vfield1}, shows that the initial vector field $v^{(0)}$ has been locally projected onto the orthogonal complement of $a=[1,0]$. This results in a vector field $v^*$ that is decoupled from $W$, i.e. vertical on the right side of the shape, while remaining unchanged elsewhere.

\subsection{Constrained registration problem}\label{sec:cons_reg_prob}

Let $I=[0,1]$ or $ \mathbb{S}^{1}$. Recall that $Q =\{ q \in C^{1}(I,\mathbb{R}^d) \mid \Vert q'(s) \Vert \ne 0 \}$ is the space of curves already defined in \cref{def:curves}. In particular, $Q$ is a shape space in the sense of Arguillere \cite{arguillere2014shapedeformationanalysisoptimal} (see \cite[Section 3.2]{mouhli2025} for more details). Let $q_S \in Q$ be a template curve. 

The motivation for the coupling score introduced earlier is its application to registration problems. A standard model to perform registration is the Large Deformation Diffeomorphic Metric Mapping (LDDMM) framework \cite{LDDMM}, which we briefly recall hereafter. Given a source $q_S$ and a target curve $q_T$, the registration problem consists of finding an optimal diffeomorphism that matches both curves. This search is reformulated as finding an optimal time-varying vector field $v \in L^2([0,1],V)$, where $V$ is a fixed RKHS, that generates a flow of diffeomorphisms via the ODE:
\begin{equation*}
       \left\{
    \begin{array}{ll}
        \dot{\varphi}_t &=  v_t \circ \varphi_t \\
        \varphi_0 &= \id
    \end{array} 
    \right.\notag
\end{equation*} 
Under this flow, the trajectory of the deformed curve $q_t= \varphi_t \circ q_S$ is derived from the infinitesimal action of vector fields $\dot{q}_t= v_t \circ q_t$.
In particular, the deformed curve $t \mapsto q_t$ is absolutely continuous and its associated varifold trajectory $t \mapsto \mu_{q_t}$ is also absolutely continuous as stated in the next proposition. We recall that for a given Banach space $E$, we denote by $AC^2([a,b], E)$ the space of continuous curves $f \colon [a,b] \to E$ for which there exists a function $F \in L^2([a,b], E)$ satisfying $f(t) = f(a) + \int_a^t F(s) \, ds$. This is equivalent to stating that $f$ is differentiable almost everywhere with a derivative $f' \in L^2([a,b], E)$.

\begin{proposition}[$AC^2$ regularity of varifolds]\label{prop:evol_eq_first_var} 
    Let $V \hookrightarrow C_0^2(\R^d,\R^d)$ be a RKHS of vector fields and $q_S \in Q$.  For $v \in L^2([0,1],V)$, let denote $t \mapsto q_t$ the trajectory of curve generated by $\dot{q}_t=v_t\circ q_t$, $q_0=q_S$. The associated varifold trajectory $t \mapsto \tilde{\mu}_{q_t}$ belongs to $AC^2([0,1],C^2_0(\mathbb{R}^d \times \mathbb{S}^{d-1},\mathbb{R})')$.
\end{proposition}

\begin{proof}
    Since the mapping $t \mapsto q_t$ belongs to $AC^2([0,1],Q)$ \cite[Theorem 3.5]{pierron2024} and $q \mapsto \tilde{\mu}_q$ has $C^1$-regularity by \cref{prop:frechet_varifold}, this is a direct consequence of \cite[Lemma 3.18]{glockner2015}
\end{proof}

The registration problem can be expressed as the following energy minimization problem:
\begin{align}\label{eq:opt_cont}
   \min_{v \in L^2([0,1],V)} J(v)&=\int_0^1 \dfrac{1}{2} \Vert v_t \Vert_V^2 dt+\mathcal{D}(q_1) \\
    \text{ s.t } & \quad      \left\{
    \begin{array}{ll}
        \dot{q}_t &=  v_t \circ q_t \\
        q_0&=q_S
    \end{array} 
    \right.\notag
\end{align}
where $\mathcal{D} : Q \rightarrow \R$ is a data attachment term measuring the discrepancy between the final deformed curve $q_1=\varphi_1\circ q_S$ and the target $q_T$. The case where an exact constraint $C : Q \to \mathcal{L}(V,\R)$ is enforced in this minimization problem, that is $C_{q_t}(v_t)=0$, has been studied in \cite{arguillere2014shapedeformationanalysisoptimal}. 

This framework can be extended to the combination of multiple types of deformation by considering different spaces of vector fields. We will particularly focus on the example of a two types of deformation model, whose associated registration problem is
\begin{align}\label{eq:LDDMM_2_vfield}
    \inf_{v \times w \in L^2([0,1],V \times W)} J_1(v,w) &= \dfrac{1}{2}\int_0^1  \Vert v_t \Vert_V^2 + \Vert w_t \Vert_W^2 \, dt + \mathcal{D}(q_1) \\
    \text{s.t. } & \quad \begin{cases}
        \dot{q}_t = v_t \circ q_t + w_t \circ q_t \\
        q_0 = q_S
    \end{cases} \notag
\end{align}

To decouple $v$ from the space of vector fields $W$, the coupling score will be added as a penalization in the functional $J_1$. The goal of this decoupling is to obtain more accurate statistics on the contribution of each mode of deformation in the total deformation and a better final matching. This defines a new minimization problem:
\begin{equation} \label{eq:min_pb_constrained}
\inf_{(v, w) \in L^2([0,1],V \times W)} J_2(v,w) =\dfrac{1}{2} \int_0^1   \| v_t \|_V^2 + \| w_t \|_W^2 + C_{q_t}(v_t,W) \, dt + \mathcal{D}(q_1)
\end{equation}
subject to the evolution equation $\dot{q}_t = (v_t + w_t) \circ q_t$ and the initial condition $q_0 = q_S$. 

This approach can be generalized to spaces of vector fields that depend dynamically on the state of the curve at time $t$, as in the deformation modular framework \cite{gris_module}. In this setting, the curve is deformed via the evolution equation $\dot{q}_t = (v_t + w_t) \circ q_t$, where $w_t \in W_{q_t}$, and the coupling score becomes $C_{q_t}(v_t, W_{q_t})$. However, in this paper, we restrict our study to the standard setting with static vector field spaces.

\begin{proposition} \label{prop:existence_minimizer}
Assume that the data attachment term $q \in Q \mapsto \mathcal{D}(q)$ is non-negative and lower semi-continuous. Then, $J_2$ admits at least one global minimizer $(v^*, w^*) \in L^2([0,1], V \times W)$.
\end{proposition}

\begin{proof}
The proof is given in \cref{app:proof_existence_minimizer}.
\end{proof}

The following proposition establishes a connection between the dynamic formulation \cref{eq:min_pb_constrained} and its static counterpart \cref{eq:stat_dyn}.

\begin{proposition}[Relation between minimizers in static and dynamic setting]\label{prop:static_to_dynamic} Assume that the data attachment $\mathcal{D}$ is invariant by reparameterization.
    Let $(v^*, w^*) \in L^2([0,1], V \times W)$ be a global minimizer of $J_2$, and let $q^*$ be the associated optimal trajectory starting from $q_0 = q_S$. For almost every $t \in [0,1]$, the pair $(v_t^*, w_t^*)$ is a solution to the static minimization problem:
    \begin{equation*} 
        \min_{(v,w) \in V \times W} \dfrac{1}{2} \left( \|v\|_V^2 + \|w\|_W^2 + C_{q_t^*}(v, W) \right)
    \end{equation*}
    subject to the constraint on the first variations:
    \begin{equation*}
        \delta \mu_{q_t^*} (v) + \delta \mu_{q_t^*} (w) = \delta \mu_{q_t^*} (v_t^*) + \delta \mu_{q_t^*} (w_t^*).
    \end{equation*}
\end{proposition}

\begin{proof}
    The proof is given in \cref{app:proof_static_to_dynamic}.
\end{proof}

\section{Experiments}\label{sec:experiments}

This section presents various numerical experiments demonstrating the utility of the coupling score to overcome specific issues in registration problems. Specifically, we illustrate how it can be used to prevent or enforce specific types of deformations, as well as to iteratively correct a matching. Depending on the experiment, the registration is performed using either a mixture of rotation and unstructured deformations (\cref{sec:exp_rot}) or unstructured deformations only (\cref{sec:exp_match_pop} and \cref{sec:exp_reg_it}). The different models are based on the standard large deformation framework \cite{LDDMM} which are briefly recalled in \cref{sec:cons_reg_prob}. 

To perform the registration while achieving decoupling, we opt for a geodesic shooting strategy. A brief reminder about geodesic shooting in the large deformation framework is provided in \cref{sec:geodesic_shooting}. We present the strategy for a standard single-deformation model, noting that the approach for the two-mode deformation model is analogous. We first consider the standard registration problem given by
\begin{align}\label{eq:opt_contt}
   \min_{v \in L^2([0,1],V)} J(v)&=\int_0^1 \dfrac{1}{2} \Vert v_t \Vert_V^2 dt+\mathcal{D}(q_1) \\
    \text{ s.t } & \quad      \left\{
    \begin{array}{ll}
        \dot{q}_t &=  v_t \circ q_t \\
        q_0&=q_S.
    \end{array} 
    \right.\notag
\end{align}
Denoting $\xi_q(v) = v \circ q$ as the infinitesimal action of $v$ on $q$, the Pontryagin Maximum Principle (PMP) \cite{100years} yields the following geodesic equations:
\begin{equation*}
    \begin{cases}
        \dot{q}_t &= v_t \circ q_t \\
        \dot{p}_t &= - (\partial_q \xi_{q_t}(v_t))^* p_t \\
        v_t &= K_V \xi_{q_t}^* p_t.
    \end{cases}
\end{equation*}
Consequently, the minimization problem \eqref{eq:opt_contt} can be expressed as an optimization over the initial covector $p_0$:
\begin{align}
\label{eq:geod_shoot_unconstrained}
   \min_{p_0 \in T_{q_0}^*Q} E_0(p_0) &= \dfrac{1}{2} \Vert v_0 \Vert_V^2 + \mathcal{D}(q_1) \\
   \text{s.t.} \quad & 
    \begin{cases}
        \dot{q}_t &= v_t \circ q_t \\
        \dot{p}_t &= - (\partial_q \xi_{q_t}(v_t))^* p_t \\
        v_t &= K_V \xi_{q_t}^* p_t
    \end{cases} \notag 
\end{align}
Finally, to decouple the deformation generated by $v$ from another subspace of vector fields $W$, we add the coupling score to the objective functional, leading to the following constrained minimization problem:
\begin{align}
\label{eq:geod_shoot_constrained}
   \min_{p_0 \in T_{q_0}^*Q} E_1(p_0) &= \dfrac{1}{2} \Vert  v_0 \Vert_V^2 + C_{q_0}( v_0 ,W) + \mathcal{D}(q_1) \\
   \text{s.t.} \quad & 
    \begin{cases}
        \dot{q}_t &= v_t \circ q_t \\
        \dot{p}_t &= - (\partial_q \xi_{q_t}(v_t))^* p_t \\
        v_t &= K_V \xi_{q_t}^* p_t
    \end{cases} \notag 
\end{align}

In summary, our optimization strategy relies on geodesic equations derived from the standard problem \eqref{eq:opt_contt}, which leads to the geodesic shooting framework \eqref{eq:geod_shoot_unconstrained}. The coupling score is then introduced within this setting as a penalty evaluated at $t=0$, leading to the minimization problem  \eqref{eq:geod_shoot_constrained}.
Introducing the penalty in \eqref{eq:geod_shoot_unconstrained}, rather than in the initial minimization problem \eqref{eq:opt_contt}, provides significant flexibility. It allows us to preserve the exact same Hamiltonian dynamics associated with the standard registration, simply incorporating this score into the objective function. If the coupling score were instead included in \eqref{eq:opt_contt}, the PMP would yield Hamiltonian equations dependent on this score, requiring the derivation of a new constrained Hamiltonian. Consequently, changing the base registration model would entail the explicit derivation of the new underlying constrained geodesic equations. In contrast, our approach allows us to directly leverage the existing geodesic equations associated with any given model.
While evaluating the constraint only at the initial time $t=0$ keeps the framework simple, our numerical experiments will demonstrate that this initial penalty is sufficient to achieve the desired decoupling.

Developing constrained registration frameworks adapted for decoupling, for instance, where the geodesic equations directly integrate the coupling score following the approach in \cite[Section 3]{arguillere2014shapedeformationanalysisoptimal}, or models with distinct covectors representing the different modes of deformation, falls beyond the scope of this paper and is left for future work.

In the following experiments, the coupling score is computed using a varifold RKHS generated by a Gaussian kernel $k_{\sigma_c}(x,y)=\exp({-\dfrac{\Vert y - x \Vert^2}{\sigma_c^2}})$ on positions only. For the data attachment term $\mathcal{D}(q_1)$, we consider a varifold metric which is the combination of a Gaussian kernel on positions and a tangential kernel $k_{\vec{t}}(\vec{t}_x,\vec{t}_y) = \langle \vec{t}_x,\vec{t}_y\rangle^2$.

All experiments are reproducible, and the source code is publicly available on GitHub\footnote{\url{https://github.com/rayanemouhli/varifold-score-decoupling-deformations}}.

\subsection{Decoupling from rotations}\label{sec:exp_rot}

This first example illustrates the application of the coupling score within a model combining two types of deformations. The score is used to prevent that the actions of both vector fields on the curve overlap. The objective is to get a better matching and more accurate statistics on the contribution of each deformation mode.
To illustrate this use case, we consider the matching of a source star with a target star defined as the source star that has been rotated with a stretched branch as shown in \cref{fig:source_target_etoiles}.
\begin{figure}[!h]
    \centering
    \includegraphics[width=0.6\linewidth]{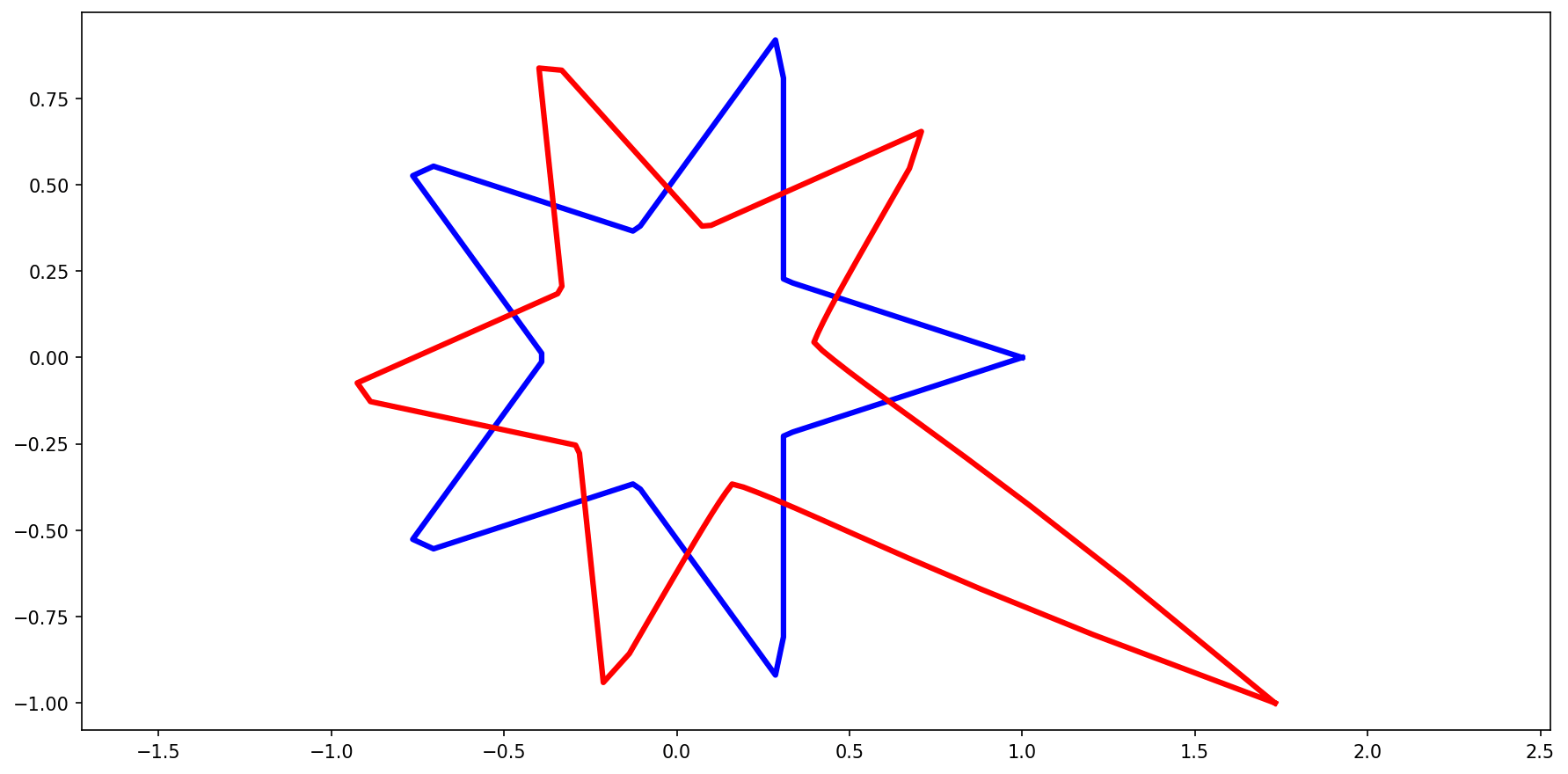}
    \caption{The source and target curves are respectively represented in blue and red. The target is the rotated source with a stretched branch.}
    \label{fig:source_target_etoiles}
\end{figure}

To perform this matching, we first consider a registration model, without coupling score, presented in \cite[Section 3.3]{mouhli2025} where  $\operatorname{SO}_d \ltimes \Diff_{C^k_0}(\R^d)$ is the group of deformations consisting in rotations and diffeomorphisms. The associated space of vector fields is $\operatorname{Skew}_d \oplus V$ where $V$ is a Gaussian RKHS with scale $\sigma_V=0.5$. The registration can be expressed as the following minimization problem
\begin{eqnarray}
    \label{energy_isom}
    \inf_{(A,v) \in L^2([0,1],\operatorname{Skew}_d \oplus V)} && \dfrac{1}{2} \int_0^1  \Vert A_t \Vert^2 +  \Vert v_t \Vert^2_V \, dt + \mathcal{D}(q_1). \\
     \text{ s.t }& &     \left\{
        \begin{array}{l}
            \dot{R}_t = A_t R_t \\
            \dot{q}_t = v_t \circ q_t + A_t q_t  \\
            (R_0,q_0) = (I_d,q_S)        
        \end{array} \notag
        \right.
\end{eqnarray}

The optimization is performed using geodesic shooting as presented in \cref{sec:geodesic_shooting}. While \cite{mouhli2025} proposes a method to decouple rotations and diffeomorphisms using reduction theory, in this experiment we rather perform decoupling by using the coupling score. Our proposed method has the advantage of being applicable to any type of deformation, contrary to the cited article which is designed for rotations and translations only.

\begin{figure}[!h]
    \centering
    \includegraphics[width=\linewidth]{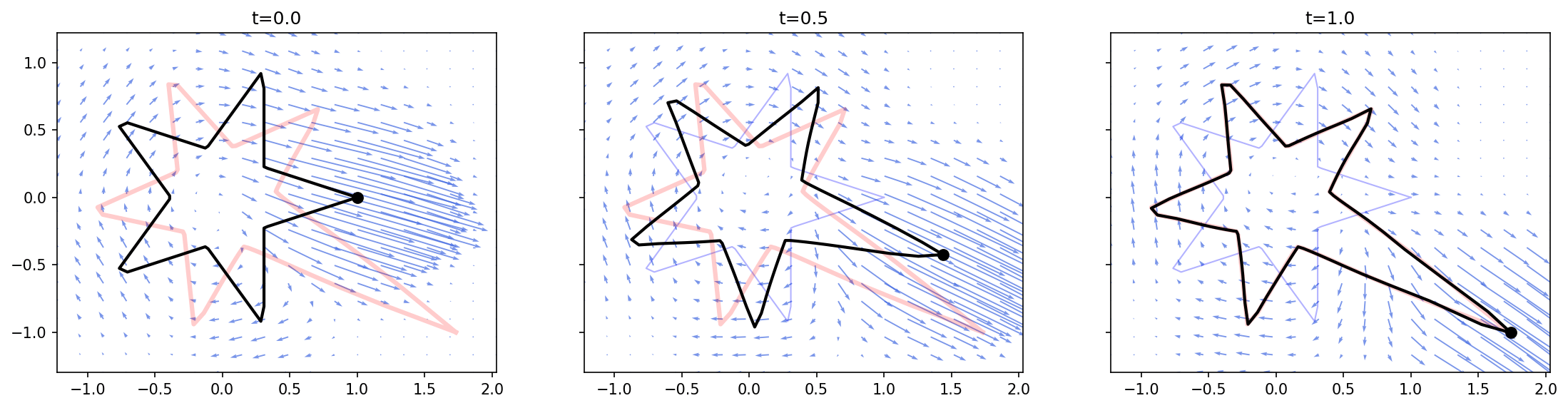}
    \caption{Geodesic associated with the first registration after optimization. The template, target and deformed curves are represented in blue, red and black respectively. The black dot serves only for illustration purposes and does not participate in the registration. The template is deformed by both rotation and diffeomorphism, although the blue arrows only represent the vector field associated with the diffeomorphism part. The RKHS of vector fields $V$ is generated by a Gaussian kernel with scale $\sigma_V=0.5$. The Gaussian kernel associated with the varifold data attachment term has a scale $\sigma=0.5$. }
    \label{fig:run_no_corr_etoiles}
\end{figure}

The objective of this registration is to ensure that the unstructured component only performs the stretching of the branch, while the rotation component handles the global rotation of the star. Even though the source is well-matched to the target, the result of the minimization problem \eqref{energy_isom} presented in \cref{fig:run_no_corr_etoiles} shows that the vector field $v$ associated with the unstructured diffeomorphism (in blue) contributes strongly to the rotation of the curve. To overcome this issue, we introduce the coupling score at initial time $t=0$, as a penalty term in the minimization problem \eqref{energy_isom}, to prevent the vector field $v$ from acting as a rotation.

Therefore, we define a new minimization problem:
\begin{eqnarray}
    \label{energy_isom2}
    \inf_{(A,v) \in L^2([0,1],\operatorname{Skew}_d \oplus V)} && \dfrac{1}{2} \int_0^1 \Vert A_t \Vert^2 + \Vert v_t \Vert^2_V \, dt +  C_{q_0}(v_0,\operatorname{Skew}_d) + \mathcal{D}(q_1) \\
     \text{ s.t }& &     \left\{
        \begin{array}{l}
            \dot{R}_t = A_t R_t \\
            \dot{q}_t =  v_t \circ q_t + A_t q_t \\
            (R_0,q_0) = (I_d,q_S)        
        \end{array} 
        \right.\notag
\end{eqnarray}

\begin{figure}[!h]
    \centering
    \includegraphics[width=\linewidth]{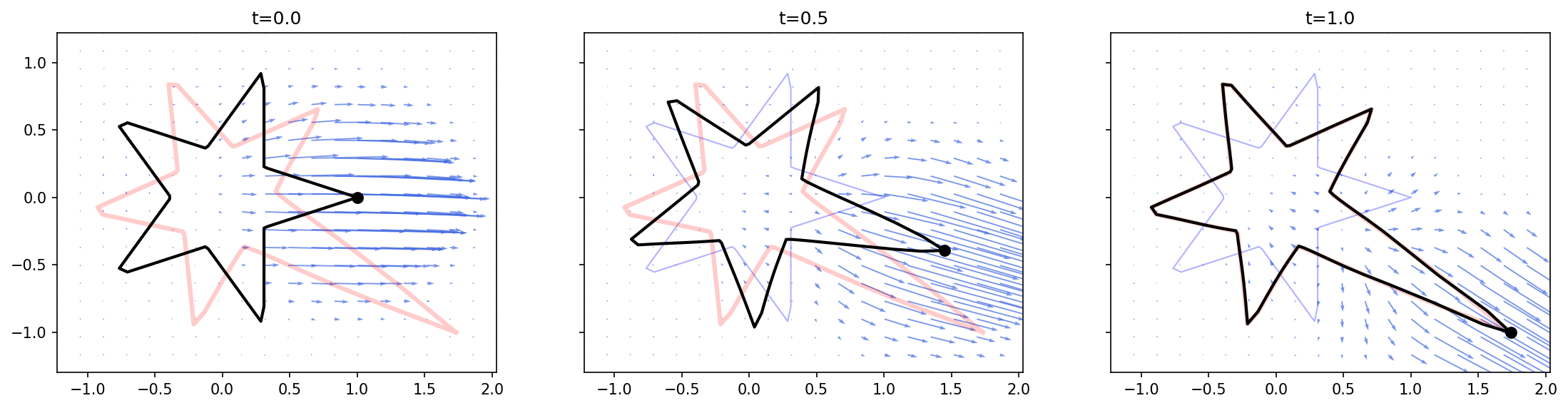}
    \caption{Geodesic associated with the second registration after optimization. The parameters are the same as in \cref{fig:run_no_corr_etoiles}. Contrary to the previous experiment, the unstructured vector field $v$ (in blue) does not perform rotation anymore and contributes only to the stretching of the branch.}
    \label{fig:run_corr_etoiles}
\end{figure}

 \cref{fig:run_corr_etoiles} shows the geodesic of the new matching after optimization of \eqref{energy_isom2}. Minimizing the coupling score successfully constrains the vector field $v$ to be decoupled from the space of rotation vector fields $\operatorname{Skew}_d$, ensuring it performs only the stretching of the branch.

\FloatBarrier
\subsection{Matching of a population}\label{sec:exp_match_pop}
The purpose of this experiment is to mimic the matching of a population of leaves. In this context, the source curve is a common template that must be matched to various targets. In the previous example, the coupling score was used to prevent the unstructured vector field $v$ from acting as a rotation. Here, we proceed to the opposite, that is enforcing the unstructured vector field to mimic the action of a specific space of vector fields. 
\begin{figure}[!h]
    \centering
    \includegraphics[width=\linewidth]{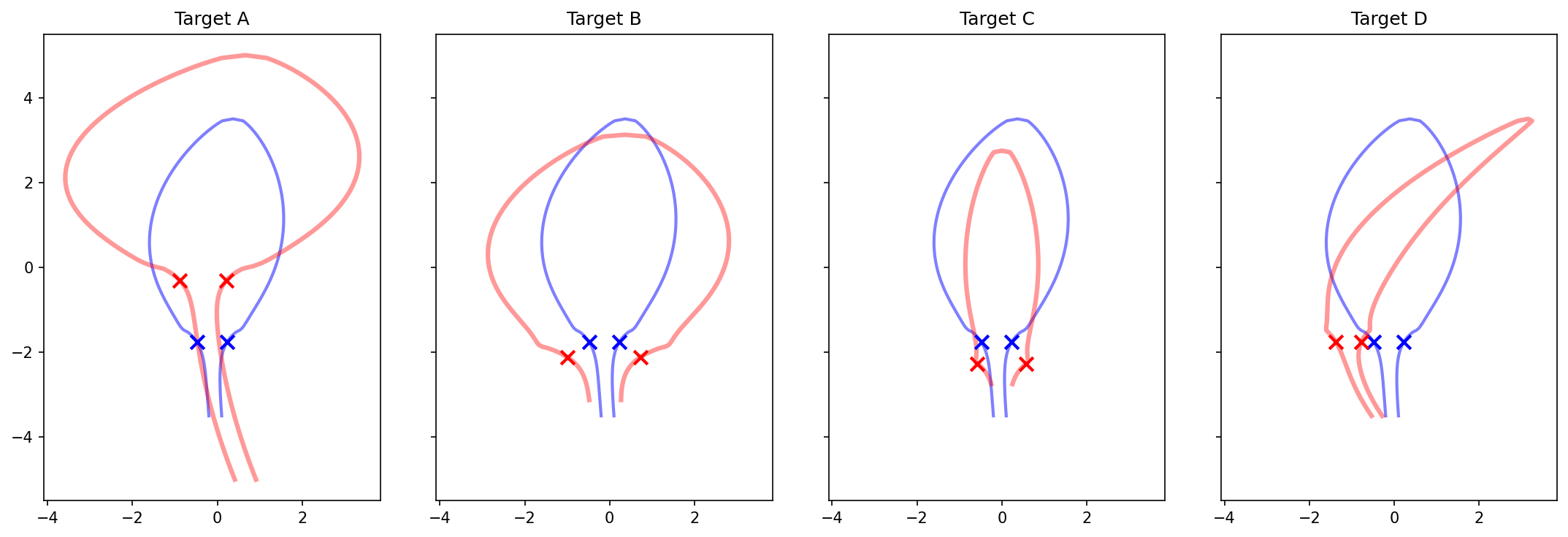}
    \caption{The template and the targets are respectively represented in blue and red. The crosses represent the boundary between the rod and the leaf. They serve only for illustration purposes and do not contribute to the registration.}
    \label{fig:source_target_leaves}
\end{figure}
We generated a toy dataset, presented in \cref{fig:source_target_leaves}, consisting of a single template (blue) and four targets (red). The crosses represent the junctions between the rod and the leaf. They serve only for illustrative purposes in the following experiments and do not contribute to the registration process since our model aims to bypass the need for prior segmentation.

We first perform four independent registrations mapping the template $q_S$ to each of the four targets $q_X \in \{q_A,q_B,q_C,q_D\}$, using identical parameters. For each target, we solve a distinct optimization problem driven by a data attachment term $\mathcal{D}_X$ measuring the discrepancy with $q_X$. For a given target $q_X$, the registration is formulated as the following minimization problem:
\begin{align}\label{eq:standard_model}
   \min_{v \in L^2([0,1],V)} & \int_0^1 \dfrac{1}{2} \Vert v_t \Vert_V^2 \, dt + \mathcal{D}_X(q_1) \\
    \text{s.t. } & \quad \begin{cases}
        \dot{q}_t &=  v_t \circ q_t \\
        q_0 &= q_S
    \end{cases} \notag
\end{align}

\begin{figure}[!h]
    \centering
    \includegraphics[width=\linewidth]{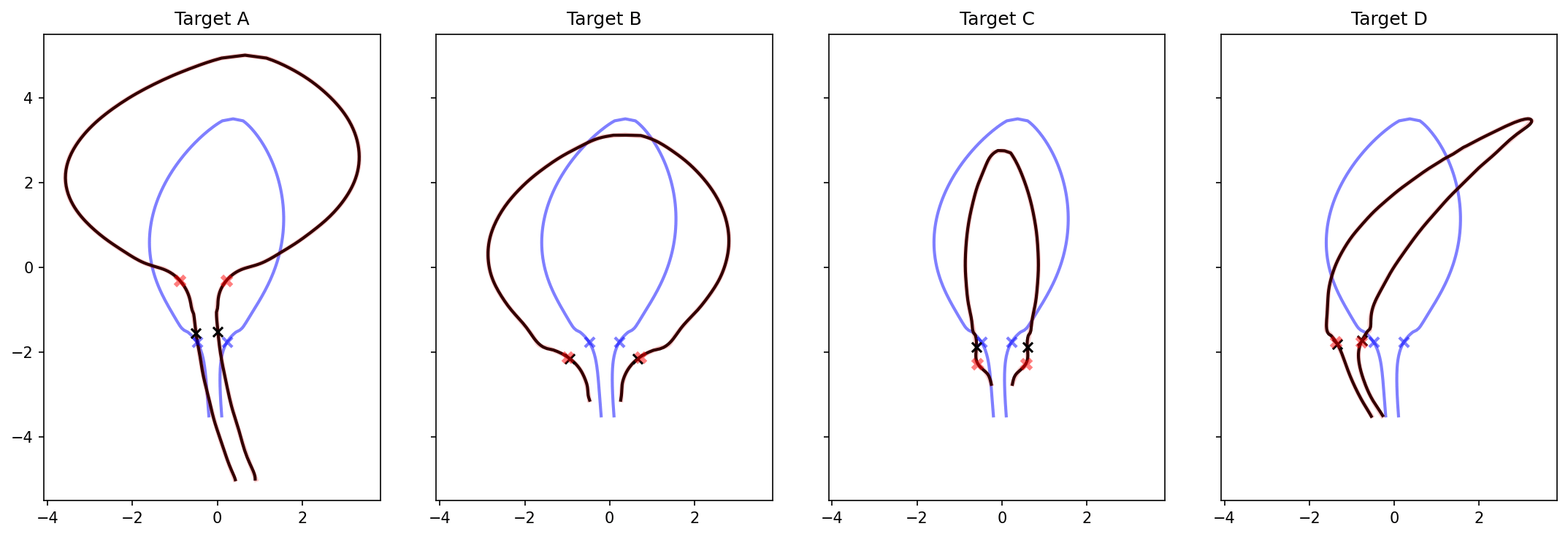}
    \caption{Results of the minimization problem \eqref{eq:standard_model} after optimization. The template, targets and deformed curves are respectively represented in blue, red and black. The RKHS of vector fields $V$ is induced by a sum of Gaussian kernels with scale $\sigma_V \in \{0.5,2\}$. The kernel associated with the varifold data attachment term is also a sum of Gaussian kernels with scale $\sigma \in \{0.5,2\}$. }
    \label{fig:run1_leaves}
\end{figure}

The result of the matching after optimization is presented in \cref{fig:run1_leaves}. The matching of targets B and D is successful, as the curves are correctly aligned and the rod/leaf segmentation is respected. However, the matching of targets A and C is unsuccessful due to a failure in transporting the segmentation. Indeed, the geodesic associated with target A transforms the bottom part of the leaf into the top of the rod, and the geodesic associated with target C transforms the top of the rod into the bottom of the leaf.  These registrations fail because the standard model \eqref{eq:standard_model} finds it cheaper to shrink or stretch parts of the leave into one another rather than performing a translation of the junction. 

We now assume that the position of the junction of the template is known, but the positions of the junctions on the multiple targets are not. To correct this segmentation issue, we introduce a new minimization problem where the vector field $v$ is constrained to act as a translation in a localized area around the junction. This constraint forces the junction to translate towards the bottom or the top as needed. We define the space of vector fields representing the phenomenon we wish to enforce by $W_0 = \operatorname{span}(w_0)$ with $w_0 = u \mathds{1}_D $, where $D$ is an area around the junction of the template that we chose once and for every targets, and $u$ is a vertical translation. It is worth noting that the exact definition of the region $D$ is not a sensitive parameter as long as it encompasses the relevant portion of the curve we are interested in. Indeed, because the optimization functional depends on $w_0$ exclusively through its induced first variation (i.e., the evaluation of the vector field along the curve), the precise geometry of $D$ has a negligible influence on the registration. The coupling score between $v_0$ and the one-dimensional space $W_0$ is given by $C_{q_0}(v_0,W_0)=\vert \alpha ^* \vert^2$ where $u^*=\alpha^*u \mathds{1}_D$ represents the optimal local translation that best mimics the action of $v_0$ on the curve. To ensure that the translation occurs in the desired direction, rather than directly using the coupling score $C_{q_0}(v_0,W_0)=\vert \alpha ^* \vert^2$, we instead employ the oriented score $\max(0,\alpha^*)$. This oriented score is chosen to account for the directionality of the deformations. Depending on the target, the leaf needs to be translated either upwards or downwards. For example, registering to target C requires a downward translation, as this target represents an extreme case where the rod is almost entirely absent. Conversely, for target A, mapping the bottom of the leaf onto the top of the rod requires the template's junction to be translated upwards.  Then, the new matching problem is

\begin{align}\label{eq:pb_leaf_corr}
   \min_{v \in L^2([0,1],V)} & \int_0^1 \dfrac{1}{2} \Vert v_t \Vert_V^2 \, dt + \dfrac{1}{1+\max(0,\alpha^*)} + \mathcal{D}_X(q_1) \\
    \text{s.t. } & \quad \begin{cases}
        \dot{q}_t &= v_t \circ q_t \\
        q_0 &= q_S
    \end{cases} \notag
\end{align}

Since the coupling score $C_{q_t}(v_0,W_0)$ aims to prevent $v_0$ from acting like elements of $W_0$, by minimizing its inverse we enforce the desired behavior.

\begin{figure}[!h]
    \centering
    \begin{subfigure}[t]{0.45\textwidth}
        \vspace{0pt}
        \centering
        \includegraphics[width=\linewidth]{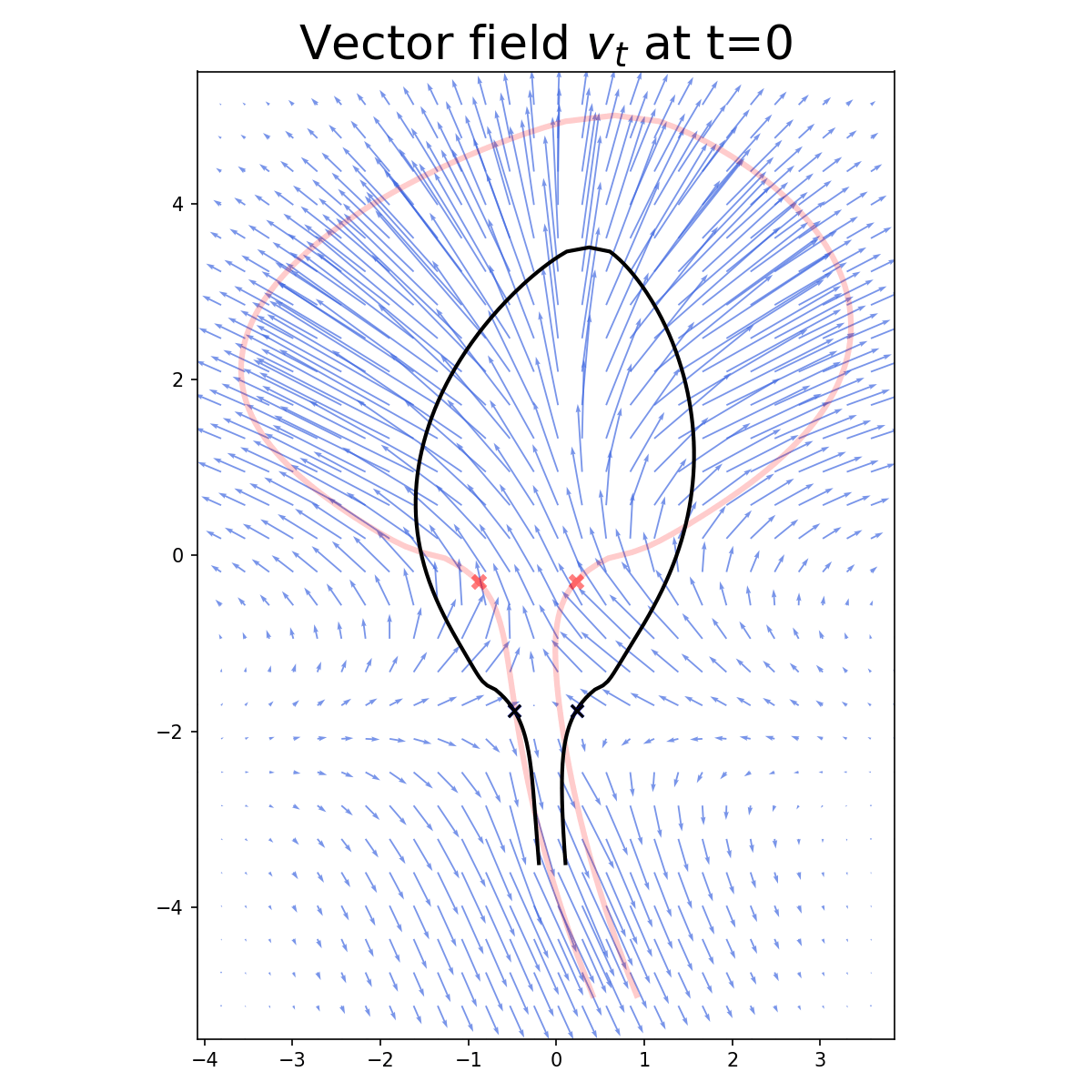}
        \caption{Vector field $v_{t=0}$ associated with the first registration (witout coupling score) after optimization.}
        \label{fig:vfield_first_run}
    \end{subfigure}
    \hspace{0.5cm} 
    \begin{subfigure}[t]{0.45\textwidth}
        \vspace{0pt}
        \centering
        \includegraphics[width=\linewidth]{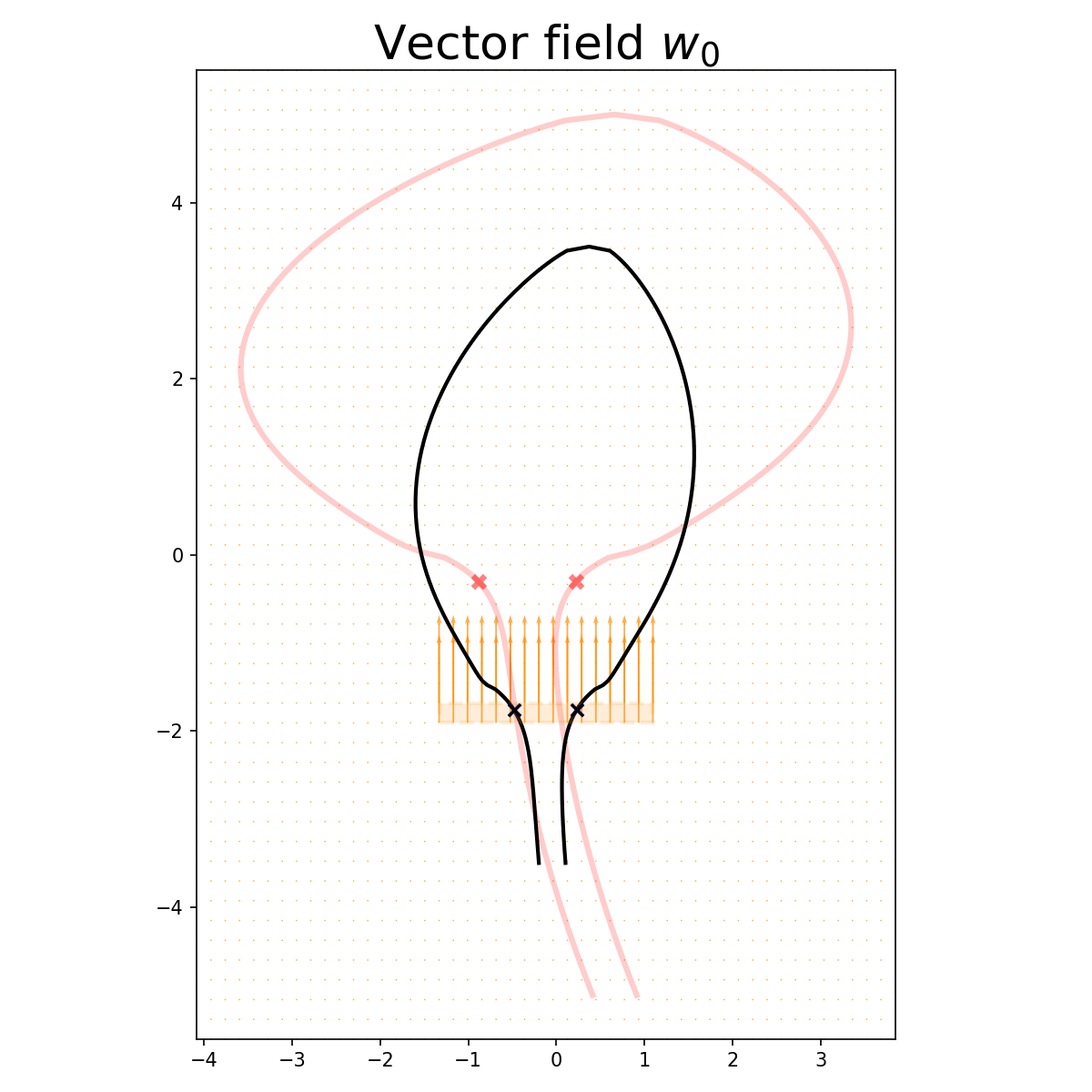}   
        \caption{The area $D$ localized around the template's junction is represented by an orange box and the vector field $w_0 = [0,1]\mathds{1}_D$ is represented by orange arrows.}
        \label{fig:wfield}
    \end{subfigure}
    
    \vspace{0.5cm} 
    
    \begin{subfigure}[t]{0.45\textwidth}
        \vspace{0pt} 
        \centering
        \includegraphics[width=\linewidth]{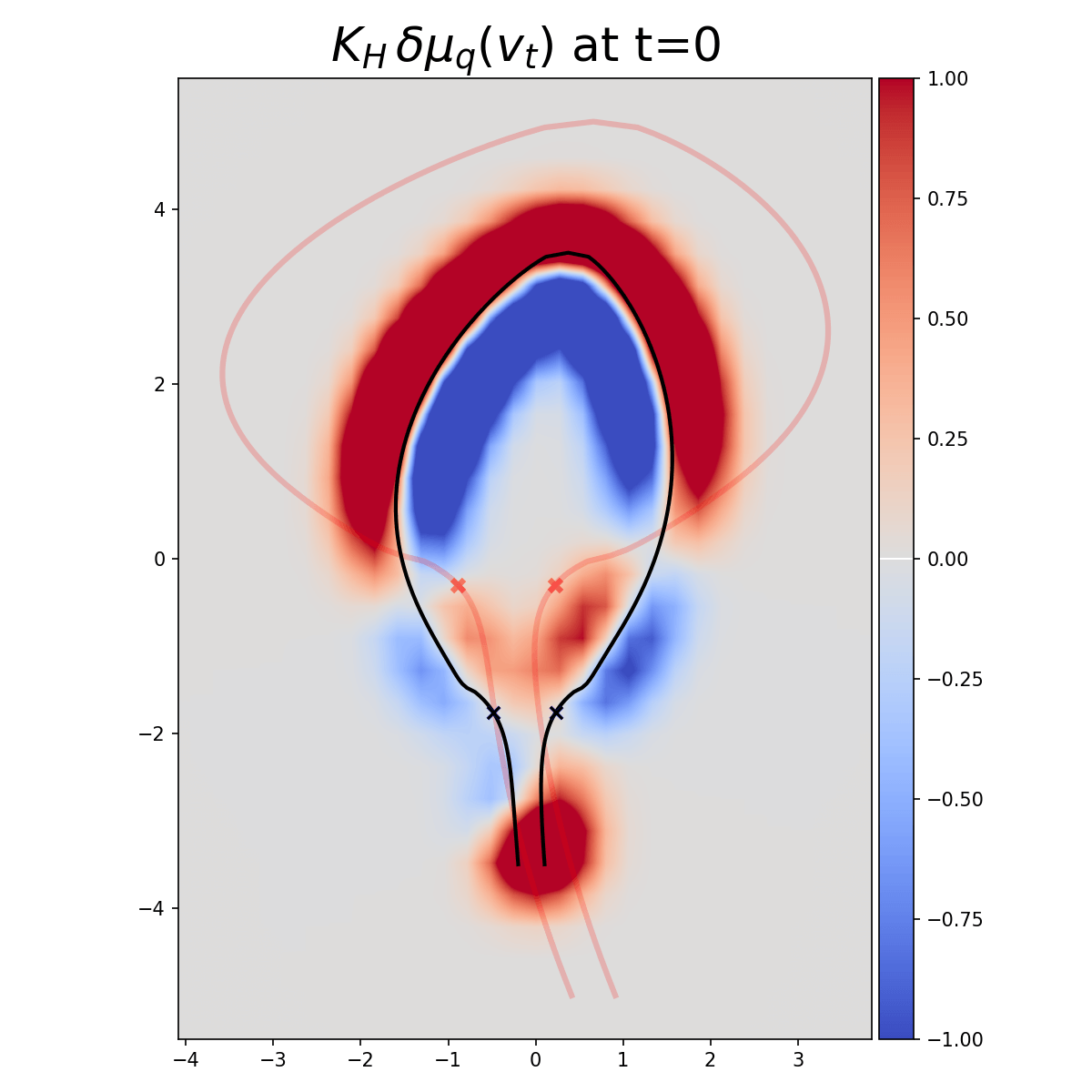}  
        \caption{Representation of the first variation of varifold associated with the template curve and induced by the vector field $v_{t=0}$.}
        \label{fig:heatmap_vfield_corr_leafA}
    \end{subfigure}
    \hspace{0.5cm}
    \begin{subfigure}[t]{0.45\textwidth}
        \vspace{0pt} 
        \centering
        \includegraphics[width=\linewidth]{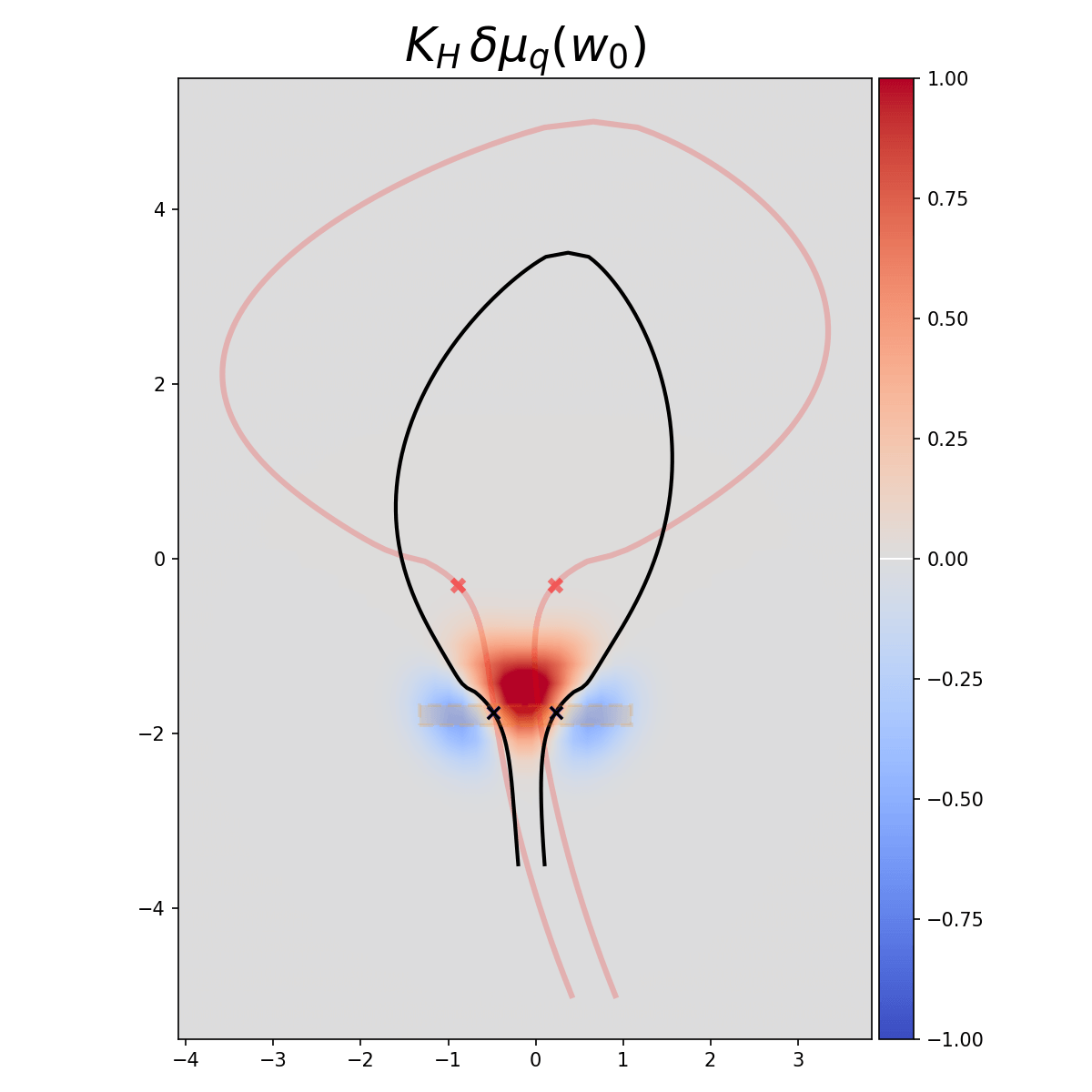}    
        \caption{Representation of the first variation of varifold associated with the template curve and induced by the vector field $w_0$. The area $D$ localized around the template's junction is represented by an orange box.}
        \label{fig:heatmap_w_leafA}
    \end{subfigure}
    \caption{These figures are associated with the matching to target A. The vector fields $v_{t=0}$ and $w_0$ are represented alongside their induced first variation of varifolds associated with the template curve $q_S$. The target is represented in red, and the deformed curve at $t=0$ (which corresponds to the template) is in black. The first variations are represented using a Gaussian kernel with scale $\sigma_H = 0.5$}
    \label{fig:vfield_heatmap_A}
\end{figure}

 \cref{fig:vfield_heatmap_A} shows illustrations associated with the registration to target A. The left column displays the initial vector field $v_0$ obtained after the first unconstrained optimization \eqref{eq:standard_model}, together with the representation of its first variation associated with the template curve. We recall that the image representation of the first variation is given by
    \begin{equation*} 
    K_{H}\delta \mu_q(v) : x \in \R^d \mapsto  \sum_{f \in F_q} \ell_f (\delta \ell_f)_v ~ k_E(c_f, x) + \ell_f   \nabla_1 k_E(c_f,x)^{\top} v_f  
    \end{equation*} 
where $\ell_f = \Vert q_{f^2}- q_{f^1} \Vert$ and $(\delta \ell_f)_v = \langle \dfrac{v(q_{f^2})-v(q_{f^1})}{\ell_f},\dfrac{q_{f^2}-q_{f^1}}{\ell_f} \rangle$. 
\cref{fig:heatmap_vfield_corr_leafA} shows that, just above the template's junction the bottom of the leaf collapses to transform into a part of the rod. The right column represents the vector field $w_0=[0,1] \mathds{1}_D$ which defines the space $W_0$ in the coupling score. The high intensity observed in the first variation representation in \cref{fig:heatmap_w_leafA} is strongly concentrated at the top of the junction, corresponding to the translation we wish to reproduce.

\begin{figure}[!h]
    \centering
    \includegraphics[width=\linewidth]{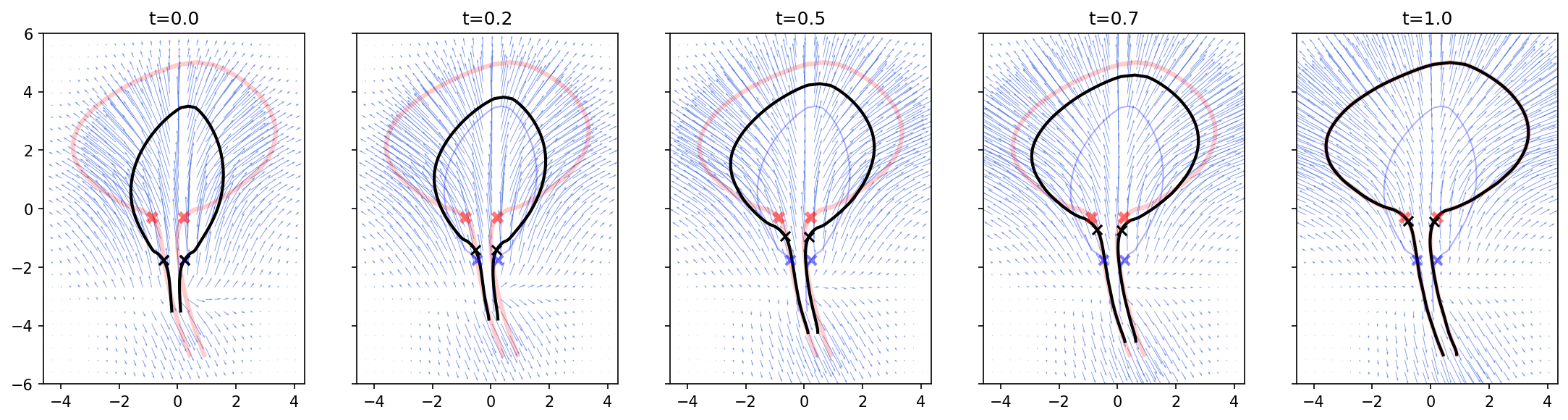}
    \caption{Geodesic associated with the target A after optimization of the constrained registration given in \eqref{eq:pb_leaf_corr}.}
    \label{fig:run_leafA_corrected}
\end{figure}

The geodesic associated with the constrained problem \eqref{eq:pb_leaf_corr} after optimization is presented in \cref{fig:run_leafA_corrected}. First, the addition of the inverse of the coupling score to the functional did not alter the global matching, which remains successful. More importantly, this second registration correctly matched the black crosses  (representing the junction of the deformed leaf) with the junction of the target, an alignment that the initial registration in \cref{fig:run1_leaves} failed to achieve.

The same process can be applied to correct the junction matching for target C. As seen previously in  \cref{fig:run1_leaves}, the matching failed because the top of the rod transformed into the bottom of the leaf, which is the exact opposite of target A. Therefore, the vector field $w_0$ inducing the space $W_0$ that $v$ must mimic is set to $w_0=[0,-1] \mathds{1}_D$ to enforce a downward translation of the junction. After optimization of \eqref{eq:pb_leaf_corr}, the obtained geodesic is presented in \cref{fig:run_leafA_corrected}, showing that both the global leaf and the junctions are successfully matched.

\begin{figure}[!h]
    \centering
    \includegraphics[width=\linewidth]{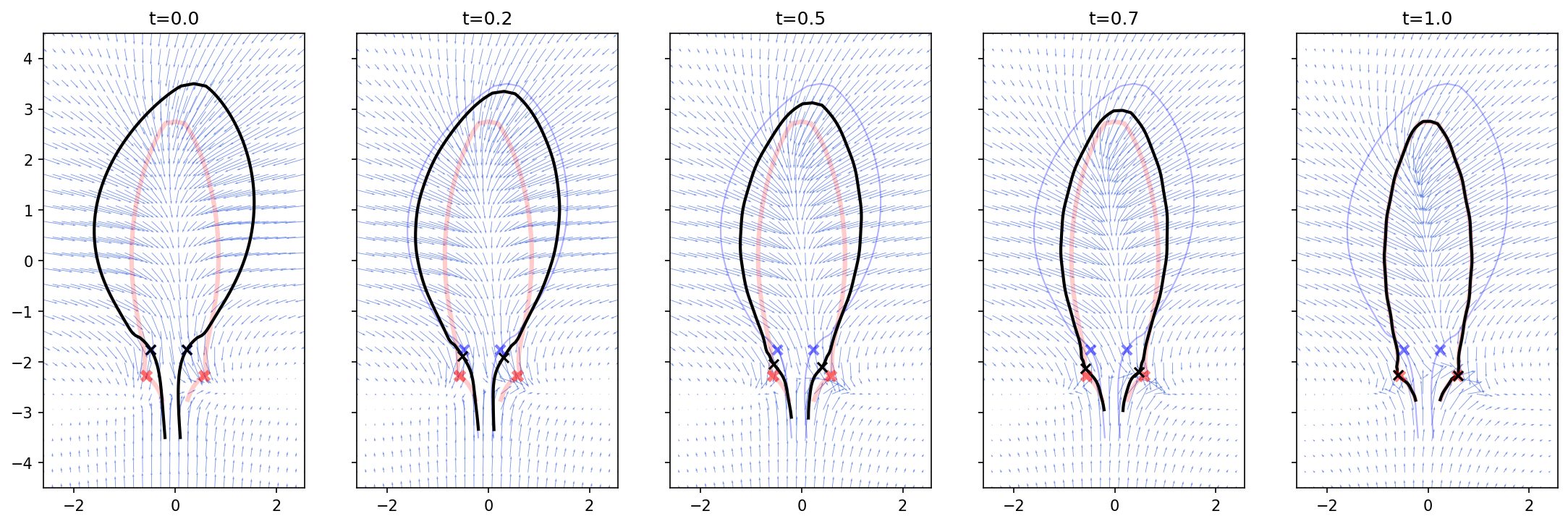}
    \caption{Geodesic associated with the target C after optimization of the constrained registration given in \eqref{eq:pb_leaf_corr}.}
    \label{fig:run_leafC_corrected}
\end{figure}

\FloatBarrier
\subsection{Correct a registration iteratively}\label{sec:exp_reg_it}

The previous sections demonstrated how the coupling score can enforce or prevent specific deformations when morphological priors are available. However, such explicit modeling is often unrealistic when dealing with complex vector fields or when prior knowledge is lacking. To address these scenarios, we introduce an iterative registration strategy. If an initial unconstrained registration yields an undesirable outcome, we extract the problematic vector field and embed it into the coupling score. Then, this score is used as a penalty term in a second registration in order to decouple the new deformation from the unsuccessful one produced during the first matching attempt.

\begin{figure}[!h]
    \centering
    \includegraphics[width=0.5\linewidth]{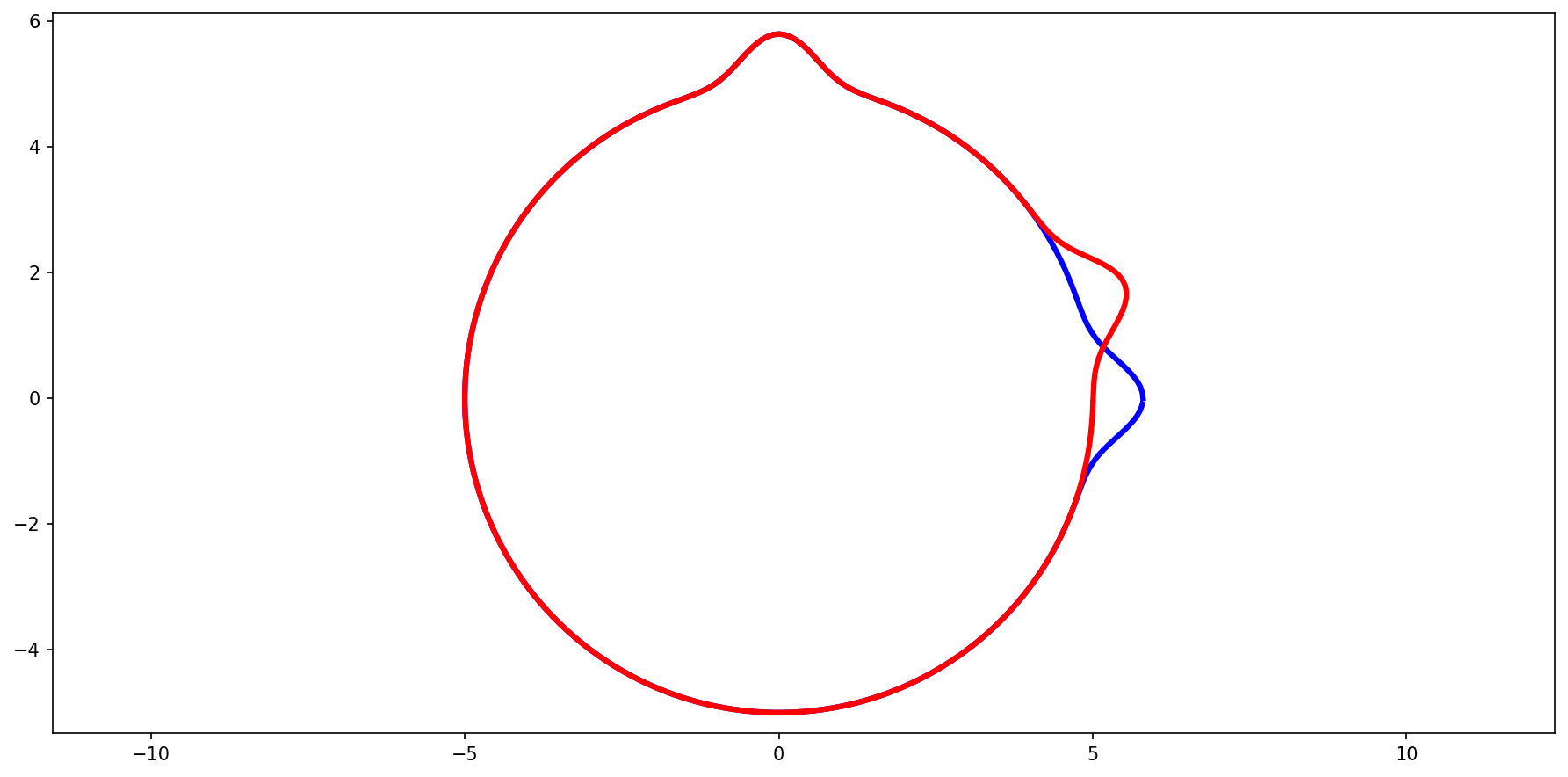}
    \caption{Source (blue) and target (red) curve }
    \label{fig:source_target_bosses}
\end{figure}

For this experiment, the source and target curves are shown in \cref{fig:source_target_bosses}. They share a bump at $\dfrac{\pi}{2}$, but their second bump differs. The expected behavior is for the rightmost source bump to be naturally transported onto the target bump. The registration is performed using standard LDDMM geodesic shooting where the deformation is induced by a RKHS $V$ whose kernel is a sum of Gaussian kernels with scales $\sigma_V \in \{0.5,2 \}$. We consider a varifold data attachment term which is also a sum of Gaussian kernels with scales $\sigma \in \{0.5,2\}$. The minimization problem associated with this registration is given in \eqref{eq:standard_model}.

\begin{figure}[!h]
    \centering
    \includegraphics[width=\linewidth]{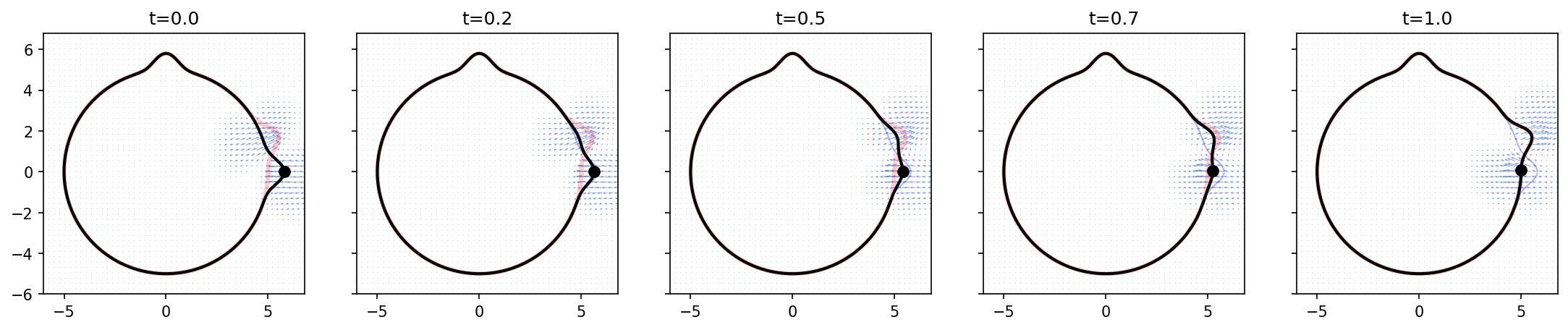}
    \caption{Geodesic of the initial registration problem defined by \eqref{eq:standard_model}. The source, target and geodesics are respectively represented in blue, red and black. The vector field $v_t$ is illustrated in blue at each time step. The black dot at the top of the right bump serves only for illustrative purposes and does not contribute to the matching process. }
    \label{fig:run1_bosses}
\end{figure}

\cref{fig:run1_bosses} shows the result of this matching. The black dot on the figure only serves to illustrate the displacement of the bump and do not contribute to the matching. We observe that the geodesic (in black) shrinks the right bump and  creates an entirely new one to match the target, which is energetically cheaper for the unstructured registration model than transporting the existing bump.

\begin{figure}[!h]
    \centering
 \begin{subfigure}{0.35\textwidth}
        \centering
         \includegraphics[width=\linewidth]{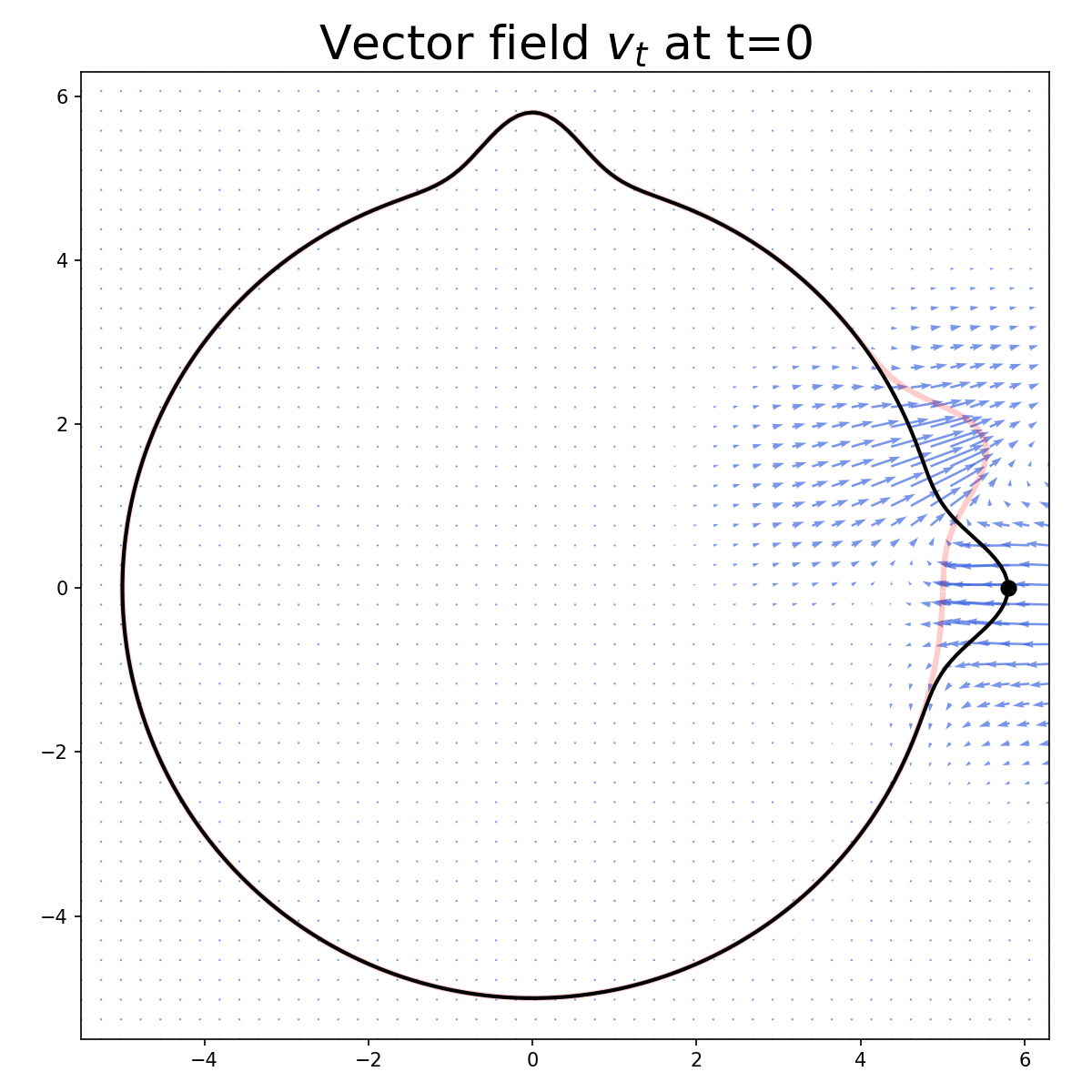}
        \label{fig:vfield_run1_bosses}
    \end{subfigure}
            \hspace{0.5cm} 
    \begin{subfigure}{0.35\textwidth}
        \centering
        \includegraphics[width=\linewidth]{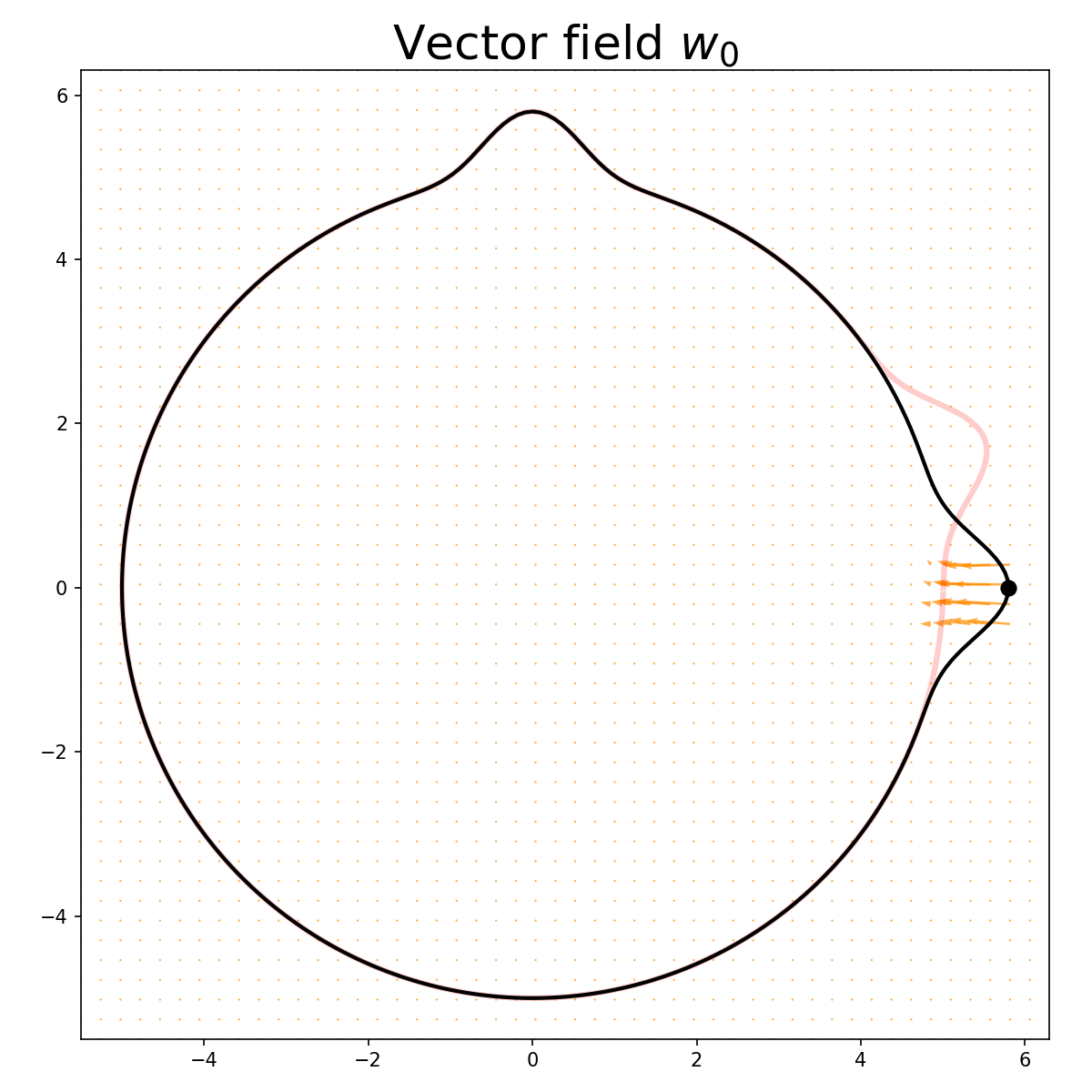}    
        \label{fig:w0_bosses}
    \end{subfigure}
    \vspace{0.5cm}
        \begin{subfigure}{0.35\textwidth}
        \centering
        \includegraphics[width=\linewidth]{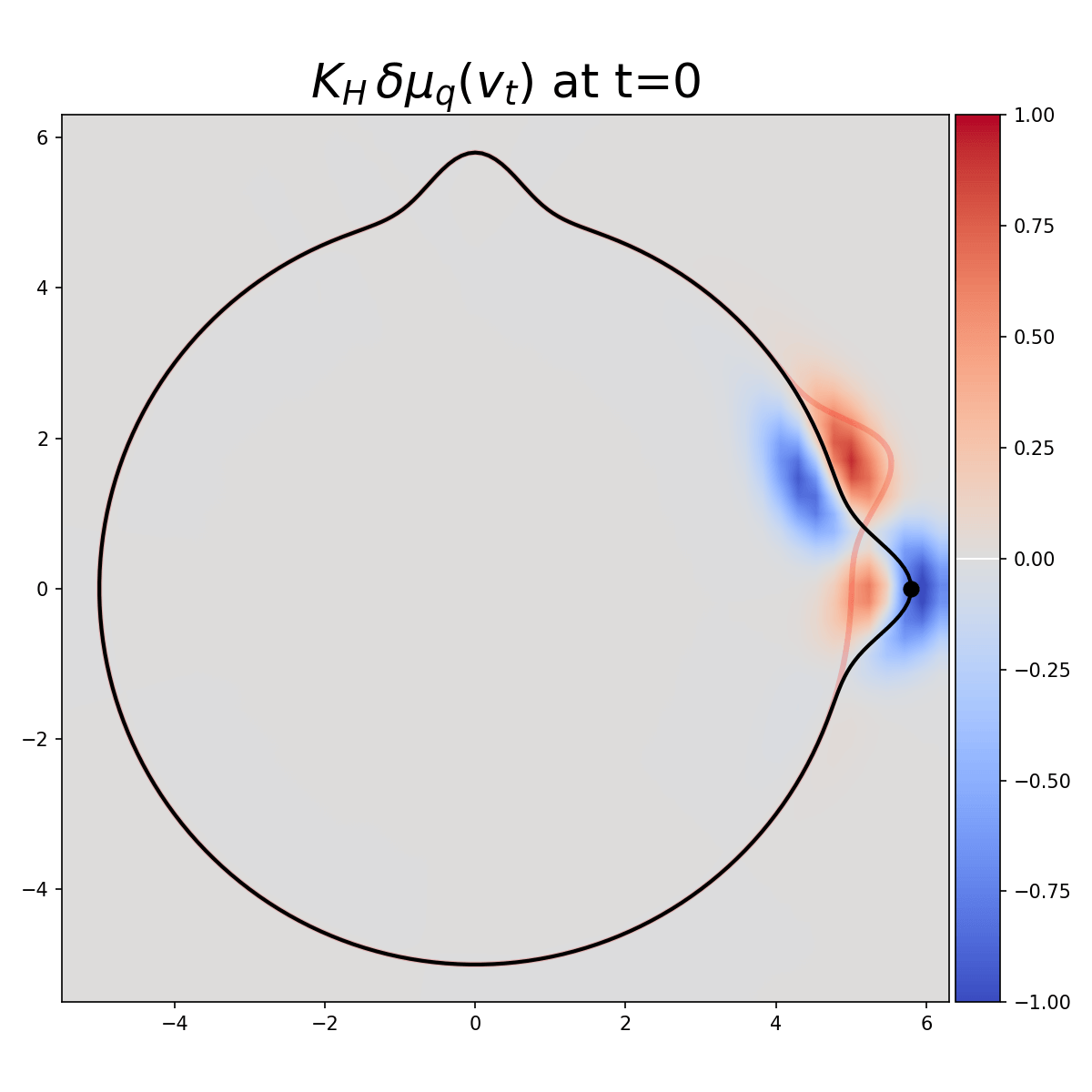}    
        \label{fig:heatmap_v_run1_bosses}
    \end{subfigure}
        \hspace{0.5cm}
        \begin{subfigure}{0.35\textwidth}
        \centering
        \includegraphics[width=\linewidth]{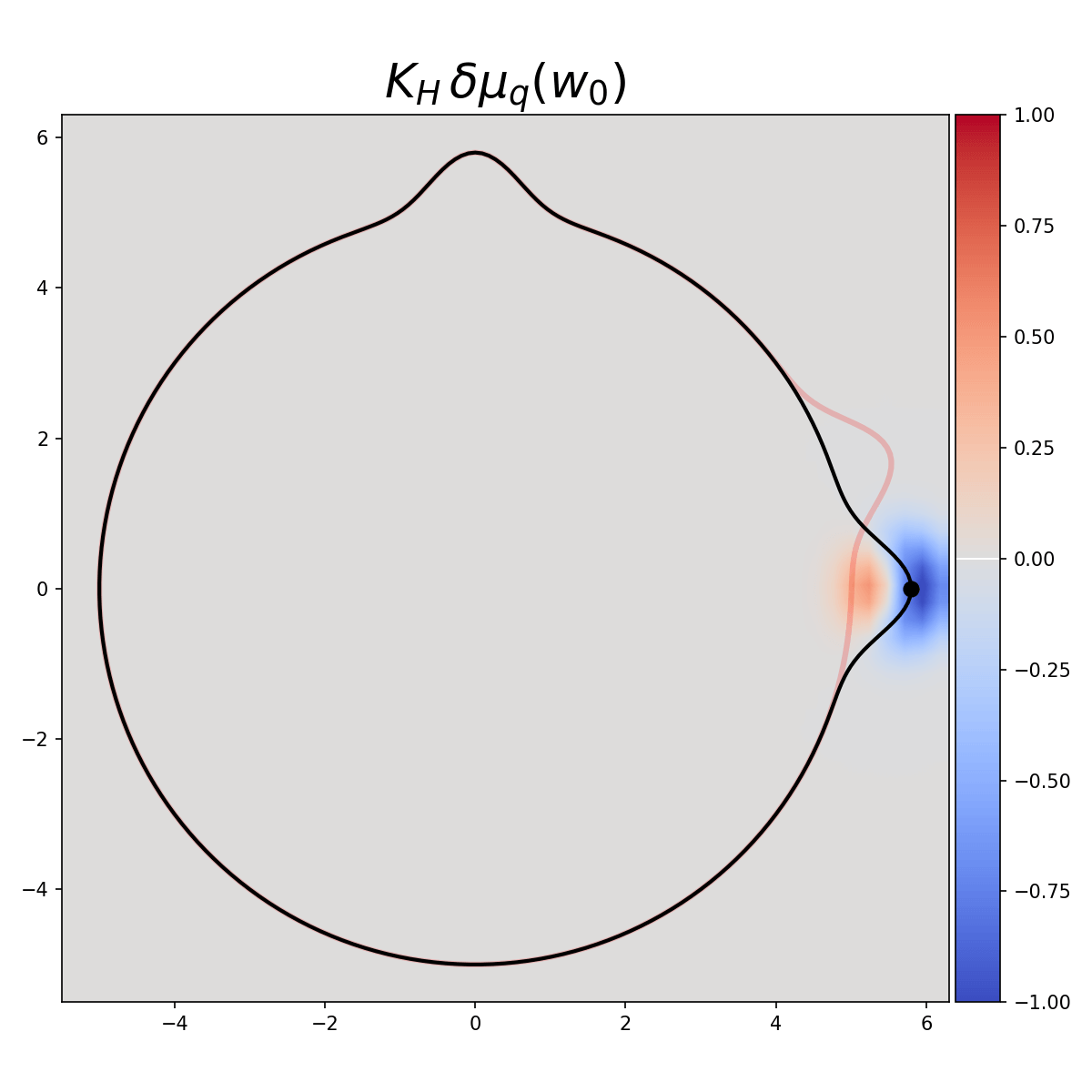}    
        \label{fig:heatmap_w_bosses}
    \end{subfigure}
    \caption{The first row shows the initial vector field $v_{t=0}$ from the first run (left) and the constraint vector field $w_0$ (right). Below them are the image representations of their first variation of varifolds associated with the template curve. The first variations are represented using a Gaussian kernel with scale $\sigma_H = 0.5$.}
    \label{fig:heatmap_run1}
\end{figure}

Indeed, \cref{fig:heatmap_run1} (top-left corner) shows that the initial vector field $v_0$ after optimization directly induces this shrinking behavior. Assume now that we would like the bump to be transported (as it could be desired for the matching of two gyri on a cortical surface). To correct the registration, we formulate a second optimization problem designed to find a new deformation decoupled from the problematic vector field generated during the first run. This second registration is similar to the initial one, but corrected using the coupling score evaluated at time $t=0$ between the new initial vector field $v_0$ and the reference space $W = \operatorname{span}(v_0^* \mathds{1}_D)$, with $D$ denoting the region of the right source bump, and $v_0^*$ representing the optimal initial vector field obtained from the first failed matching. The energy associated with this refined registration problem becomes:
\begin{equation}\label{eq:reg_pb2_bosses}
    \min_{v \in L^2([0,1],V)} \dfrac{1}{2} \int_0^1 \Vert v_t \Vert_V^2 \, dt + C_{q_0}\left(v_0, W\right) + \mathcal{D}(q_1)
\end{equation}
The vector field from which we aim to decouple is represented in the top-right corner. The bottom row shows the associated first variation, which are the quantities compared in the definition of the coupling score.

\begin{figure}[!h]
        \centering
        \includegraphics[width=\linewidth]{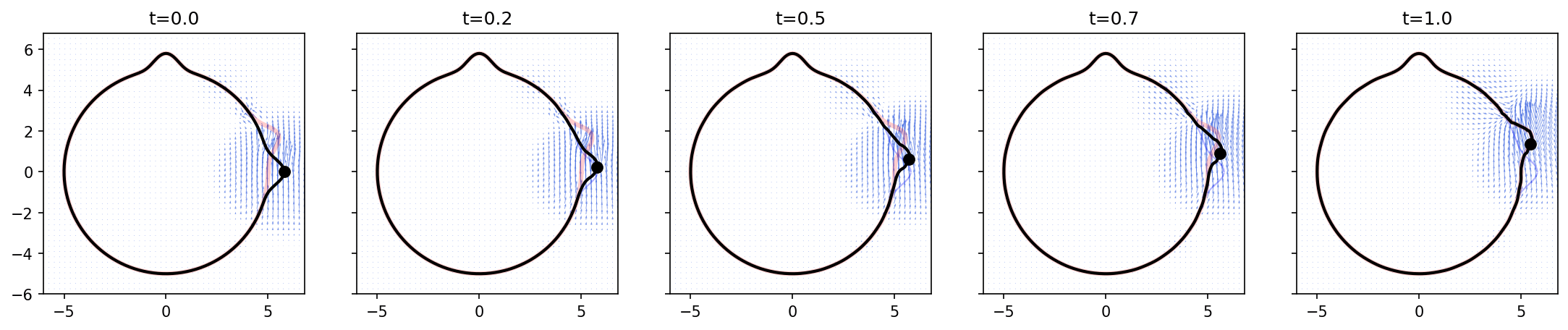}    
    \caption{Geodesic of the corrected registration problem \eqref{eq:reg_pb2_bosses} with coupling score.}
    \label{fig:run2_bosses}
\end{figure}

The geodesic generated by this second run is shown in \cref{fig:run2_bosses}. The coupling score successfully corrected the vector field, forcing the model to achieve the matching through a natural tangential displacement of the existing bump.

\section{Conclusion and future works}

In this paper, we proposed a novel framework designed to decouple different deformation modes in shape registration. We introduced the first variation of varifolds, a tool adapted from geometric measure theory, to provide a reparameterization-invariant representation of the infinitesimal action of vector fields on shapes. This allows us to define a coupling score that quantifies the extent to which the action of a vector field on shapes can be replicated by a chosen subspace. We provided a theoretical analysis of this score, illustrating its behavior for finite-dimensional spaces such as global translations and one-dimensional scalings. Finally, we integrated this coupling score into a registration model and  showed its versatility through different numerical experiments, such as decoupling deformations,  preventing or enforcing specific types of deformation, and iteratively correcting complex matching scenarios using structured directional priors.

While this study primarily focused on curves to work with simpler and more interpretable expressions, the proposed framework can be generalized to larger families of shapes, such as surfaces. Furthermore, the numerical experiments presented were based on toy examples in order to isolate and understand specific geometric phenomena that we wanted to highlight. A next step would be to evaluate the efficiency of this coupling score on real anatomical or biological datasets such as MRI for the matching of gyris on cortical surfaces. Regarding the registration strategy, our approach achieves decoupling by leveraging the geodesic equations of an unconstrained model. Consequently, future works could focus on the development of constrained registration frameworks specifically tailored for decoupling. Such approaches could include, for instance, deriving geodesic equations that directly incorporate the coupling score, or formulating models with distinct covectors to explicitly represent the different modes of deformation. From a modeling perspective, the coupling score could be used as a generative tool to construct and generate structured vector fields, as done with implicit deformation modules \cite{modules}. It might also be considered for partial matching applications in order to reconstruct a full global vector field, or to infer the underlying global deformation when the action of a vector field on a shape is only partially observed.

\clearpage
\appendix
\section{Proofs}

 \subsection{Proof of \cref{prop:frechet_varifold}}\label{app:proof_frechet}
 
\begin{proposition}[Fréchet differentiability of varifolds]
    Let $I = [0,1]$ or $\mathbb{S}^{1}$ and $Q = \{ q \in C^1(I, \mathbb{R}^d) \mid \inf_{s \in I} \|q'(s)\| > 0 \}$ be the set of immersions. The mapping
    \begin{equation*}
        \tilde{\mu} \colon \app{Q}{C_0^2(\mathbb{R}^d \times \mathbb{S}^{d-1}, \mathbb{R})'}{q}{\tilde{\mu}_q := \omega \mapsto \int_I \omega\left(q(s), \dfrac{q'(s)}{\|q'(s)\|}\right) \|q'(s)\| \, ds}
    \end{equation*}
    is continuously Fréchet differentiable.
\end{proposition}

\begin{proof}
    Let $q \in Q$. Since $q$ is an immersion defined on a compact domain $I$, there exist constants $M \ge c > 0$ such that $c \le \|q'(s)\| \le M$ for all $s \in I$. Because the space of immersions $Q$ is open in the $C^1$ topology, any variation $h \in C^1(I, \mathbb{R}^d)$ with a sufficiently small $C^1$-norm ensures that $q+h \in Q$. Specifically, we consider variations $h$ satisfying $\|h\|_{C^1} < \dfrac{c}{2}$, which guarantees $\|q'(s) + h'(s)\| \ge \dfrac{c}{2}$ for all $s \in I$.
    
    For $\omega \in C_0^2(\mathbb{R}^d \times \mathbb{S}^{d-1}, \mathbb{R})$ and for any $s \in I$, the mapping $t \mapsto \omega(q(s)+th(s), \tau_{q+th}(s))\|q'(s)+th'(s)\|$ is continuously differentiable in a neighborhood of $t=0$. Therefore, the mapping $t \mapsto \tilde{\mu}_{q+th}(\omega)$ is differentiable, and differentiation under the integral sign yields:
\begin{align*}
        \dfrac{d}{dt}\Big\vert_{t=0} \tilde{\mu}_{q+th}(\omega) &= \int_I \big\langle \nabla_x \omega(q(s), \tau_q(s)), h(s) \big\rangle \|q'(s)\| \, ds \nonumber\\
        &+ \int_I \big\langle \nabla_\tau \omega(q(s), \tau_q(s)), (I_d - \tau_q(s) \tau_q(s)^T)h'(s) \big\rangle \, ds \nonumber\\
        &+ \int_I \omega(q(s), \tau_q(s)) \big\langle \tau_q(s), h'(s) \big\rangle  \, ds.
    \end{align*}
    We propose the candidate for the Fréchet differential to be the linear operator $A_q$ defined by $A_q(h)(\omega) = \dfrac{d}{dt}\big\vert_{t=0} \tilde{\mu}_{q+th}(\omega)$ for $h \in C^1(I,\R^d)$ and $\omega \in C^2_0(\R^d \times \mathbb{S}^{d-1},\R)$, .
    By applying the triangle inequality, the Cauchy-Schwarz inequality, and using $\Vert I_d - \tau_q \tau_q^T \Vert \le 1$, the following bound holds
    \begin{align*}
        \vert A_q(h)(\omega) \vert &\le \int_I \Vert \nabla_x \omega \Vert \Vert h \Vert \Vert q' \Vert \, ds + \int_I \Vert \nabla_\tau \omega \Vert \Vert h' \Vert \, ds + \int_I \vert \omega \vert \Vert h' \Vert \, ds \\
        &\le \Vert \omega \Vert_{C^2_0} \Vert h \Vert_{C^1} \int_I \Vert q' \Vert \, ds + \Vert \omega \Vert_{C^2_0} \Vert h \Vert_{C^1} |I| + \Vert \omega \Vert_{C^2_0} \Vert h \Vert_{C^1} |I| \\
        &\le \Vert \omega \Vert_{C^2_0} \Vert h \Vert_{C^1} \left( \int_I \Vert q' \Vert \, ds + 2|I| \right),
    \end{align*}
   Consequently, $\omega \mapsto A_q(h)(\omega)$ is continuous and since it is linear, we conclude that $A_q(h) \in C_0^2(\R^d \times \mathbb{S}^{d-1},\R)'$.
   
    Let us show that $A_q$ is continuous. For $\omega \in C_0^2(\mathbb{R}^d \times \mathbb{S}^{d-1}, \mathbb{R})$, applying the Cauchy-Schwarz inequality, using $\|\tau_q(s)\| = 1$ and $\|I_d - \tau_q(s) \tau_q(s)^T\| \le 1$, it follows that:
    \begin{align*}
        |A_q(h)(\omega)| &\leq \int_I \left( \|\omega\|_{C^2_0} \|h\|_{C^1} M + \|\omega\|_{C^2_0} \|h\|_{C^1} + \|\omega\|_{C^2_0} \|h\|_{C^1} \right) \, ds \\
        &\leq |I| (M + 2) \|\omega\|_{C^2_0} \|h\|_{C^1}.
    \end{align*}
    Taking the supremum over all test functions $\omega$ satisfying $\|\omega\|_{C^2_0} \leq 1$, we obtain $\|A_q(h)\|_{(C^2_0)'} \leq C_q \|h\|_{C^1}$. Therefore, we conclude that $A_q$ is a bounded linear operator.

    Let us introduce the mapping $t \mapsto \tilde{\mu}_{q+th}(\omega)$, which is twice continuously differentiable. Its Taylor expansion with integral remainder evaluated at $t=1$ reads as:
    \begin{equation*}
         \tilde{\mu}_{q+h}(\omega) = \tilde{\mu}_{q}(\omega) + A_q(h)(\omega) + \int_0^1 (1 - \lambda) \dfrac{d^2}{dt^2} \Big\vert_{t= \lambda} \tilde{\mu}_{q+th}(\omega) \, d\lambda.
    \end{equation*}

    To make the second-order derivative explicit, we introduce the function $F_\omega \colon (x,y) \in \mathbb{R}^d \times (\mathbb{R}^d \setminus \{0\}) \mapsto \omega\left(x, \dfrac{y}{\|y\|}\right) \|y\|$, which is $C^2$ over its domain. Let us denote for each $s \in I$, $(q_\lambda(s), q'_\lambda(s)) = (q(s) + \lambda h(s), q'(s) + \lambda h'(s))$ and $\tau_\lambda(s) = \dfrac{q'_\lambda(s)}{\|q'_\lambda(s)\|}$. The second order derivative is given by: 
    \begin{equation*}
        \dfrac{d^2}{dt^2} \Big\vert_{t= \lambda} \tilde{\mu}_{q+th}(\omega) = \int_I d^2F_{\omega}(q_\lambda(s), q'_\lambda(s)) [(h(s), h'(s))^2] \, ds.
    \end{equation*}
    
    Recalling that the Jacobian matrix of the radial projection $y \mapsto \dfrac{y}{\|y\|}$ evaluated at $q'_{\lambda}(s)$ is $\dfrac{1}{\|q'_\lambda(s)\|}(I_d - \tau_\lambda(s) \tau_\lambda(s)^T)$, the partial Hessians of $F_\omega$ are given by:
    \begin{align*}
    \nabla_{xx}^2 F_{\omega}(q_\lambda, q'_\lambda) &= \|q'_\lambda\| \nabla_{xx}^2 \omega(q_\lambda, \tau_\lambda) \\
    \nabla_{xy}^2 F_{\omega}(q_\lambda, q'_\lambda) &= \nabla_{x\tau}^2 \omega(q_\lambda, \tau_\lambda) (I_d - \tau_\lambda \tau_\lambda^T) + \nabla_x \omega(q_\lambda, \tau_\lambda) \tau_\lambda^T \\
    \nabla_{yy}^2 F_{\omega}(q_\lambda, q'_\lambda) &= \dfrac{1}{\|q'_\lambda\|} \Big( \nabla_{\tau\tau}^2 \omega(q_\lambda, \tau_\lambda) (I_d - \tau_\lambda \tau_\lambda^T) \\
    &\quad + \tau_\lambda \nabla_\tau \omega(q_\lambda, \tau_\lambda)^T (I_d - \tau_\lambda \tau_\lambda^T) + \omega(q_\lambda, \tau_\lambda) (I_d - \tau_\lambda \tau_\lambda^T) \Big).
\end{align*}
where we omitted the parameter $s$ for readability.

    Since $\|h\|_{C^1} < c/2$, it follows that $M + c/2 \ge \|q'_\lambda(s)\| \ge c/2$. Combining these partial Hessians, there exists a constant $M_q > 0$ depending only on $q$ such that for all $\lambda \in [0,1]$ and $s \in I$:
    \begin{equation*}
        \left\| d^2F_\omega(q_\lambda(s), q'_\lambda(s)) \right\| \le M_q \|\omega\|_{C^2_0}.
    \end{equation*}
    
    Substituting this uniform bound into the remainder integral, we obtain:
    \begin{align*}
        \left| \int_0^1 (1 - \lambda) \dfrac{d^2}{dt^2} \Big\vert_{t= \lambda} \tilde{\mu}_{q+th}(\omega) \, d\lambda \right| 
        &\leq \int_0^1 \int_I (1-\lambda) \left\| d^2F_{\omega}(q_\lambda(s), q'_\lambda(s)) \right\| \|(h(s), h'(s))\|^2 \, ds \, d\lambda \\
        &\leq \int_0^1 \int_I (1-\lambda) M_q \|\omega\|_{C^2_0} \ \|h\|_{C^1}^2 \big) \, ds \, d\lambda \\
        &\leq \dfrac{1}{2} |I| M_q \|\omega\|_{C^2_0} \|h\|_{C^1}^2.
    \end{align*}

    Taking the supremum over all test functions satisfying $\|\omega\|_{C^2_0} \le 1$, we obtain that $\|\tilde{\mu}_{q+h} - \tilde{\mu}_{q} - A_q(h)\|_{(C^2_0)'} \le C'_q \|h\|_{C^1}^2$. This establishes that the mapping $q \mapsto \tilde{\mu}_q$ is Fréchet differentiable, with differential $d\tilde{\mu}_q = A_q$.

    Finally, we show that the mapping $q \mapsto A_q$ is continuous. Let $(q_n)_{n \in \mathbb{N}}$ be a sequence in $Q$ converging to $q$ in the $C^1$ topology. Let $h \in C^1(I,\R^d)$ and $\omega \in C_0^2(\R^d \times \mathbb{S}^{d-1},\R)$.    
    Using the triangle inequality, we have:
    \begin{align*}
        |A_q(h)(\omega) - A_{q_n}(h)(\omega)| &\le \int_I \left| \langle \nabla_x \omega(q, \tau_q), h \rangle \|q'\| - \langle \nabla_x \omega(q_n, \tau_{q_n}), h \rangle \|q_n'\| \right| ds \\
        &\quad + \int_I \left| \langle \nabla_\tau \omega(q, \tau_q), (I_d - \tau_q \tau_q^T)h' \rangle - \langle \nabla_\tau \omega(q_n, \tau_{q_n}), (I_d - \tau_{q_n} \tau_{q_n}^T)h' \rangle \right| ds \\
        &\quad + \int_I \left| \omega(q, \tau_q) \langle \tau_q, h' \rangle - \omega(q_n, \tau_{q_n}) \langle \tau_{q_n}, h' \rangle \right| ds.
    \end{align*}

    Since $q_n \to q$ in $C^1(I, \R^d)$, the sequences $(q_n)$ and $(q_n')$ converge uniformly on $I$. Therefore for $n$ sufficiently large, $c/2 \le \|q_n'(s)\| \le M + c/2$.
    On $\R^d \setminus B(0, c/2)$, the normalization mapping $y \mapsto \dfrac{y}{\|y\|}$ and the projection mapping $y \mapsto I_d - \dfrac{y y^T}{\|y\|^2}$ are smooth and Lipschitz continuous. Let $L > 0$ be a common Lipschitz constant for these maps on this domain. In particular, $\|\tau_q - \tau_{q_n}\| \le L \|q' - q_n'\| \le L \|q - q_n\|_{C^1}$.
    Furthermore, since $\omega \in C^2_0$, the function $\omega$ and its first derivatives $\nabla_x \omega$ and $\nabla_\tau \omega$ are bounded and Lipschitz continuous, with Lipschitz constants bounded by $\|\omega\|_{C^2_0}$.
    We bound the first integrand using the triangle inequality and Cauchy-Schwarz inequality:
    \begin{align*}
        &\left| \langle \nabla_x \omega(q, \tau_q), h \rangle \|q'\| - \langle \nabla_x \omega(q_n, \tau_{q_n}), h \rangle \|q_n'\| \right| \\
        &\le \left| \langle \nabla_x \omega(q, \tau_q) - \nabla_x \omega(q_n, \tau_{q_n}), h \rangle \right| \|q'\| + \left| \langle \nabla_x \omega(q_n, \tau_{q_n}), h \rangle \right| \big| \|q'\| - \|q_n'\| \big| \\
        &\le \|\omega\|_{C^2_0} \Big( \|q - q_n\| + \|\tau_q - \tau_{q_n}\| \Big) \|h\| M + \|\omega\|_{C^2_0} \|h\| \|q' - q_n'\| \\
        &\le K_1 \|\omega\|_{C^2_0} \|h\|_{C^1} \|q - q_n\|_{C^1}.
    \end{align*}

    Applying the same decomposition strategy to the second and third integrals and using the Lipschitz continuity of the projection operator and the unit tangent vectors, we obtain similar bounds:
    \begin{align*}
        \left| \langle \nabla_\tau \omega(q, \tau_q), (I_d - \tau_q \tau_q^T)h' \rangle - \langle \nabla_\tau \omega(q_n, \tau_{q_n}), (I_d - \tau_{q_n} \tau_{q_n}^T)h' \rangle \right| &\le K_2 \|\omega\|_{C^2_0} \|h\|_{C^1} \|q - q_n\|_{C^1} \\
        \left| \omega(q, \tau_q) \langle \tau_q, h' \rangle - \omega(q_n, \tau_{q_n}) \langle \tau_{q_n}, h' \rangle \right| &\le K_3 \|\omega\|_{C^2_0} \|h\|_{C^1} \|q - q_n\|_{C^1}.
    \end{align*}

    Integrating these point-wise bounds over the compact interval $I$, and defining $K = K_1 + K_2 + K_3$, we obtain:
    \begin{equation*}
        |A_q(h)(\omega) - A_{q_n}(h)(\omega)| \le |I| K \|\omega\|_{C^2_0} \|h\|_{C^1} \|q - q_n\|_{C^1}.
    \end{equation*}
    Taking the supremum over the unit balls $\|h\|_{C^1} \le 1$ and $\|\omega\|_{C^2_0} \le 1$ directly provides the bound on the operator norm difference:
    \begin{equation*}
        \|A_q - A_{q_n}\|_{\mathcal{L}(C^1, (C^2_0)')} \leq |I| K \|q - q_n\|_{C^1}.
    \end{equation*}
    Since $\|q - q_n\|_{C^1} \to 0$ as $n \to \infty$, then   $\|A_q - A_{q_n}\|_{\mathcal{L}(C^1, (C^2_0)')} \to 0$. This proves that the mapping $q \mapsto A_q$ is continuous, and therefore $q \mapsto \tilde{\mu}_q$ is continuously Fréchet differentiable.
\end{proof}

\subsection{Proof of \cref{prop:approx_first_var}}\label{app:proof_approx}
Recall the notations $\ell_f = \Vert q_{f^2}-q_{f^1} \Vert$, $\vec{t}_f = \dfrac{q_{f^2} - q_{f^1}}{\ell_f}$, $c_f = \dfrac{q_{f^1} + q_{f^2}}{2}$,
     $(\delta \ell_f)_v = \langle \dfrac{v(q_{f^2})-v(q_{f^1})}{\Vert q_{f^2}-q_{f^1} \Vert}, \vec{t}_f \rangle$ and  $ \nabla^\perp v_f=\dfrac{v(q_{f^2})-v(q_{f^1})}{\ell_f} - \Big \langle \dfrac{v(q_{f^2})-v(q_{f^1})}{\ell_f} , \vec{t}_f\Big \rangle\vec{t}_f$. We also recall from \cref{exam:smooth_curve} that the varifold associated with $q$, defined by $\tilde{\mu}_q(\omega) = \int_I \omega\left(q(s),\tau_{q}(s)\right) \Vert q'(s) \Vert ds$ can be approximated by a weighted sum of Dirac masses $\mu_q(\omega)= \sum_{f \in F_q} \ell_f \omega(c_f,\vec{t}_f)$ for a sufficiently fine discretization. The first variation of the varifold associated with $\tilde{\mu}_q$ is given by
\begin{align}\label{eq:cont_first_var_prop}
    \delta \tilde{\mu}_q(v)(\omega) &= \int_I \left( \omega ( q(s),\tau_q(s) ) \langle dv(q(s)) \cdot \tau_q(s), \tau_q(s) \rangle   + \left\langle \nabla_x \omega ( q(s),\tau_q(s) ), v(q(s)) \right\rangle \right) \Vert q'(s) \Vert ds  \nonumber \\
    &+ \int_I \left\langle \nabla_{\tau} \omega ( q(s),\tau_q(s) ), (dv(q(s))\cdot \tau_q(s))^\perp \right\rangle \Vert q'(s) \Vert ds
\end{align}
and the first variation of the varifold associated with $\mu_q$ is given by
\begin{equation}\label{eq:disc_first_var_prop}
      \delta \mu_q(v)(\omega) =\sum_{f \in F_q} \left[ \ell_f  (\delta \ell_f)_v \omega(c_f,\vec{t}_f)  +  \ell_f \langle \nabla_x \omega(c_f,\vec{t}_f), v_f \rangle  + \ell_f \langle \nabla_{\tau} \omega(c_f,\vec{t}_f), \nabla^\perp v_f \rangle \right].
\end{equation}

\begin{proposition}[Approximation of the first variation of varifold]
    Let $q \in Q$ be an immersion and $((q_i),F_q) $ a discretization of $q$. For any vector field $v \in C_0^2(\mathbb{R}^d,\mathbb{R}^d)$, 
    \begin{equation*}
        \left\Vert \delta \tilde{\mu}_q(v) - \delta \mu_q(v) \right\Vert_{C_0^2(\mathbb{R}^d \times \mathbb{S}^{d-1},\mathbb{R})'} \leq C_q \operatorname{length}(q) \Vert v \Vert_{C_0^2} \max_{f \in F_q} \ell_f
    \end{equation*}
    where $C_q > 0$ is a constant that depends on the curve and its discretization.
\end{proposition}
\begin{proof}
    For a fixed test function $\omega \in C_0^2(\mathbb{R}^d \times \mathbb{S}^{d-1},\mathbb{R})$, we have:
    \begin{align}\label{eq:global_bound}
        (\delta \tilde{\mu}_{q}(v) - \delta \mu_q(v))(\omega)  &=  \int_I \omega(q, \tau_q) \langle dv(q) \cdot \tau_q, \tau_q \rangle \Vert q' \Vert ds - \sum_{f \in F_q} \ell_f (\delta \ell_f)_v \omega(c_f, \vec{t}_f) \\
        &\quad +  \int_I \langle \nabla_x \omega(q, \tau_q), v(q) \rangle \Vert q' \Vert ds - \sum_{f \in F_q} \ell_f \langle \nabla_x \omega(c_f, \vec{t}_f), v_f \rangle \notag\\
        &\quad +  \int_I \langle \nabla_{\tau} \omega(q, \tau_q), (dv(q) \cdot \tau_q)^\perp \rangle \Vert q'\Vert ds - \sum_{f \in F_q} \ell_f \langle \nabla_{\tau} \omega(c_f, \vec{t}_f), \nabla^\perp v_f \rangle . \notag
    \end{align}
    
    For any face $f \in F_q$ of the discretized curve $((q_i),F_q)$ we define the associated parameterization subinterval $I_f = [s_1, s_2]$ (with $s_1 < s_2$). Because $q$ is a $C^2$ immersion on a compact domain, its derivative is bounded, positive, and its maximal curvature $\kappa_{\max}$ is well-defined and finite. This guarantees that its continuous arc length $L_f = \int_{I_f} \Vert q'(s) \Vert ds$ is positive. Since the discretized curve consists of a finite set of faces $F_q$ and $\ell_f > 0$, we can define the constant $C_0 = \max_{f \in F_q} \dfrac{L_f}{\ell_f} \geq 1$. Thus, for any face, we have $L_f \leq C_0 \ell_f$.

    We first establish some inequalities that will be useful later. Let $f \in F_q$ be a face and $s \in I_f$. By the triangle inequality, we have:
    \begin{equation*}
        \Vert q(s) - c_f \Vert \leq \Vert q(s) - q_{f^1} \Vert + \Vert q_{f^1} - c_f \Vert =  \Vert q(s) - q_{f^1} \Vert + \dfrac{\ell_f}{2}.
    \end{equation*}
    To bound the first term, we apply the mean value inequality:
    \begin{equation*}
        \Vert q(s) - q_{f^1} \Vert = \left\Vert \int_{s_1}^s q'(u) du \right\Vert \leq \int_{s_1}^s \Vert q'(u) \Vert du \leq \int_{s_1}^{s_2} \Vert q'(u) \Vert du = L_f \leq C_0 \ell_f.
    \end{equation*}
    Combining these expressions yields the bound:
    \begin{equation}
        \Vert q(s) - c_f \Vert \leq \left(C_0 + \dfrac{1}{2}\right) \ell_f = K_1 \ell_f.
    \end{equation}

    We now bound the tangential deviation  $\Vert \tau_q(s) - \vec{t}_f \Vert$. To simplify the calculus, we  define $\tilde{q}_{\vert_{I_f}} : [0,L_f] \to \R^d$ which is the reparameterization of $q\vert_{I_f} : I_f \to \R^d$ by its arclength. We denote $\alpha : [0,L_f] \to I_f$ the reparameterization mapping such that $q\vert_{I_f} \circ \alpha = \tilde{q}\vert_{I_f}$. 
    For $s \in I_f$, we denote $u_s=\alpha^{-1}(s) \in [0,L_f]$. Then, the following bound holds
    \begin{align*}
        \Vert \tau_q(s) - \tau_q(s_1) \Vert = \Vert \tau_{\tilde{q}}(u_s) - \tau_{\tilde{q}}(0) \Vert = \Vert \int_0^{u_s} \tau_{\tilde{q}}'(x) dx \Vert \leq \int_0^{u_s} \Vert \tau'_{\tilde{q}}(x) \Vert dx \leq \kappa_{\max} L_f
    \end{align*}
    For the relative length error $ \vert 1 - \dfrac{L_f}{\ell_f} \vert$, we first establish that:
    \begin{align*}
        \ell_f  = \Vert \int_0^{L_f} \tau_{\tilde{q}}(u)du \Vert = \Vert \int_0^{L_f} \tau_{\tilde{q}}(0) + (\tau_{\tilde{q}}(u) - \tau_{\tilde{q}}(0))du \Vert &\geq \Vert \int_0^{L_f} \tau_{\tilde{q}}(0)du \Vert - \Vert \int_0^{L_f} \tau_{\tilde{q}}(u) - \tau_{\tilde{q}}(0)du \Vert \\ 
        &\geq L_f - \kappa_{\max} L_f^2 .
    \end{align*}
    Therefore,
\begin{equation}
\left\vert 1 - \dfrac{L_f}{\ell_f} \right\vert = \dfrac{L_f}{\ell_f} - 1 = \dfrac{L_f}{\ell_f} \left( 1 - \dfrac{\ell_f}{L_f} \right) \leq C_0  \kappa_{\max} L_f \leq C_0^2 \kappa_{\max}\ell_f.    
\end{equation}

    Finally, to bound the tangential deviation  $\Vert \tau_q(s) - \vec{t}_f \Vert$, we express the discrete tangent vector as the integral $\vec{t}_f = \dfrac{1}{\ell_f} \int_{0}^{L_f} \tau_{\tilde{q}}(u) du$. By subtracting and adding $\dfrac{L_f}{\ell_f}\tau_q(s)$, we obtain:
    \begin{align*}
        \Vert \tau_q(s) - \vec{t}_f \Vert &= \left\Vert \tau_{\tilde{q}}(u_s) \left( 1 - \dfrac{L_f}{\ell_f} \right)  - \dfrac{1}{\ell_f} \int_0^{L_f} (\tau_{\tilde{q}}(u) - \tau_{\tilde{q}}(u_s)) du \right\Vert \\
        &\leq \left\vert 1 - \dfrac{L_f}{\ell_f} \right\vert + \dfrac{1}{\ell_f} \int_0^{L_f} \Vert\tau_{\tilde{q}}(u_s) - \tau_{\tilde{q}}(u) \Vert du \\
        &\leq \kappa_{\max} L_f +\dfrac{L_f^2}{\ell_f} \kappa_{\max}
        \\
        &\leq 2 \kappa_{\max} C_0 \ell_f =K_2 \ell_f.
    \end{align*}

For a fixed test function $\omega \in C_0^2(\mathbb{R}^d \times \mathbb{S}^{d-1},\mathbb{R})$ with $\Vert \omega \Vert_{C_0^2} \leq 1$, we want to bound the global difference $\vert (\delta \tilde{\mu}_{q}(v) - \delta \mu_q(v))(\omega) \vert$ given in \cref{eq:global_bound} using the bounds we just established. To do so, we distribute the discrete sums into the integrals over each interval $I_f$. By the triangle inequality, the error is bounded by the sum of local pointwise errors:
    \begin{equation*}
        \vert (\delta \tilde{\mu}_{q}(v) - \delta \mu_q(v))(\omega) \vert \leq \sum_{f \in F_q} \int_{I_f} \Big( A_1(s) + A_2(s) + A_3(s) \Big) \|q'(s)\| ds,
    \end{equation*}
    where 
    \begin{align*}
        A_1(s) &= \left\vert \omega(q(s),\tau_q(s)) \langle dv(q(s)) \cdot \tau_q(s), \tau_q(s) \rangle - \dfrac{\ell_f}{L_f} (\delta \ell_f)_v \omega(c_f,\vec{t}_f) \right\vert, \\
        A_2(s) &= \left\vert \langle \nabla_x \omega(q(s),\tau_q(s)), v(q(s)) \rangle - \dfrac{\ell_f}{L_f} \langle \nabla_x \omega(c_f,\vec{t}_f), v_f \rangle \right\vert, \\
        A_3(s) &= \left\vert \langle \nabla_{\tau} \omega(q(s),\tau_q(s)), (dv(q(s))\cdot \tau_q(s))^\perp \rangle - \dfrac{\ell_f}{L_f} \langle \nabla_{\tau} \omega(c_f,\vec{t}_f), \nabla^\perp v_f \rangle \right\vert.
    \end{align*}
    
   We proceed to bound each of these three terms using the Lipschitz continuity of the integrands and the bounds established earlier.

Let $H_2(x,\tau) = \langle \nabla_x \omega(x,\tau), v(x) \rangle$. Because $v$ and $\omega$ have $C^2_0$ regularity, their derivatives are bounded by $\Vert v \Vert_{C^2_0}$ and $\Vert \omega \Vert_{C^2_0}$. Therefore, $H_2$ is Lipschitz continuous with constant bounded by $2\Vert v \Vert_{C^2_0} \Vert \omega \Vert_{C^2_0}$. We recall that $v_f = \dfrac{v(q_{f^1}) + v(q_{f^2})}{2}$ which differs from $v(c_f)$. To bound $A_2(s)$, we add and subtract both $H_2(c_f, \vec{t}_f) = \langle \nabla_x \omega(c_f, \vec{t}_f), v(c_f) \rangle$ and $\langle \nabla_x \omega(c_f, \vec{t}_f), v_f \rangle$, the triangle inequality yields:
    \begin{align*}
        A_2(s) &\leq \vert H_2(q(s),\tau_q(s)) - H_2(c_f,\vec{t}_f) \vert  + \left\vert \langle \nabla_x \omega(c_f,\vec{t}_f), v(c_f) - v_f \rangle \right\vert  + \left\vert 1 - \dfrac{\ell_f}{L_f} \right\vert \left\vert \langle \nabla_x \omega(c_f,\vec{t}_f), v_f \rangle \right\vert.
    \end{align*}
    To bound the middle term, we use the mean value inequality on $v$. Since $\|q_{f^i} - c_f\| = \ell_f/2$, we have $\|v(q_{f^i}) - v(c_f)\| \le \|v\|_{C^2_0} \dfrac{\ell_f}{2}$ for $i \in \{1,2\}$. Consequently, the difference is bounded by:
    \begin{equation*}
        \Vert v(c_f) - v_f \Vert \leq \dfrac{1}{2} \Vert v(c_f) - v(q_{f^1}) \Vert + \dfrac{1}{2} \Vert v(c_f) - v(q_{f^2}) \Vert \leq \Vert v \Vert_{C^2_0} \dfrac{\ell_f}{2}.
    \end{equation*}
    Substituting this bound, along with the Lipschitz continuity of $H_2$, we obtain:
    \begin{align*}
        A_2(s) &\leq   \Vert v \Vert_{C^2_0} \Vert \omega \Vert_{C^2_0} \left( 2\Vert q(s) - c_f \Vert + 2\Vert \tau_q(s) - \vec{t}_f \Vert +\dfrac{\ell_f}{2} + \left\vert 1 - \dfrac{\ell_f}{L_f} \right\vert\right) \\
        &\leq   2\Vert v \Vert_{C^2_0} \Vert \omega \Vert_{C^2_0} \left( 2K_1  + 2K_2  + \dfrac{1}{2}+ (C_0)^2 \kappa_{\max} \right)\ell_f
    \end{align*}

For $A_1(s)$, we introduce the $C^1$ function $H_1(x,\tau) = \omega(x,\tau) \langle dv(x) \cdot \tau, \tau \rangle$ which is Lipschitz continuous with constant bounded by $\Vert \omega \Vert_{C^2_0} \Vert v \Vert_{C_0^2}$.
We then proceed to bound $A_1(s)$. Taylor's inequality reads:
    \begin{align*}
        \left\Vert v(q_{f^i}) - v(c_f) - dv(c_f) \cdot \left( \dfrac{\ell_f}{2} \vec{t}_f \right) \right\Vert &\leq \dfrac{1}{2} \Vert v \Vert_{C_0^2} \left\Vert \dfrac{\ell_f}{2} \vec{t}_f \right\Vert^2 = \dfrac{1}{8} \Vert v \Vert_{C_0^2} \ell_f^2.
    \end{align*}
Therefore, subtracting the expression for $i=2$ and $i=1$ and applying the triangle inequality to this difference yields:
    \begin{equation*}
        \Vert v(q_{f^2}) - v(q_{f^1}) - dv(c_f) \cdot \ell_f \vec{t}_f \Vert \leq \dfrac{1}{8} \Vert v \Vert_{C_0^2} \ell_f^2 + \dfrac{1}{8} \Vert v \Vert_{C_0^2} \ell_f^2 = \dfrac{1}{4} \Vert v \Vert_{C_0^2} \ell_f^2.
    \end{equation*}
    Dividing this expression by $\ell_f$ provides the bound:
    \begin{equation*}
        \left\Vert \dfrac{v(q_{f^2}) - v(q_{f^1})}{\ell_f} - dv(c_f) \cdot \vec{t}_f \right\Vert \leq \dfrac{1}{4} \Vert v \Vert_{C_0^2} \ell_f.
    \end{equation*}

By adding and subtracting both $H_1(c_f,\vec{t}_f)$ and $\dfrac{\ell_f}{L_f}H_1(c_f,\vec{t}_f)$, we can decompose $A_1(s)$ into three terms using the triangle inequality:
    \begin{align*}
        A_1(s) &= \left\vert H_1(q(s),\tau_q(s)) - \dfrac{\ell_f}{L_f} \omega(c_f, \vec{t}_f) (\delta \ell_f)_v \right\vert \\
        &\leq \vert H_1(q(s),\tau_q(s)) - H_1(c_f,\vec{t}_f) \vert  + \dfrac{\ell_f}{L_f} \left\vert \omega(c_f,\vec{t}_f) \left\langle dv(c_f) \cdot \vec{t}_f - \dfrac{v(q_{f^2}) - v(q_{f^1})}{\ell_f}, \vec{t}_f \right\rangle \right\vert \\
        &\quad +\left\vert 1 - \dfrac{\ell_f}{L_f} \right\vert \vert H_1(c_f,\vec{t}_f) \vert.
    \end{align*}

    By the Lipschitz continuity of $H_1$, and the bounds previously established,the term $A_1$ is bounded by:
    \begin{align*}
        A_1(s) &\leq  \Vert \omega \Vert_{C^2_0} \Vert v \Vert_{C_0^2} \left( \Vert q(s) - c_f \Vert + \Vert \tau_q(s) - \vec{t}_f \Vert + \dfrac{1}{4}  \ell_f + \left\vert 1 - \dfrac{\ell_f}{L_f} \right\vert \right) \\
        &\leq \Vert \omega \Vert_{C^2_0} \Vert v \Vert_{C_0^2} \left( K_1 + K_2 + \dfrac{1}{4}  + (C_0)^2 \kappa_{\max} \right)\ell_f .
    \end{align*}
    
For $A_3(s)$, we introduce the $C^1$ function $H_3(x,\tau) = \langle \nabla_{\tau} \omega(x,\tau), (dv(x)\cdot \tau)^\perp \rangle$, which is Lipschitz continuous with constant bounded by $2\Vert \omega \Vert_{C_0^2} \Vert v \Vert_{C_0^2}$.
By applying the same strategy, we decompose $A_3(s)$ into three terms using the triangle inequality:
    \begin{align*}
        A_3(s) &= \left\vert H_3(q(s),\tau_q(s)) - \dfrac{\ell_f}{L_f} \langle \nabla_{\tau} \omega(c_f,\vec{t}_f), \nabla^\perp v_f \rangle \right\vert \\
        &\leq \vert H_3(q(s),\tau_q(s)) - H_3(c_f,\vec{t}_f) \vert + \dfrac{\ell_f}{L_f} \left\vert \left\langle \nabla_{\tau} \omega(c_f,\vec{t}_f), \left(dv(c_f) \cdot \vec{t}_f \right)^\perp - \nabla^\perp v_f \right\rangle \right\vert \\
        &\quad + \left\vert 1 - \dfrac{\ell_f}{L_f} \right\vert \vert H_3(c_f,\vec{t}_f) \vert
    \end{align*}
    
    We apply the Cauchy-Schwarz inequality to the inner product on the second term. The orthogonal projection is a linear operator with a norm bounded by $1$, meaning that: $\big\| (dv(c_f) \cdot \vec{t}_f)^\perp - \big(\dfrac{v(q_{f^2})-v(q_{f^1})}{\ell_f}\big)^\perp \big\| \leq \big\| dv(c_f) \cdot \vec{t}_f - \dfrac{v(q_{f^2})-v(q_{f^1})}{\ell_f} \big\|$. Combined with the bounds established previously yields the bound for $A_3(s)$:
    \begin{align*}
        A_3(s) &\leq \Vert \omega \Vert_{C^2_0} \Vert v \Vert_{C^2_0} \left( 2\Vert q(s) - c_f \Vert + 2\Vert \tau_q(s) - \vec{t}_f \Vert  + \dfrac{1}{4} \ell_f + \left\vert 1 - \dfrac{\ell_f}{L_f} \right\vert \right) \\
        &\leq \Vert \omega \Vert_{C^2_0} \Vert v \Vert_{C^2_0} \left( 2K_1 + 2K_2 + \dfrac{1}{4}  + (C_0)^2 \kappa_{\max} \right) \ell_f
    \end{align*}
    
    We combine the three terms $A_1(s) + A_2(s) + A_3(s)$ and substitute the bounds established previously, such that for all $s \in I_f$:
    \begin{equation*}
        A_1(s) + A_2(s) + A_3(s) \leq C'_q \Vert \omega \Vert_{C^2_0} \Vert v \Vert_{C_0^2} \ell_f.
    \end{equation*}
    where $C'_q$ is a constant that depends only on $q$ and its chosen discretization.

    Finally, we integrate this uniform pointwise bound over $I$ by summing over all the faces $f \in F_q$:
    \begin{align*}
        \vert (\delta \tilde{\mu}_{q}(v) - \delta \mu_q(v))(\omega) \vert &\leq \sum_{f \in F_q} \int_{I_f} \Big( A_1(s) + A_2(s) + A_3(s) \Big) \|q'(s)\| ds \\
        &\leq \sum_{f \in F_q} \int_{I_f} \left( C'_q \Vert \omega \Vert_{C^2_0} \Vert v \Vert_{C_0^2} \max_{f \in F_q}(\ell_f) \right) \|q'(s)\| ds \\
        &= C'_q \Vert \omega \Vert_{C^2_0} \Vert v \Vert_{C_0^2} \max_{f \in F_q}(\ell_f) \sum_{f \in F_q} \int_{I_f} \|q'(s)\| ds \\
        &= C'_q \Vert \omega \Vert_{C^2_0} \Vert v \Vert_{C_0^2} \max_{f \in F_q}(\ell_f) \operatorname{length}(q).
    \end{align*}

    Taking the supremum over all test functions such that $\Vert \omega \Vert_{C^2_0} \leq 1$, we definitively obtain the dual norm bound on the space of varifolds:
    \begin{equation*}
        \left\Vert \delta \tilde{\mu}_q(v) - \delta \mu_q(v) \right\Vert_{(C_0^2)'} \leq C_q \Vert v \Vert_{C_0^2} \max_{f \in F_q} \ell_f,
    \end{equation*}
    with $C_q=C_q'\operatorname{length}(q) $, which concludes the proof.
\end{proof}

\subsection{Proof of \cref{prop:continuous_invariance}}\label{app:continuous_invariance}

\begin{proposition}[Invariance of the curve under the action of a vector field] 
    Let $q \in C^1(I, \mathbb{R}^d)$ be an immersed curve with a finite number of self-intersections and let $v \in C_0^2(\mathbb{R}^d, \mathbb{R}^d)$ be a vector field. If $\delta \tilde{\mu}_q(v) = 0$, then $\operatorname{im}(q)$ is invariant by the action of $v$.
\end{proposition}

\begin{proof}
    Assume that $\delta \tilde{\mu}_q(v)(\omega) = 0$ for all test functions $\omega \in C_0^2(\mathbb{R}^d \times \mathbb{S}^{d-1}, \mathbb{R})$, that is
    \begin{align}\label{eq:first_var_zero}
        \int_I \Big( \omega(q, \tau_q) \langle dv(q) \cdot \tau_q, \tau_q \rangle &+ \langle \nabla_x \omega(q, \tau_q), v(q) \rangle + \langle \nabla_{\tau} \omega(q, \tau_q), (dv(q) \cdot \tau_q)^\perp \rangle \Big) \| q'(s) \| ds = 0
    \end{align}
    where we omitted the parameter $s$ in the integrand for the sake of readability. For any parameter $s \in I$ and any vector field $w : \R^d \to \R^d$, we define the normal component of $w$ with respect to the curve as $w(q(s))^{\perp} =w(q(s)) - \langle w(q(s)), \tau_q(s) \rangle \tau_q(s)$.
    
    We first prove that the vector field $v$ is purely tangential everywhere along the curve, that is $(v \circ q)^{\perp}= 0$. Suppose that there exists a parameter $s^* \in I$ such that $v(q(s^*))^{\perp} \neq 0$. If $q(s^*)$ is a self-intersection point or an endpoint, because the set of self-intersections is finite and the mapping $s \mapsto v(q(s))^{\perp}$ is continuous, we can choose a parameter $s_0$ arbitrarily close to $s^*$ such that the point $q_0 = q(s_0)$ is not a self-intersection or an endpoint, and the normal component remains strictly non-zero, i.e., $v(q(s_0))^{\perp} \neq 0$. 

    Because $q$ is an immersion, it is locally an embedding at $q_0$ \cite[Theorem 4.25]{smooth_manifold_lee}, yielding a small closed interval $J \subset I$ around $s_0$ such that $q(J)$ is a one-dimensional embedded submanifold with boundary. As a result, $q(J)$ admits a neighborhood $U$ on which the nearest-point projection $\pi: U \to q(J)$ is well-defined and smooth \cite[Proposition 6.25]{smooth_manifold_lee}.
Furthermore, by \cite[Lemma 5.34]{smooth_manifold_lee}, there exists a smooth vector field $\tilde{v}^\perp \colon U \to \mathbb{R}^d$ such that $\tilde{v}^\perp(q(s)) = v(q(s))^\perp$ for all $s \in J$. Let $\chi \colon \mathbb{R}^d \to \mathbb{R}$ be a smooth non-negative bump function with compact support $\operatorname{supp}(\chi) \Subset U$ satisfying $\chi(q_0) > 0$. We define the test function $\omega$ over $\mathbb{R}^d \times \mathbb{S}^{d-1}$ by setting $\omega(x,\tau) = 0$ outside $\operatorname{supp}(\chi)$, and for $x \in U$:
\begin{equation}\label{eq:omega}
    \omega(x, \tau) = \chi(x) \langle x - \pi(x), \tilde{v}^\perp(x) \rangle.
\end{equation}
    
    Evaluated on the curve, the test function vanishes $\omega(q(s), \tau) = 0$ since $q(s) - \pi(q(s)) = 0$ for all $s$. Moreover, since $\omega$ is independent of $\tau$, $\nabla_{\tau} \omega = 0$. By substituting the test function defined in \cref{eq:omega} into \cref{eq:first_var_zero}, the first and third terms vanish entirely, leaving only:
    \begin{equation*}
        \int_I \langle \nabla_x \omega(q(s), \tau_q(s)), v(q(s)) \rangle \| q'(s) \| ds = 0.
    \end{equation*}
    The spatial gradient at $x \in U$ is given by the product rule:
    \begin{equation*}
        \nabla_x \omega(x) = \nabla \chi(x) \langle x - \pi(x), \tilde{v}^\perp(x) \rangle + \chi(x) \Big[ d(\id-\pi)(x)^T \, \tilde{v}^\perp(x) + d(\tilde{v}^\perp)(x)^T \, (x-\pi(x)) \Big].
    \end{equation*}
    Evaluated on the curve again, the first and third terms vanish since $q(s)-\pi(q(s)) = 0$. Furthermore, the differential $d(x-\pi)(x) = I_d - d\pi(x)$ evaluated on $q(J)$ corresponds to the orthogonal projection onto the normal space of $q(J)$ at $x$. Since $\tilde{v}^\perp(q(s)) = v(q(s))^\perp$ belongs to this normal space, multiplying it by this projection leaves it unchanged. Thus, $\nabla_x \omega(q(s)) = \chi(q(s)) v(q(s))^\perp$. Since $\chi$ vanishes outside of $U$, the integral over $I$ reduces to an integral over $J$:
    \begin{equation*}
        \int_J \chi(q(s)) \langle v(q(s))^\perp, v(q(s)) \rangle \| q'(s) \| ds = \int_J \chi(q(s)) \| v(q(s))^\perp \|^2 \| q'(s) \| ds = 0.
    \end{equation*}
    Because the integrand is continuous, non-negative, and strictly positive at $s = s_0$, the integral must be strictly positive, which contradicts \cref{eq:first_var_zero}. Therefore, $v(q(s))^\perp = 0$ for all $s \in I$, meaning $v(q(s)) = \alpha(s)\tau_q(s)$ for some scalar function $\alpha$. It remains to show that if $I=[0,1]$, then $\alpha(0)=\alpha(1)=0$.
    Thanks to this expression, the differential of $v$ along the curve becomes:
    \begin{equation*}
        dv(q(s)) \cdot \tau_q(s) = \dfrac{1}{\| q'(s) \|} \dfrac{d}{ds}\left(v(q(s))\right) = \dfrac{\alpha'(s)}{\| q'(s) \|} \tau_q(s) + \alpha(s) \dfrac{\tau_q'(s)}{\| q'(s) \|}.
    \end{equation*}
    We substitute this back into \cref{eq:first_var_zero} for an arbitrary test function $\omega$. In particular, noting that $\tau_q$ is a unit vector, we have $\langle \tau_q', \tau_q \rangle = 0$. Gathering the remaining terms yields:
    \begin{equation*}
        \int_I \left( \alpha'(s) \omega(q(s),\tau_q(s)) + \alpha(s) \langle \nabla_x \omega(q(s),\tau_q(s)), \tau_q(s) \rangle \| q'(s) \| + \alpha(s) \langle \nabla_{\tau} \omega(q(s),\tau_q(s)), \tau_q'(s) \rangle \right) ds = 0.
    \end{equation*}
    By the chain rule, the total derivative of $\omega$ along the curve is $\dfrac{d}{ds} [\omega(q(s),\tau_q(s))] = \langle \nabla_x \omega, q'(s) \rangle + \langle \nabla_{\tau} \omega, \tau_q'(s) \rangle$. Recalling that $q'(s) = \| q'(s) \| \tau_q(s)$, the integral simplifies to:
    \begin{equation*}
        \int_I \left( \alpha'(s) \omega + \alpha(s) \dfrac{d}{ds}\omega \right) ds = \int_I \dfrac{d}{ds} \big( \alpha(s) \omega(q(s),\tau_q(s)) \big) ds = 0.
    \end{equation*}
    If $I = \mathbb{S}^1$ (closed curve), the integral is trivially zero. If $I = [0,1]$ (open curve), the integral simplifies to:
    \begin{equation*}
        \alpha(1) \omega(q(1), \tau_q(1)) - \alpha(0) \omega(q(0), \tau_q(0)) = 0.
    \end{equation*}
    Since the test function $\omega$ can be chosen arbitrarily and independently at the endpoints $q(1)$ and $q(0)$, we have $\alpha(1) = 0$ and $\alpha(0) = 0$.

    In both scenarios, the action of the vector field $v$ on the curve is purely tangential ($v = \alpha \tau_q$), and in the case of an open curve, it vanishes at the boundaries. Therefore, the image of the curve remains unchanged under the action of $v$.
\end{proof}

\subsection{Proof of \cref{prop:continuity_first_var}} \label{app:proof_continuity_first_var}

\begin{proposition}[Continuity of $\delta \mu_q$] 
Let $V \hookrightarrow C_0^2(\R^d,\R^d)$ be a RKHS of vector fields and $q$ be a curve. The first variation of varifolds
$$\delta \mu_q : \app{V}{ C_0^2(\R^d \times \mathbb{S}^{d-1},\R)'}{v}{\delta \mu_q(v)}$$ is linear and continuous.
\end{proposition}

For the sake of clarity, we present the proof only for the approximation of a varifold by a weighted sum of Dirac masses. However, the result holds for any varifold associated with submanifolds embedded in $\R^d$, and the proof follows similar arguments.
\begin{proof} 
Linearity is immediate from the definition. Let show that the operator $\delta \mu_q$ is bounded. Let $\omega \in C_0^2(\R^d \times \mathbb{S}^{d-1},\R)$ be a test function such that $\Vert \omega \Vert_{C_0^2} \leq 1$.

\begin{align*}
    \vert \delta \mu_q(v) (\omega) \vert &\leq \sum_{f \in F_q} \ell_f  \vert(\delta \ell_f)_v \vert \vert \omega(c_f,\vec{t}_f) \vert  +  \ell_f \vert\langle \nabla_x \omega(c_f,\vec{t}_f), v_f \rangle \vert  + \ell_f \vert \langle \nabla_{\tau} \omega(c_f,\vec{t}_f), \nabla^\perp v_f \rangle  \vert \\
    &\leq \Vert \omega \Vert_{C_0^2}  \sum_{f \in F_q} \left( \ell_f  \vert(\delta \ell_f)_v \vert   +  \ell_f  \Vert v \Vert_{C_0^2}  + \ell_f \Vert \nabla^\perp v_f \Vert \right)
\end{align*}
where $\ell_f = \Vert q_{f^2}-q_{f^1} \Vert$, $(\delta \ell_f)_v = \langle \dfrac{v(q_{f^2})-v(q_{f^1})}{\Vert q_{f^2}-q_{f^1} \Vert}, \vec{t}_f \rangle$ and  $ \nabla^\perp v_f=\dfrac{v(q_{f^2})-v(q_{f^1})}{\ell_f} - \Big \langle \dfrac{v(q_{f^2})-v(q_{f^1})}{\ell_f} , \vec{t}_f\Big \rangle\vec{t}_f$.

For a face $f=[q_{f^1},q_{f^2}]$, since the discrete terms are linear combinations of evaluations of $v$ at the vertices, it follows that $\ell_f \vert (\delta \ell_f)_v \vert \leq 2 \Vert v \Vert_{C_0^2}$ and $\ell_f \Vert \nabla^\perp v_f \Vert \leq 4 \Vert v \Vert_{C_0^2}$. Consequently,
\begin{align*}
    \vert \delta \mu_q(v) (\omega) \vert &\leq \Vert \omega \Vert_{C_0^2}  \sum_{f \in F_q} \left( 2 \Vert v \Vert_{C_0^2} +  \ell_f  \Vert v \Vert_{C_0^2}  + 4 \Vert v \Vert_{C_0^2} \right) \\
    &\leq \Vert \omega \Vert_{C_0^2}  \Vert v \Vert_{C_0^2} \sum_{f \in F_q} (6  +  \ell_f). 
\end{align*}
Moreover, since $V \hookrightarrow C_0^2(\R^d,\R^d)$, there exists $\gamma_V >0$ such that $\Vert v \Vert_{C_0^2} \leq \gamma_V \Vert v \Vert_V$, 
\begin{align*}
    \Vert \delta \mu_q(v) \Vert_{C_0^2(\R^d\times \mathbb{S}^{d-1},\R)'} \leq \gamma_V \left( \sum_{f \in F_q} (6  +  \ell_f) \right) \Vert v \Vert_{V}     
\end{align*}
which concludes the proof.

\end{proof}

\subsection{Proof of \cref{prop:closed_form_min_J}} \label{app:proof_closed_form}

Recall that the coupling score is defined through the minimization of the following functional: 
\begin{equation*}
    J_v : \app{W}{\R}{w}{\Vert \delta \mu_q(w) - \delta \mu_q(v) \Vert^2_{H'} + \lambda \Vert w \Vert^2_W}
\end{equation*}
We also recall that $(e_1, \dots , e_d)$ denotes the canonical basis of $\R^d$, and $\delta_x^e : w \in W \mapsto \langle w(x), e \rangle_{\R^d}$ is the continuous linear evaluation functional. 
For a discretized curve $q = (q_i)_{1 \leq i \leq N}$ in $\R^d$, we define the set $\Theta_q \subset [| 1, N |] \times [| 1, d |]$ such that the set $\{ w_{i,j} \mid (i,j) \in \Theta_q \}$, where $w_{i,j} := K_W \delta_{q_i}^{e_j}$, forms a maximal linearly independent subset of $\{w_{i,j} \mid 1 \leq i \leq N \,, 1 \leq j \leq d \} \subset W$.
\begin{proposition}
Let $q = (q_i)_{1 \leq i \leq N}$ be a discretized curve in $\R^d$. Using the basis elements $w_{i,j}$ defined above, we construct the column vector 
$$Y = \left( \langle \delta \mu_q (v),  \delta \mu_q (w_{i,j}) \rangle_{H'} \right)_{(i,j) \in \Theta_q} \in \R^{\# \Theta_q}$$ 
and the matrix 
$$S = \left( \langle  (K_W (K_{H} \delta \mu_q)^\ast \delta \mu_q  + \lambda \id) w_{i,j}, w_{i',j'} \rangle_{W} \right)_{(i,j) , (i', j') \in \Theta_q} \in \R^{\# \Theta_q \times \# \Theta_q} \,.$$
Then the minimizer of $J_v$ is given by 
$$w^* = \sum_{(i,j) \in \Theta_q} \beta_{i,j} w_{i,j}$$ 
where $\beta = S^{-1}Y$.
\end{proposition}
\begin{proof}[Proof of \cref{prop:closed_form_min_J}]
    Since $J_v$ only depends on the evaluations of $w$ at the discrete vertices $q_i$, the Representer Theorem ensures that its minimizer $w^*$ belongs to the finite-dimensional space $\operatorname{span}( \{w_{i,j} \mid (i,j) \in \Theta_q \})$. Thus, there exists a column coefficient vector $\beta \in \R^{\#\Theta_q}$ such that $w^* = \sum_{(i,j) \in \Theta_q} \beta_{i,j} w_{i,j}$.

    Substituting this expression into $J_v$, the problem reduces to minimizing a standard finite-dimensional quadratic function $\tilde{J}(\beta)$ with respect to the vector $\beta$:
    \begin{equation*}
        \tilde{J}(\beta) = \beta^\top S \beta - 2 \beta^\top Y + \Vert \delta \mu_q(v) \Vert^2_{H'}
    \end{equation*}
    where $Y \in \R^{\#\Theta_q}$ is the column vector with components $Y_{i,j} = \langle \delta \mu_q(v), \delta \mu_q(w_{i,j}) \rangle_{H'}$, and $S \in \R^{\#\Theta_q \times \#\Theta_q}$ is the symmetric matrix with entries defined by:
    \begin{eqnarray*}
        S_{(i,j),(i',j')} &=& \langle \delta \mu_q(w_{i,j}), \delta \mu_q(w_{i',j'}) \rangle_{H'} + \lambda \langle w_{i,j}, w_{i',j'} \rangle_{W} \\
        &=& \langle K_W (K_H \delta \mu_q)^* \delta \mu_q(w_{i,j}), w_{i',j'} \rangle_W + \lambda \langle w_{i,j}, w_{i',j'} \rangle_{W} \\
        &=& \langle (K_W (K_H \delta \mu_q)^* \delta \mu_q + \lambda \operatorname{id_w})(w_{i,j}), w_{i',j'} \rangle_W\\
    \end{eqnarray*}
    Since $\tilde{J}$ is a strictly convex quadratic function (provided $S$ is positive definite), its unique minimum is reached when its gradient is zero, yielding the linear system $S \beta = Y$.
    Let us now prove that the matrix $S$ is indeed strictly positive definite, and thus invertible. Let $X \in \R^{\#\Theta_q}$ be a column vector such that $X^\top S X = 0$, and define the associated vector field $w_X = \sum_{(i,j) \in \Theta_q} X_{i,j} w_{i,j}$. By definition of the matrix $S$, we have:
    \begin{equation*}
        X^\top S X = \Vert \delta \mu_q(w_X) \Vert_{H'}^2 + \lambda \Vert w_X \Vert_W^2 = 0
    \end{equation*}
    Since $\lambda > 0$ and both terms are non-negative, it follows that $\Vert w_X \Vert_W^2 = 0$, which implies $w_X = 0$. Because the family $\{w_{i,j} \mid (i,j) \in \Theta_q\}$ is linearly independent by the definition of $\Theta_q$, the linear combination $w_X$ is equal to zero if and only if all its coefficients are zero, i.e., $X = 0$. 
    Thus, $X^\top S X = 0 \implies X = 0$. Consequently, $S$ is strictly positive definite and invertible, leading to the unique optimal coefficients $\beta = S^{-1} Y$.
\end{proof}

\subsection{Proof of \cref{prop:existence_minimizer}}\label{app:proof_existence_minimizer}

Let $I=[0,1]$ or $ \mathbb{S}^{1}$ and $Q =\{ q \in C^{1}(I,\mathbb{R}^d) \mid \Vert q'(s) \Vert \ne 0 \}$. Let $q_S \in Q$ be a template curve and $V,W$ be two RKHS continuously embedded in $C_0^{2}(\mathbb{R}^d, \mathbb{R}^d)$.
We consider the following energy minimization problem:
\begin{equation} \label{eq:J2_energy}
\inf_{(v, w) \in L^2([0,1],V \times W)} J_2(v,w) = \int_0^1 \left( \dfrac{1}{2} \| v_t \|_V^2 + \dfrac{1}{2} \| w_t \|_W^2 + C_{q_t}(v_t,W) \right) dt + \mathcal{D}(q_1)
\end{equation}
subject to the evolution equation $\dot{q}_t = (v_t + w_t) \circ q_t$ and the initial condition $q_0 = q_S$. 

For the sake of generality, we rely on the continuous formulation of the first variation rather than its discrete approximation via Dirac measures. We recall that the first variation associated with a curve $q$ and a vector field $v$ is given by, for any test function $\omega \in C^2_0(\R^d \times \mathbb{S}^{d-1},\R)$,
\begin{align*}
    \delta \tilde{\mu}_q(v)(\omega) &= \int_I \left( \langle \nabla_x \omega ( q(s),\tau_q(s) ), v(q(s)) \rangle + \langle \nabla_{\tau} \omega ( q(s),\tau_q(s) ), (dv(q(s))\cdot \tau_q(s))^\perp \rangle \right) \Vert q'(s) \Vert ds  \\
    &\quad + \int_I \omega ( q(s),\tau_q(s) ) \langle dv(q(s)) \cdot \tau_q(s), \tau_q(s) \rangle \Vert q'(s) \Vert  ds 
\end{align*}
where $\tau_q(s) = \dfrac{q'(s)}{\Vert q'(s) \Vert}$, $v(q(s))^\top = \langle v(q(s)),\tau_q(s) \rangle \tau_q(s)$ and $v(q(s))^\perp = v(q(s)) - v(q(s))^\top$.

\begin{proposition}
    Assume that the data attachment term $q \in Q \mapsto \mathcal{D}(q)$ is lower semi-continuous and non-negative. Then, the functional $J_2:(v,w)\mapsto \int_0^1 \left( \dfrac{1}{2} \| v_t \|_V^2 + \dfrac{1}{2} \| w_t \|_W^2 + C_{q_t}(v_t,W) \right) dt + \mathcal{D}(q_1)$ admits at least one global minimizer $(v^*, w^*) \in L^2([0,1], V \times W)$.
\end{proposition}

We first proves the following lemma, which ensures the continuity of the first variation with respect to its underlying shape.

\begin{lemma} \label{lemma:continuity_varifold_continuous}
The mapping $q \in Q \mapsto \delta \tilde{\mu}_q \in \mathcal{L}(V, H')$ is continuous with respect to the $C^1$ topology on $Q$ and the operator norm topology on $\mathcal{L}(V, H')$.
\end{lemma}

\begin{proof}
 Let $q \in Q$ be a curve and let $(q^n)_{n \in \mathbb{N}}$ be a sequence in $Q$ converging to $q$ in the $C^1$ topology. In particular, $(q^n)_n$ and $((q^n)')_n$ converge uniformly to $q$ and $q'$, respectively, on $I$. Therefore, there exists a constant $M>0$ such that for all $n \in \N$, $\Vert q^n \Vert_{C^1} \leq M$. 
    Because $q$ is an immersion on a compact set, there exists a constant $m > 0$ such that $\|q'(s)\| \ge m$ for all $s \in I$. For $n$ sufficiently large, uniform convergence ensures that $\|(q^n)'(s)\| \ge m/2$.

    For $v \in V$ and $\omega \in H$, we aim to bound $\vert \delta \tilde{\mu}_{q^n}(v)(\omega)-\delta \tilde{\mu}_q(v)(\omega) \vert$. First, we establish some uniform bounds that will be useful later. Since $\omega$ and $v$ have $C^2_0$ regularity, the functions and their first derivatives are bounded by their $C_0^2$ norms, and Lipschitz continuous with constant bounded by their $C_0^2$ norms.
    
    We bound the pointwise difference of the curves  by the difference in $C^1$ norm: $\Vert q^n(s) - q(s)\Vert \leq \Vert q^n - q \Vert_{C^1}$ and $\big| \Vert (q^n)'(s) \Vert - \Vert q'(s) \Vert \big| \leq \Vert (q^n)'(s) - q'(s) \Vert \leq \Vert q^n - q \Vert_{C^1}$. The mapping $y \mapsto \dfrac{y}{\|y\|}$ is smooth and Lipschitz continuous on the domain $\R^d \setminus B(0, m/2)$. Because both $(q^n)'$ and $q'$ take values in this domain, there exists a constant $L_m > 0$ such that the difference of the unit tangent vectors is bounded by $\Vert \tau_{q^n}(s) - \tau_q(s) \Vert \leq L_m \Vert (q^n)'(s) - q'(s) \Vert \leq L_m \Vert q^n - q \Vert_{C^1}$. 

    In the following, we omit the parameter $s$ for readability.
    We first focus on the term that depends on $\nabla_x \omega$. Using the triangle inequality, the Cauchy-Schwarz inequality, and the uniform bounds established above, it follows that:
    \begin{align*}
        \vert \langle \nabla_x \omega(q^n, \tau_{q^n})&,v(q^n)\rangle \Vert (q^n)' \Vert - \langle \nabla_x \omega(q, \tau_{q}),v(q)\rangle \Vert q' \Vert \vert \\
        &\leq \left\vert \langle \nabla_x \omega(q^n, \tau_{q^n}) - \nabla_x \omega(q, \tau_q), v(q^n) \rangle \right\vert \Vert (q^n)' \Vert + \left\vert \langle \nabla_x \omega(q, \tau_q), v(q^n) - v(q) \rangle \right\vert \Vert (q^n)' \Vert \\
        &\quad + \left\vert \langle \nabla_x \omega(q, \tau_q), v(q) \rangle \right\vert \big\vert \Vert (q^n)' \Vert - \Vert q' \Vert \big\vert \\
        &\leq  \Vert \omega \Vert_{C_0^2} \Big(\Vert q^n - q \Vert_{C^1} + \Vert \tau_{q^n}- \tau_q \Vert \Big) \Vert v \Vert_{C_0^2} M  + \Vert \omega \Vert_{C_0^2} \Vert v \Vert_{C_0^2} \Vert q^n - q \Vert_{C^1} M \\
        &\quad + \Vert \omega \Vert_{C_0^2} \Vert v \Vert_{C_0^2} \Vert q^n - q \Vert_{C^1} \\
        &\leq \Vert \omega \Vert_{C_0^2} \Vert v \Vert_{C_0^2}  ( 2M+L_mM +1)\Vert q^n - q \Vert_{C^1}
    \end{align*}

    Similarly, we bound the term involving $\nabla_\tau \omega$ using the same strategy. The projected vector is bounded by $\Vert (dv(q) \cdot \tau_q)^\perp \Vert \leq \Vert dv \Vert_\infty \Vert \tau_q \Vert \leq \Vert v \Vert_{C^2_0}$. The mapping $F:(x, \tau) \mapsto (dv(x)\cdot \tau)^\perp = (I_d - \tau\tau^T) dv(x) \tau$ is Lipschitz continuous on its domain with a constant bounded by $3\Vert v \Vert_{C_0^2}$. Thus, the difference is bounded as follows:
    \begin{align*}
        \vert \langle \nabla_{\tau} \omega ( q^n,\tau_{q^n} )&, (dv(q^n)\cdot \tau_{q^n})^\perp \rangle \Vert (q^n)' \Vert - \langle \nabla_{\tau} \omega ( q,\tau_q), (dv(q)\cdot \tau_q)^\perp \rangle \Vert q' \Vert \vert \\
        &\leq \left\vert \langle \nabla_{\tau} \omega(q^n, \tau_{q^n}) - \nabla_{\tau} \omega(q, \tau_q), (dv(q^n)\cdot \tau_{q^n})^\perp \rangle \right\vert \Vert (q^n)' \Vert \\
        &\quad + \left\vert \langle \nabla_{\tau} \omega(q, \tau_q), (dv(q^n)\cdot \tau_{q^n})^\perp - (dv(q)\cdot \tau_q)^\perp \rangle \right\vert \Vert (q^n)' \Vert \\
        &\quad + \left\vert \langle \nabla_{\tau} \omega(q, \tau_q), (dv(q)\cdot \tau_q)^\perp \rangle \right\vert \big\vert \Vert (q^n)' \Vert - \Vert q' \Vert \big\vert \\
        &\leq \Vert \omega \Vert_{C^2_0} \Vert v \Vert_{C^2_0} \Big( \Vert q^n - q \Vert_{C^1} + \Vert \tau_{q^n} - \tau_q \Vert \Big) M \\
        &\quad + 3\Vert \omega \Vert_{C^2_0}  \Vert v \Vert_{C^2_0} \Big( \Vert q^n - q \Vert_{C^1} + \Vert \tau_{q^n} - \tau_q \Vert \Big) M \\
        &\quad  + \Vert \omega \Vert_{C^2_0} \Vert v \Vert_{C^2_0} \Vert q^n - q \Vert_{C^1} \\
        &\leq \Vert \omega \Vert_{C^2_0} \Vert v \Vert_{C^2_0} \Vert q^n - q \Vert_{C^1} (4M +4M L_m + 1)  .
    \end{align*}
    
    Finally, we apply again the same decomposition strategy to the term $\omega(q,\tau_q) \langle dv(q) \cdot \tau_q, \tau_q \rangle \|q'\|$. The mapping $(x, \tau) \mapsto \langle dv(x)\cdot \tau, \tau \rangle$ is Lipschitz with a constant bounded by $\Vert v \Vert_{C_0^2}$. Evaluating the difference yields:
    \begin{align*}
        &\vert \omega ( q^n,\tau_{q^n} ) \langle dv(q^n)\cdot \tau_{q^n}, \tau_{q^n} \rangle \Vert (q^n)' \Vert - \omega ( q,\tau_q) \langle dv(q)\cdot \tau_q, \tau_q \rangle \Vert q' \Vert \vert \\
        &\leq \left\vert \Big( \omega(q^n, \tau_{q^n}) - \omega(q, \tau_q) \Big) \langle dv(q^n)\cdot \tau_{q^n}, \tau_{q^n} \rangle \right\vert \Vert (q^n)' \Vert  + \left\vert \omega(q, \tau_q) \Big( \langle dv(q^n)\cdot \tau_{q^n}, \tau_{q^n} \rangle - \langle dv(q)\cdot \tau_q, \tau_q \rangle \Big) \right\vert \Vert (q^n)' \Vert \\
        &\quad + \left\vert \omega(q, \tau_q) \langle dv(q)\cdot \tau_q, \tau_q \rangle \right\vert \big\vert \Vert (q^n)' \Vert - \Vert q' \Vert \big\vert \\
        &\leq \Vert \omega \Vert_{C^2_0} \Vert v \Vert_{C^2_0} \Big( \Vert q^n - q \Vert_{C^1} + \Vert \tau_{q^n} - \tau_q \Vert \Big) M  + \Vert \omega \Vert_{C^2_0} \Vert v \Vert_{C^2_0} \Big( \Vert q^n - q \Vert_{C^1} + \Vert \tau_{q^n} - \tau_q \Vert \Big) M \\
        &\quad + \Vert \omega \Vert_{C^2_0} \Vert v \Vert_{C^2_0} \Vert q^n - q \Vert_{C^1} \\
          &\leq \Vert \omega \Vert_{C^2_0} \Vert v \Vert_{C^2_0} \Vert q^n - q \Vert_{C^1} (2M +2M L_m + 1).
    \end{align*}

  Integrating over $I$ gives the final bound:
    \begin{equation*}
        \left| \delta \mu_{q^n}(v)(\omega) - \delta \mu_q(v)(\omega) \right| \le K'_q \|q^n - q\|_{C^1} \|v\|_{C_0^2} \|\omega\|_{C_0^2}
    \end{equation*}
    where $K'_q > 0$ is a global constant depending on the curve $q$.
    Since $V \hookrightarrow C_0^2(\R^d,\R^d)$ and $H\hookrightarrow C_0^2(\R^d \times \mathbb{S}^{d-1},\R)$, there exists constants such that $\Vert v \Vert_{C_0^2}\leq \gamma_V \Vert v \Vert_V$ and $\Vert \omega \Vert_{C_0^2}\leq \gamma_H\Vert \omega \Vert_H$ . So, the previous inequality can be written:
    \begin{equation*}
        \left| \delta \tilde{\mu}_{q^n}(v)(\omega) - \delta \tilde{\mu}_q(v)(\omega) \right| \le K_q \|q^n - q\|_{C^1} \|v\|_{V} \|\omega\|_{H}
    \end{equation*}

    Finally, by taking the supremum over the unit balls $\|v\|_V \le 1$ and $\|\omega\|_H \le 1$, we obtain:
    \begin{equation*}
        \left\Vert \delta \tilde{\mu}_{q^n} - \delta \tilde{\mu}_q \right\Vert_{\mathcal{L}(V, H')} \le K_q \|q^n - q\|_{C^1(I)}.
    \end{equation*}
    Then, $\delta \tilde{\mu}_{q^n} \to \delta \tilde{\mu}_q$ as $q^n \to q$, which concludes the continuity of the mapping $q \mapsto \delta \tilde{\mu}_q$
\end{proof}

Building on the continuity of the first variation, we can show that the coupling score is also continuous with respect to the curve.

\begin{lemma} \label{lemma:continuity_correlation}
For a fixed $v \in V$, the mapping $q \in Q \mapsto C_q(v, W)$ is continuous.
\end{lemma}

\begin{proof}
    Let $\delta \mu_q^W \in \mathcal{L}(W,H')$ and $\delta \mu_q^V \in \mathcal{L}(V,H')$ denote the restrictions of the first variation $\delta \mu_q$ to the spaces $W$ and $V$, respectively. 
    We recall that the coupling score is defined by $C_q(v,W) = \Vert w^* \Vert^2_W$, where $w^* \in W$ is given by:
    \begin{equation*}
        w^* = (\Lambda_q \delta \mu_q^W + \lambda \operatorname{id}_W)^{-1} \Lambda_q \delta \mu_q^V (v)
    \end{equation*}
    with $\Lambda_q \eqdef K_W (K_H \delta \mu_q^W)^* \in \mathcal{L}(H',W)$. The explicit derivation of $w^*$ is provided in the proof of \cref{prop:min_w}.

    First, by the continuity of the mapping $q \mapsto \delta \mu_q^W$ (\cref{lemma:continuity_varifold_continuous}), of  the Riesz isomorphisms $K_H$ and $K_W$, and of the adjoint operator, it follows that $q \mapsto \Lambda_q$ is continuous. Consequently, the mapping $q \mapsto \Lambda_q \delta\mu_q^W + \lambda \id_W$ is also continuous. For any $w \in W$, we have:
    \begin{equation*}
        \langle \Lambda_q \delta \mu_q^W w, w \rangle_W = \langle (K_H \delta \mu_q^W)^* \delta \mu_q^W w, w \rangle_{W',W} = \langle \delta \mu_q^W w, K_H \delta \mu_q^W w \rangle_{H',H} = \| \delta \mu_q^W w \|_{H'}^2 \ge 0.
    \end{equation*}
    Consequently, for any $\lambda > 0$, $\langle (\Lambda_q \delta \mu_q^W  + \lambda \operatorname{id}_W) w, w \rangle_W \ge \lambda \|w\|_W^2$. The coercivity ensures that $(\Lambda_q \delta \mu_q^W + \lambda \operatorname{id}_W)$ is an invertible operator.
    The set of invertible operators $W \to W$ forms an open subset of $\mathcal{L}(W,W)$, and the inversion mapping $M \mapsto M^{-1}$ is continuous on this set. Therefore, the mapping $q \mapsto (\Lambda_q \delta\mu_q^W + \lambda \operatorname{id}_W)^{-1}$ is continuous.
    By composition of continuous operators, the mapping $ q \mapsto L_q \in \mathcal{L}(V,W)$ defined by:
    \begin{equation*}
        L_q = (\Lambda_q \delta \mu_q^W  + \lambda \operatorname{id}_W)^{-1} \Lambda_q \delta \mu_q^V
    \end{equation*}
    is continuous.
    
    Since for a fixed $v \in V$, the evaluation map $L_q \mapsto L_q(v)$ is continuous, as is the norm on $W$, then by composition, the coupling score $q \mapsto C_q(v,W)$ is continuous.
\end{proof}

Finally, we prove \cref{prop:existence_minimizer} that establishes the existence of a minimizer for the optimal control problem.

\begin{proof}
    Let $(v^n, w^n)_{n \in \mathbb{N}}$ be a minimizing sequence for $J_2$. Therefore, for $n$ sufficiently large, $J_2(v^n, w^n) \leq J_2(0,0) < \infty$. Since the data attachment term $\mathcal{D}$ and the coupling score $C_q$ are non-negative, the sequence $(v^n,w^n)_n$ is bounded in $L^2([0,1], V \times W)$.
    Therefore, there exists a subsequence $(v^{n_k},w^{n_k})_{n_k}$ converging weakly to some limit $(v^*, w^*) \in L^2([0,1], V \times W)$. Let $\varphi^{n_k}$ and $\varphi^*$ be the flows generated by $(v^{n_k} + w^{n_k})$ and $(v^* + w^*)$ respectively, and let $q^{n_k}_t = \varphi^{n_k}_t \circ q_S$ and $q^*_t = \varphi^*_t \circ q_S$ be the corresponding trajectories. In particular, since $q_S(I)$ is compact, for any $t \in [0,1]$ and any $n_k \in \mathbb{N}$, $q_t^{n_k}(I)$ and $q^*_t(I)$ are also compact.

    Since the RKHS $V$ and $W$ are continuously embedded in $C_0^2(\R^d,\R^d)$, and the data
    attachment term $\mathcal{D}$  is lower semi-continuous, \cite[Theorem 1]{arguillere2014shapedeformationanalysisoptimal} guarantee that $q^{n_k}$ converges uniformly to $q^*$ on $[0,1]$ and we have the following lower semi-continuity inequality:

    \begin{equation}\label{eq:lsc_lddmm}
        \int_0^1 \dfrac{1}{2} \left( \| v_t^* \|_V^2 + \| w_t^* \|_W^2 \right)\, dt + \mathcal{D}(q^*_1) \leq \liminf_{k \to \infty}  \int_0^1 \dfrac{1}{2} \left( \| v_t^{n_k} \|_V^2 + \| w_t^{n_k} \|_W^2 \right)\, dt + \mathcal{D}(q^{n_k}_1).
    \end{equation} 

    Let us now show that $L_{q^{n_k}}(v^{n_k}) \rightharpoonup L_{q^*}(v^*)$ weakly in $L^2([0,1], W)$. Because $q^{n_k}$ converges uniformly to $q^*$, the set $A = \{q^{n_k}\}_{k \in \mathbb{N}} \cup \{ q^* \}$ is compact in $C([0,1],C^1(I,\mathbb{R}^d))$. Let us introduce the continuous mapping $f \colon [0,1] \times C([0,1],C^1(I,\mathbb{R}^d)) \to C^1(I,\mathbb{R}^d)$ defined by $(t,q) \mapsto q_t$. Therefore, $\{ q_t^{n_k}\}_{t,k} \cup \{q_t^* \}_t = f([0,1] \times A)$ is a compact set of $C^1(I,\R^d)$ as the image of a compact set under a continuous mapping. Consequently, since $q \mapsto L_q$ is continuous (\cref{lemma:continuity_correlation}), it is uniformly continuous on $f([0,1] \times A)$. As a result, the sequence of bounded operators $t \mapsto L_{q_t^{n_k}}$ converges uniformly to $t \mapsto L_{q_t^*}$ on $[0,1]$.
    
    For any test function $\psi \in L^2([0,1], W)$, we can decompose the inner product difference as:
    \begin{equation*}
        \int_0^1 \langle L_{q^{n_k}_t}(v^{n_k}_t) - L_{q^*_t}(v_t^*), \psi_t \rangle_W \, dt = \int_0^1 \langle (L_{q^{n_k}_t} - L_{q^*_t})(v^{n_k}_t), \psi_t \rangle_W \, dt + \int_0^1 \langle v^{n_k}_t - v^*_t, (L_{q^*_t})^* \psi_t \rangle_V \, dt.
    \end{equation*}
    By the Cauchy-Schwarz inequality inequality, the first term tends to $0$ due to the convergence $L_{q^{n_k}_t} \to L_{q^*_t}$ and the boundedness of the sequence $v^{n_k}$ in $L^2$. The second term tends to $0$ due to the weak convergence $v^{n_k} \rightharpoonup v^*$ in $L^2([0,1], V)$. Thus, $L_{q^{n_k}}(v^{n_k}) \rightharpoonup L_{q^*}(v^*)$ in $L^2([0,1], W)$. 
    
    By the weak lower semi-continuity of the norm, we directly obtain:
    \begin{equation} \label{eq:lsc_corr}
        \int_0^1 \|L_{q^*_t}(v^*_t)\|_W^2 \, dt \le \liminf_{k \to \infty} \int_0^1 \|L_{q^{n_k}_t}(v^{n_k}_t)\|_W^2 \, dt.
    \end{equation}
    Summing inequalities \eqref{eq:lsc_lddmm} and \eqref{eq:lsc_corr}, and recalling that $C_{q_t}(v_t,W) = \|L_{q_t}(v_t)\|_W^2$, we conclude that $J_2(v^*, w^*) \le \liminf_{k \to \infty} J_2(v^{n_k}, w^{n_k})$. Since $(v^{n_k}, w^{n_k})_{n_k}$ is a minimizing sequence, it follows that $J_2(v^*, w^*) = \inf J_2(v,w)$, which proves the existence of a global minimizer.
\end{proof}

\subsection{Proof of \cref{prop:static_to_dynamic}}\label{app:proof_static_to_dynamic}

\begin{proposition}[Relation between minimizers in static and dynamic setting] Assume that the data attachment $\mathcal{D}$ is invariant by reparameterization.
    Let $(v^*, w^*) \in L^2([0,1], V \times W)$ be a global minimizer of $J_2$, and let $q^*$ be the associated optimal trajectory starting from $q_0 = q_S$. For almost every $t \in [0,1]$, the pair $(v_t^*, w_t^*)$ is a solution to the static minimization problem:
    \begin{equation} \label{eq:static_split_prop}
        \min_{(v,w) \in V \times W} \dfrac{1}{2} \left( \|v\|_V^2 + \|w\|_W^2 + C_{q_t^*}(v, W) \right)
    \end{equation}
    subject to the constraint on the first variations:
    \begin{equation*}
        \delta \mu_{q_t^*} (v) + \delta \mu_{q_t^*} (w) = \delta \mu_{q_t^*} (v_t^*) + \delta \mu_{q_t^*} (w_t^*).
    \end{equation*}
\end{proposition}
\begin{proof}
Suppose there exists a subset $A \subset [0,1]$ of strictly positive Lebesgue measure such that for all $t \in A$, the pair $(v_t^*, w_t^*)$ does not minimize the static problem \eqref{eq:static_split_prop}. 
    
    Consequently, for each $t \in A$, there exists $(\tilde{v}_t, \tilde{w}_t) \in V \times W$ satisfying the constraint:
    \begin{equation} \label{eq:constraint_proof}
        \delta \mu_{q_t^*} (\tilde{v}_t + \tilde{w}_t) = \delta \mu_{q_t^*} (v_t^* + w_t^*)
    \end{equation}
    with a strictly lower cost:
    \begin{equation} \label{eq:cost_proof}
        \| \tilde{v}_t \|_V^2 + \| \tilde{w}_t \|_W^2 + C_{q_t^*}(\tilde{v}_t, W) < \| v_t^* \|_V^2 + \| w_t^* \|_W^2 + C_{q_t^*}(v_t^*, W).
    \end{equation}
    
    We construct a time-dependent vector field $(\hat{v}, \hat{w}) \in L^2([0,1], V \times W)$ by replacing the optimal fields strictly on the set $A$:
    \begin{equation*}
        (\hat{v}_t, \hat{w}_t) = 
        \begin{cases} 
            (\tilde{v}_t, \tilde{w}_t) & \text{if } t \in A \\
            (v_t^*, w_t^*) & \text{if } t \notin A
        \end{cases}
    \end{equation*}
    By definition, the new vector field $\hat{v}_t + \hat{w}_t$ coincides with $v_t^* + w_t^*$ outside of $A$, and satisfies the constraint \eqref{eq:constraint_proof} inside $A$. Thus, for almost every $t \in [0,1]$, we have:
\begin{equation} \label{eq:varifold_equality_proof}
    \delta \mu_{q_t^*}(\hat{v}_t + \hat{w}_t) = \delta \mu_{q_t^*}(v_t^* + w_t^*)
\end{equation}

Thanks to the linearity of the first variation, it follows that $\delta \mu_{q^*_t}\left( (\hat{v}_t + \hat{w}_t) - (v_t^* + w_t^*) \right)=0$, which implies from \cref{prop:continuous_invariance} that the vector field $(\hat{v}_t + \hat{w}_t) - (v_t^* + w_t^*)$ is tangent to the curve $q_t^*$. In other words, $(\hat{v}_t + \hat{w}_t) \circ q^*_t$ and  $(v_t^* + w_t^*) \circ q^*_t$ only differs by a tangential component along the curve.

Let $t\mapsto \hat{q}_t$ be the trajectory of curves generated by the flow of the new vector field $(\hat{v}_t + \hat{w}_t)$ starting from $\hat{q}_0 = q_S$. The image of the curve induced by the flow of $\hat{v}_t + \hat{w}_t$ and $v^*_t + w^*_t$ is the same since they only differ by a tangential component. Consequently, for all $t \in [0,1]$, $\hat{q}_t$ is a reparameterization of $q_t^*$ and so $\mathcal{D}(q_1^*)=\mathcal{D}(\hat{q}_1)$ using that $\mathcal{D}$ is invariant by reparameterization.
Then, evaluating the energy of $(\hat{v},\hat{w})$ yields:
\begin{align*}
    \int_0^1  \|\hat{v}_t\|_V^2 + \|\hat{w}_t\|_W^2 + C_{\hat{q}_t}(\hat{v}_t, W)  dt 
    &= \int_{[0,1] \setminus A}  \|v_t^*\|_V^2 + \|w_t^*\|_W^2 + C_{q_t^*}(v_t^*, W)  dt \\
    &+ \int_{A}  \|\tilde{v}_t\|_V^2 + \|\tilde{w}_t\|_W^2 + C_{q_t^*}(\tilde{v}_t, W)  dt 
\end{align*}
Because the inequality \eqref{eq:cost_proof} holds strictly on $A$, and $A$ has a strictly positive Lebesgue measure, integrating this inequality implies that the new total dynamic energy is strictly lower:
\begin{equation*}
    J_2(\hat{v}, \hat{w}) < J_2(v^*, w^*)
\end{equation*}
This strictly contradicts the assumption that $(v^*, w^*)$ is a global minimizer of $J_2$, which concludes the proof.
\end{proof}

\section{Alternative formulation of the first variation}

Recall that for $\omega \in C_0^1(\R^d \times \mathbb{S}^{d-1},\R)$, the first variation $\delta \tilde{\mu}_q(v)$ is given by:
\begin{align*}
   \delta \tilde{\mu}_q(v) &= \int_I \left( \langle \nabla_x \omega ( q(s),\tau_q(s) ), v(q(s)) \rangle + \langle \nabla_{\tau} \omega ( q(s),\tau_q(s) ), (dv(q(s))\cdot \tau_q(s))^\perp \rangle \right) \Vert q'(s) \Vert ds  \\
    &\quad + \int_I \omega ( q(s),\tau_q(s) ) \langle dv(q(s)) \cdot \tau_q(s), \tau_q(s) \rangle \Vert q'(s) \Vert  ds 
\end{align*}
where $\tau_q(s) = \dfrac{q'(s)}{\Vert q'(s) \Vert}$ denotes the unit tangent vector at point $q(s)$ and $v(q(s))^\perp = v(q(s)) -  \langle v(q(s)),\tau_q(s) \rangle \tau_q(s)$ denote the normal components of the vector field along the curve. 

An integration by parts facilitates the interpretation of the first variation, as it provides an expression depending directly on $v$ rather than its differential. This requires the assumption $\omega \in C_0^2(\R^d \times \mathbb{S}^{d-1},\R)$.
We proceed by integration by parts on the last two terms of $\delta \tilde{\mu}_q(v)$.  In the following, we omit the parameter $s$ for readability.

Differentiating the decomposition $v\circ q = (v\circ q)^\perp + \langle v\circ q, \tau_q \rangle \tau_q$ with respect to $s$ yields:
\begin{equation*}
    (v\circ q)' = ((v\circ q)^\perp)' + \langle (v\circ q)', \tau_q \rangle \tau_q + \langle v\circ q, \tau_q' \rangle \tau_q + \langle v\circ q, \tau_q \rangle \tau_q'.
\end{equation*}
By definition, the normal projection of this derivative is $((v\circ q)')^\perp = (v\circ q)' - \langle (v\circ q)', \tau_q \rangle \tau_q$. Combining with the last equality gives:
\begin{equation*}
    (dv(q)\cdot q')^\perp = ((v\circ q)')^\perp = ((v\circ q)^\perp)' + \langle v\circ q, \tau_q' \rangle \tau_q + \langle v\circ q, \tau_q \rangle \tau_q'.
\end{equation*}
Using this equality and the fact that the spherical gradient $\nabla_{\tau} \omega$ belongs to the tangent space of the sphere (i.e., $\langle \nabla_{\tau} \omega, \tau_q \rangle = 0$), the inner product simplifies to:
 \begin{equation*}
     \langle \nabla_{\tau} \omega, ((v\circ q)')^\perp \rangle = \langle \nabla_{\tau} \omega, ((v\circ q)^\perp)' \rangle + \langle v\circ q, \tau_q \rangle \langle \nabla_{\tau} \omega, \tau_q' \rangle.
 \end{equation*}
We apply a first integration by parts.
\begin{align*}
    \int_I \langle \nabla_{\tau} \omega ( q,\tau_q ), (dv(q) \cdot q')^\perp \rangle ds &= \int_I \big( \langle \nabla_{\tau} \omega, (v\circ q^\perp)' \rangle + \langle v\circ q, \tau_q \rangle \langle \nabla_{\tau} \omega, \tau_q' \rangle \big) ds \\
    &=\Big[ \langle \nabla_{\tau} \omega ( q,\tau_q ), (v\circ q)^\perp \rangle \Big]_I \\
    &\quad - \int_I \left\langle \nabla^2_{x,\tau} \omega ( q,\tau_q )\tau_q, (v\circ q)^\perp \right\rangle \Vert q'\Vert ds \\
    &\quad -  \int_I \left\langle  \nabla^2_{\tau,\tau} \omega ( q,\tau_q )H_q , (v\circ q)^\perp \right\rangle \Vert q'\Vert ds\\
    &\quad + \int_I \langle \nabla_{\tau} \omega ( q,\tau_q), H_q \rangle \langle v\circ q, \tau_q \rangle \Vert q' \Vert ds.
\end{align*}
where  $H_q$ denotes the curvature and $\tau_q' = \Vert q' \Vert H_q$.

For the last term of $\delta \tilde{\mu}_q(v)$, we first rewrite the integrand as $\langle dv(q)\cdot \tau_q, \tau_q \rangle \Vert q' \Vert = \langle (v\circ q)', \tau_q \rangle= \dfrac{d}{ds}\langle v  \circ q, \tau_q \rangle - \langle v \circ q, \tau_q' \rangle$. Integrating by parts then yields:
\begin{align*}
    \int_I \omega ( q,\tau_q ) \langle (v\circ q)', \tau_q \rangle ds &= \int_I \omega ( q,\tau_q )  \dfrac{d}{ds}\langle v\circ q, \tau_q \rangle ds - \int_I \omega ( q,\tau_q )\langle v\circ q, \tau_q' \rangle ds \\
    &= \Big[ \omega ( q,\tau_q ) \langle v\circ q,\tau_q \rangle \Big]_I \\
    &\quad - \int_I \langle \nabla_x \omega ( q,\tau_q ) , \tau_q \rangle \langle v\circ q,\tau_q\rangle\Vert q' \Vert ds \\
    &\quad - \int_I \langle \nabla_{\tau} \omega ( q,\tau_q ) ,H_q \rangle \langle v\circ q, \tau_q\rangle \Vert q' \Vert ds \\
    &\quad - \int_I \omega ( q,\tau_q ) \langle v\circ q , H_q \rangle \Vert q' \Vert ds.
\end{align*}

Substituting these expressions back into the original expression of $\delta \tilde{\mu}_{q}(v)$ leads to:
\begin{align*}
    \delta \tilde{\mu}_q(v) &= \int_I \langle \nabla_x \omega ( q,\tau_q ), (v\circ q)^\perp \rangle \Vert q' \Vert ds \\
    &\quad - \int_I \left\langle \nabla^2_{x,\tau} \omega ( q,\tau_q )\tau_q + \nabla^2_{\tau,\tau} \omega ( q,\tau_q )H_q , (v\circ q)^\perp \right\rangle \Vert q'\Vert ds \\
    &\quad - \int_I \omega ( q,\tau_q ) \langle H_q, (v\circ q)^\perp \rangle \Vert q' \Vert ds \\
    &\quad + \Big[ \omega ( q,\tau_q ) \langle v \circ q, \tau_q \rangle + \langle \nabla_{\tau} \omega ( q,\tau_q ), (v\circ q)^\perp \rangle \Big]_I.
\end{align*}

Note that except for the boundary term, the first variation depends only on the orthogonal component of $v \circ q$. This is a consequence of the invariance by reparameterization of varifolds.

\section{Reminder on differential calculus in submanifolds of $\R^d$} \label{app:diff_submanifolds}

This section provides a brief review of differential calculus on submanifolds of $\R^d$ used throughout this paper. A detailed presentation of these results and their proofs can be found in~\cite{boumal2023intromanifolds}.

Let $M, N$ be smooth submanifolds embedded in $\R^d$.

\begin{definition}[Differentiability between submanifolds]
    The map $f : M \to N$ is $C^k$ at $p \in M$ if there exist $C^k$ charts $(U,\varphi)$ containing $p$ and $(V,\psi)$ containing $f(p)$ such that $f(U) \subset V$, and the composition $\psi \circ f \circ \varphi^{-1}$ is $C^k$ from $\varphi(U)$ to $\psi(V)$.
\end{definition}

To define the derivative of such mappings, we first need to formalize the notion of tangent vector. For $p \in M$, a tangent vector at $p$ is an equivalence class $[\gamma]$ of smooth curves $\gamma \colon (-\varepsilon,\varepsilon) \to M$ with $\gamma(0) = p$, where the equivalence relation is defined as follows: 
\begin{equation*}
    \gamma_1 \sim \gamma_2 \Longleftrightarrow \exists \text{ a chart } (U,\varphi) \text{ such that } (\varphi \circ \gamma_1)'(0) = (\varphi \circ \gamma_2)'(0).
\end{equation*}

We denote by $T_xM$ the tangent space to $M$ at $x \in M$ and $TM$ the tangent bundle of $M$.

\begin{definition}[Expression of the differential]
    Let $f \colon M \to N$ be a differentiable map and $p \in M$. The derivative of $f$ at $p$ is the linear map defined by:
    \begin{equation*}
        df(p) \colon [\gamma] \in T_pM \mapsto [f \circ \gamma] \in T_{f(p)}N.
    \end{equation*}
\end{definition}

Once the differential established, we can define the Riemannian gradient. For a submanifold embedded in $\R^d$, the Riemannian gradient defined over the submanifold is related to the classical ambient gradient.
We now assume that $M$ is an embedded submanifold of $\mathbb{R}^d$, endowed with the Riemannian metric induced by the ambient Euclidean inner product. 

\begin{definition}
Let $f \colon M \to \mathbb{R}$ be a differentiable map. The Riemannian gradient of $f$ is the unique vector field $\nabla_M f : M \to TM$ such that
\begin{equation}
    df(x) \cdot v = \langle \nabla_M f(x) ,v\rangle_x
\end{equation}
for all $(x,v) \in TM$. 
\end{definition}

We denote $\nabla f$ the standard gradient in $\R^d$ and $\operatorname{proj}_x \colon \mathbb{R}^d \to \mathbb{R}^d$ the orthogonal projector onto $T_x M$.

\begin{proposition}
    Let $f \colon M \to \mathbb{R}$ be a differentiable map. For any $x \in M$, the Riemannian gradient of $f$ is given by:
    \begin{equation*}
        \nabla_M f(x) = \operatorname{proj}_{x}(\nabla \tilde{f}(x))
    \end{equation*}
    where $\tilde{f}$ is any smooth extension of $f$ to a neighborhood of $M$ in $\mathbb{R}^d$, and $\nabla \tilde{f}(x)$ denotes its standard Euclidean gradient.
\end{proposition}

Note that for a differentiable mapping $f \colon M \to \R$ there always exists a neighborhood $U$ of $ M$ and a differentiable function $\tilde{f} \colon U \to \R$ that is an extension of  $f$ \cite[Lemma 5.34]{smooth_manifold_lee}. 

\begin{example}[Application to the unit sphere $M = \mathbb{S}^{d-1}$]\label{app:exam_grad}
    Let $\mathbb{S}^{d-1} = \{ x \in \mathbb{R}^d \mid \|x\| = 1 \}$ denote the unit sphere. For any point $x \in \mathbb{S}^{d-1}$, the tangent space at $x$ is $T_x\mathbb{S}^{d-1} = \{ v \in \mathbb{R}^d \mid \langle v, x \rangle = 0 \}$. The orthogonal projection of any vector $w \in \mathbb{R}^d$ onto $T_x\mathbb{S}^{d-1}$ is given by:
    \begin{equation*}
        \operatorname{proj}_x(w) = w - \langle w, x \rangle x = (I_d - xx^T)w.
    \end{equation*}

    Let $f \colon \mathbb{S}^{d-1} \to \mathbb{R}$ be a differentiable function. To compute its gradient on $\mathbb{S}^{d-1}$, we consider an arbitrary differentiable extension $\tilde{f}$ defined on a neighborhood of the sphere. Following the previous proposition, the Riemannian gradient at $x \in \mathbb{S}^{d-1}$ is the projection of the Euclidean gradient of $\tilde{f}$ in $\mathbb{R}^d$:
    \begin{equation*}
        \nabla_{\mathbb{S}^{d-1}} f(x) = \operatorname{proj}_x(\nabla \tilde{f}(x)) = (I_d - xx^T) \nabla \tilde{f}(x).
    \end{equation*}

    A natural choice for such an extension on $\mathbb{R}^d \setminus \{0\}$ is $\tilde{f} \colon y \in \mathbb{R}^d \setminus \{0\} \mapsto f\left(\dfrac{y}{\|y\|}\right)$. 
    For this extension, we have $\langle \nabla \tilde{f}(x), x \rangle = 0$, which yields the equality $\nabla_{\mathbb{S}^{d-1}} f(x) = \nabla \tilde{f}(x)$. 
\end{example}
In this article, we often refer to this gradient as the spherical or tangential gradient, and we denote it $\nabla_{\tau}$.
Having established the first-order derivatives, we now focus on second-order calculus on manifolds. On a general smooth manifold $M$, there is a priori no canonical way to identify tangent spaces at different points. This lack of canonical isomorphism makes it impossible to differentiate a vector field by simply comparing its values across the manifold. To overcome this, one introduces a connection, which provides a rigorous rule for taking directional derivatives of vector fields. Let $\mathfrak{X}(M)$ denote the set of all differentiable vector fields on $M$. Recall that for $f \in C^\infty(M)$ and $X \in \mathfrak{X}(M)$, $X \cdot f \colon M \to \mathbb{R}$ denotes the directional derivative of $f$ along $X$. An affine connection on $M$ is an operator $\nabla \colon \mathfrak{X}(M) \times \mathfrak{X}(M) \to \mathfrak{X}(M)$, denoted $(X,Y) \mapsto \nabla_X Y$, that is $C^\infty(M)$-linear with respect to the first coordinate, $\mathbb{R}$-linear with respect to the second, and satisfies the Leibniz rule, meaning that for any $f \in C^\infty(M)$, $\nabla_X (fY) = f \nabla_X Y + (X \cdot f)Y$. In particular, if $M$ is a Riemannian manifold, there exists an unique connection that is compatible with the metric and symmetric, named the Levi-Civita connection.

\begin{definition}[Riemannian Hessian]
    Let $M$ be a Riemannian manifold and let $\nabla$ be its Levi-Cita connection. The Riemannian Hessian of the $C^2$ mapping $f \colon M \to \R$ at $x \in M$ is the linear map $\operatorname{Hess}_M f(x) \colon T_xM \to T_xM$ defined as follows, for $u \in T_xM$:
    \begin{equation*}
        \operatorname{Hess}_M f(x)[u] = \nabla_u \operatorname{grad}_M f(x).
    \end{equation*}
\end{definition}
In the previous definition, we denoted the gradient $\operatorname{grad}_M$ instead of $\nabla_M$ to prevent confusion with the connection.

For embedded submanifolds, we can benefit from the ambient metric to compute the Hessian, in the same fashion as we did for the gradients.

\begin{proposition}[Riemannian Hessian on submanifolds]\label{app:hessian}
    Let $M$ be a Riemannian submanifold embedded in $\mathbb{R}^d$. Consider a $C^2$ function $f \colon M \to \mathbb{R}$. Let $\bar{G}$ be a smooth extension of $\nabla_M f$, that is, $\bar{G}$ is any smooth vector field defined on a neighborhood of $M$ in the embedding space such that $\bar{G}(x) = \nabla_M f(x)$ for all $x \in M$. Then, for $u \in T_xM$,
    \begin{equation*}
        \operatorname{Hess}_M f(x)[u] = \operatorname{proj}_x(d\bar{G}(x)[u]).
    \end{equation*}
\end{proposition}

\begin{example}[Application to the unit sphere $M = \mathbb{S}^{d-1}$]
    Consider a smooth function $f \colon \mathbb{S}^{d-1} \to \mathbb{R}$ and an extension of $f$ defined by $\tilde{f} \colon y \in \mathbb{R}^d \setminus \{0\} \mapsto f\left(\dfrac{y}{\|y\|}\right)$. 
    As established in \cref{app:exam_grad}, the Riemannian gradient coincides with the Euclidean gradient:
    \begin{equation*}
        \operatorname{grad}_{\mathbb{S}^{d-1}} f(x) = \nabla \tilde{f}(x).
    \end{equation*}
    This provides a natural extension $\tilde{G}$ of the Riemannian gradient $\operatorname{grad}_{\mathbb{S}^{d-1}} f$ to the ambient neighborhood, namely $\tilde{G}(x) = \nabla \tilde{f}(x)$. 
    The differential of this extension in the direction $u \in T_x\mathbb{S}^{d-1}$ is given by $d\tilde{G}(x)[u] = \nabla^2 \tilde{f}(x)u$, where $\nabla^2 \tilde{f}(x)$ is the Euclidean Hessian matrix. 
    Applying \cref{app:hessian}, the Riemannian Hessian is the orthogonal projection of this differential onto $T_x\mathbb{S}^{d-1}$:
    \begin{align*}
        \operatorname{Hess}_{\mathbb{S}^{d-1}} f(x)[u] &= \operatorname{proj}_x \big( d\tilde{G}(x)[u] \big) \\
        &= (I_d - xx^T)\nabla^2 \tilde{f}(x)u.
    \end{align*}
\end{example}

Finally, we recall some relations between differentiability and Lipschitz continuity.

\begin{proposition}[Lipschitz continuity of $f$]
    If $f \colon M \to \mathbb{R}$ has a continuous gradient, then $f$ is $L$-Lipschitz continuous if and only if for all $x\in M$, $\Vert \operatorname{grad}_M f(x) \Vert \leq L$.
\end{proposition}

\begin{proposition}[Lipschitz continuity of $\operatorname{grad}_M f$]
    If $f \colon M \to \mathbb{R}$ is twice continuously differentiable, then $\operatorname{grad}_M f$ is $L$-Lipschitz continuous if and only if the operator norm of $\operatorname{Hess}_M f(x)$ is bounded by $L$ for all $x$, that is, for all $x\in M$, $\Vert \operatorname{Hess}_M f(x) \Vert \leq L$.
\end{proposition}
\section{Reminder on geodesic shooting}\label{sec:geodesic_shooting}

This appendix provides further details about the large deformation framework \cite{LDDMM} recalled in \cref{sec:cons_reg_prob}. The variational problem 

\begin{align}\label{eq:app_opt_cont}
   \min_{v \in L^2([0,1],V)} J(v) &= \int_0^1 \dfrac{1}{2} \Vert v_t \Vert_V^2 dt + \mathcal{D}(q_1) \\
    \text{s.t.} \quad & 
    \begin{cases}
        \dot{q}_t &= v_t \cdot q_t \\
        q_0 &= q^{(0)}
    \end{cases} \notag
\end{align}
can be interpreted as an optimal control problem where $v$ plays the role of the control. The Pontryagin Maximum Principle (PMP) provides a characterization of the solution to this variational problem. Introducing the Hamiltonian $H$, where $p \in T_q^*Q$ is a costate (or momentum) variable:
\begin{equation*}
    H(q,p,v) = \langle p, v \cdot q \rangle - \dfrac{1}{2} \Vert v \Vert_V^2
\end{equation*}
If $v$ is an optimal solution of problem \eqref{eq:app_opt_cont}, then there exists a time-dependent costate $p_t$ satisfying the Hamiltonian equations:
\begin{equation*} \label{eq:hamilt_eq}
    \begin{cases}
        \dot{q}_t &= \partial_p H(q_t,p_t,v_t) \\
        \dot{p}_t &= - \partial_q H(q_t,p_t,v_t) \\
        0 &= \partial_v H(q_t,p_t,v_t)
    \end{cases}
    \quad \Longleftrightarrow \quad
    \begin{cases}
        \dot{q}_t &= v_t \cdot q_t \\
        \dot{p}_t &= - (\partial_q \xi_{q_t}(v_t))^* p_t \\
        v_t &= K_V \xi_{q_t}^* p_t
    \end{cases}
\end{equation*}
Since the Hamiltonian remains constant with respect to time along optimal trajectories (geodesics), it follows that:
\begin{equation*}
    \int_0^1 \dfrac{1}{2} \Vert v_t \Vert^2_V dt = \dfrac{1}{2} \Vert v_0 \Vert^2_V = \dfrac{1}{2} \Vert K_V \xi_{q_0}^* p_0 \Vert^2_V
\end{equation*}
The costate variable thus enables us to perform geodesic shooting, which simplifies the previous energy minimization problem to:
\begin{align}
\label{eq:geod_shoot}
   \min_{p_0 \in T_{q_0}^*Q} E(p_0) &= \dfrac{1}{2} \Vert K_V \xi_{q_0}^* p_0 \Vert_V^2 + \mathcal{D}(q_1) \\
   \text{s.t.} \quad & 
    \begin{cases}
        \dot{q}_t &= v_t \cdot q_t \\
        \dot{p}_t &= - (\partial_q \xi_{q_t}(v_t))^* p_t
    \end{cases} \notag 
\end{align}

\clearpage
\addcontentsline{toc}{section}{References}
\begingroup
\raggedright 
\printbibliography
\endgroup

\end{document}